\documentclass[a4paper,11pt]{amsart}
\usepackage{graphicx,amssymb,amsmath,amsfonts,amsthm, color}
\usepackage{nicefrac}
\usepackage{geometry}
\usepackage{anysize}
\marginsize{2.5cm}{2.5cm}{2.5cm}{2.5cm}
\usepackage{mathtools}
\usepackage[normalem]{ulem}
\usepackage{textcomp}
\usepackage{paralist}
\usepackage[linkcolor=black,citecolor=black]{hyperref}
\usepackage[english]{babel}
\usepackage{bbm}
\newcounter{michicounter}

\newcommand{\m}{\mathbbm{m}}

\newcommand{\EE}{\mathbb{E}}

\newcommand{\NN}{\mathbb{N}}

\newcommand{\PP}{\mathbb{P}}
\newcommand{\QQ}{\mathbb{Q}}
\newcommand{\RR}{\mathbb{R}}

\newcommand{\ZZ}{\mathbb{Z}}

\newcommand{\aA}{\mathcal{A}}
\newcommand{\bB}{\mathcal{B}}
\newcommand{\cC}{\mathcal{C}}
\newcommand{\dD}{\mathcal{D}}

\newcommand{\jJ}{\mathcal{J}}
\newcommand{\kK}{\mathcal{K}}
\newcommand{\lL}{\mathcal{L}}

\newcommand{\nN}{\mathcal{N}}
\newcommand{\oO}{\mathcal{O}}

\newcommand{\qQ}{\mathcal{Q}}

\newcommand{\sS}{\mathcal{S}}

\newcommand{\al}{\alpha}
\newcommand{\eps}{\varepsilon}
\newcommand{\e}{\varepsilon}
\newcommand{\la}{\lambda}

\newcommand{\si}{\sigma}

\newcommand{\om}{\omega}

\newcommand{\ra}{\rightarrow}

\newcommand{\ti}{\tilde}

\newcommand{\ind}{\mathbf{1}}

\newcommand{\lqq}{\leqslant}
\newcommand{\gqq}{\geqslant}

\newtheorem{thm}{Theorem}
\newtheorem{prop}{Proposition}
\newtheorem{cor}{Corollary}

\newtheorem{lem}{Lemma}

\newtheorem{rem}{Remark}
\newtheorem{exm}{Example}

\DeclareMathSymbol{\ophi}{\mathalpha}{letters}{"1E}

\renewcommand{\phi}{\varphi}

\newcommand{\bI}{\mathbf{1}} 

\newfont{\cyrfnt}{wncyr10}
\def\J3{\cyrfnt{\rm \u{\cyrfnt I}}}
\def\j3{\cyrfnt{\rm \u{\cyrfnt i}}}

\usepackage[]{color}
\definecolor{DarkGreen}{rgb}{0.1,0.7,0.3}   

\definecolor{DarkGreen}{rgb}{0.1,0.7,0.3}   

\allowdisplaybreaks[4]

\begin{document}

\title[Deviation frequencies of Brownian path property approximations]{
Deviation frequencies of Brownian path property approximations}
\author{Michael A. H\"ogele$\qquad$}
\address{Departamento de Matem\'aticas, Universidad de los Andes, Bogot\'a, Colombia, ma.hoegele@uniandes.edu.co,\\ https://orcid.org/0000-0001-5744-0494}
\author{$\qquad$ Alexander Steinicke}
\date{\today}
\address{Chair of Applied Mathematics, Technical University of Leoben, Austria. alexander.steinicke@leoben.ac.at, https://orcid.org/0000-0001-6330-0295} 
\hfill\\[-1cm]
\maketitle
    
\begin{abstract}
This case study proposes a.s.~convergence quantifications of many classical sample path property approximations of Brownian motion in terms of the tradeoff between a.s.~rates and the integrability of the modulus of convergence, as well as the deviation frequencies. This includes L\'evy's construction of Brownian motion, the Kolmogorov-Chentsov (and the Kolmogorov-Totoki) continuity theorem, L\'evy's modulus of continuity, the Paley-Wiener-Zygmund theorem, the a.s.~approximation of the quadratic variation as well as the laws of the iterated logarithm by Khinchin, Chung and Strassen, among others. 
\end{abstract}

\section{\textbf{Introduction}} 

Since its beginnings in the first half of the 20th century \cite{Ka47, K043, Lev37, Lev54} the elegance of many results in stochastic calculus  unarguably relies on the precise identification of almost sure topological properties of the trajectories of Brownian motion (the Wiener process).  
These results present Brownian path properties solely on the level 
of random functions without referring to any underlying probabilistic approximations \cite{KS98, MP10}. 
This purely path centered view - in retrospective - laid the grounds for rough path calculus much later \cite{Ly98, LC07, FV} and the breakthrough of solving stochastic (partial) differential equations by the associated regularity structures in \cite{Hai13, Hai14}. 

Many classical results, such as L\'evy's modulus of continuity \cite{Lev37}, the convergence of the partial sums for the quadratic variation \cite{Lev54} or the laws of the iterated logarithm \cite{Lev54, Ch48} can be shown with an application of the first Borel-Cantelli lemma, where the summability of error events yields the almost sure non-existence of exceptional sequences, and hence almost sure convergence. However, the summability often hides the rates of decrease of the error probabilities, since in many situations the mentioned summability of the events is often not even close to being sharp, but considerably better, such as for instance, of some exponential or Gamma type order, or even faster. On the other hand, in several occasions, where the probabilities are only barely summable, the events have additional properties such as being nested or independent, for instance, as a consequence of independent increments. This structural surplus which generalizes the notion of complete convergence \cite{HE74,Yu99} seemingly has not been used systematically to quantify almost sure convergence. In \cite{EH22}, the authors started this work quantifying higher moments of the overlap statistics in the first Borel-Cantelli lemma. These results open the door to distinguish and quantify different types of almost sure convergence according to the finiteness of moments of the sequence error events, such as in the strong law of large numbers, or the presence of a large deviations principle (see Theorem 7 and 8 in \cite{EH22}). For a short comprehensive introduction on the respective literature of the Borel-Cantelli lemma we refer to the introduction there. In \cite{EHS26} these results were refined from the mean deviation frequencies to the moments maximal deviation index, with the same type of upper bounds. 
The main application are the tradeoff between a.s.~error rates of convergence and the integrability of the modulus of continuity applied to martingale convergence theorems, explained in item c.~below.
    
The idea of this article is a case study about the deviation frequencies of Brownian path properties, by kind of reverse engineering the beautiful work of hiding the probabilistic structure and to distinguish almost sure properties by its robustness in terms of deviation frequencies along the underlying discretizations in the proofs. While the rates of the approximations in probability have been well-known in the field, they have not been translated to almost sure convergence statements of the sample paths, such as \textit{deviation frequencies} or incidence of error. With deviation frequencies we mean the full count of occurrences of certain 'error' events $(E_n(\e))_{n\in \NN}$ mostly of the form 
\begin{equation}\label{e:fehlerereignis}
E_n(\e) := \{|\mbox{approximation}(n)-\mbox{limit}|>\e\}\qquad \mbox{ for some parameter }\e>0. 
\end{equation}
The events $E_n(\e)$ may depend on several parameters, such as $\e$ here. 
This means we focus on the integrability (or higher order means) of the (random) deviation frequency count $\oO_\e := \sum_{n=1}^\infty \ind(E_n(\e))$.  
Note that a finite expectation of $\oO_\e$ coincides with the notion of complete convergence by \cite{Yu99} and is related to fast convergence in \cite{Kle08}. In \cite{EH22} this notion is generalized to arbitrary higher moments of $\oO_\e$, often given through expressions of the type $\EE[\sS(\oO_\e)]<\infty$, for the discrete antiderivative $\sS(N):=\sum_{n=1}^Na_n$ with $\sS(0) := 0$ of an increasing sequence $(a_n)_{n\gqq 1}$. The decay rate of the tails of the error or deviation frequencies, $\PP(\oO_\e\gqq k)$, that one infers by the existence of $\EE[\sS(\oO_\e)]$ and (typically) by Markov's inequality, distinguishes well different types of almost sure approximations. For instance, in the law of the iterated logarithm the upcrossing frequencies of a given $\e$-error decays rather slowly, that is to say, even for exponentially small time scales, the probability of more than $k$ occurrences in the overlap statistics decays rather slowly (of order $k^{1+\delta}$ for some $\delta$). On the other hand, the probabilities of trespassing frequencies $\gqq k$ of an arbitrary $\e$-error strip in the upper bound of L\'evy's modulus of continuity decay (as a function of growing $k$) exceptionally sharp, that is, with a Gumbel type rate (of order $\exp(-p\exp(\vartheta k))$ for some $p, \vartheta>0$). Of course, full comparability between different results is hard to obtain, since most approximations are tailormade precisely for the particular result. Still, the assessment of the asymptotic decay of the deviation frequency is rather useful. The same kind of estimates is obtained if the error count $\oO_\e$ is replaced by the last error index $\m_\e := \sum_{n=1}^\infty \ind(\bigcup_{m\gqq n} E_m(\e))$, which represents for each realization $\omega$ the (random) \emph{maximal index} $n$, at which an event $E_m(\e)$ is realized, that is, for any $m>\m_\e$ the error bound $\e$ is satisfied. In other words, $\m_\e$ is the modulus of a.s.~convergence. This concept is further generalized to the situation of $\e$ being replaced by some rates of convergence $(\e_n)_{n\in \NN}$, $\e_n \searrow 0$, as $n\to \infty$. 
In this case, $\m_{(\e_n)}$ represents the time step $n$ beyond which a certain a.s.~rate of convergence is always met, $\m_\e=\sup\{n\in\NN~|~ \om \in E_n(\e_n)\}$. \\
We want to stress the following differences between $\oO_\e$ and $\m_\e$: Clearly, $\oO_\e \lqq \m_\e$ a.s.~since the number of errors is always smaller than the last error occurrence. If $\oO_\e(\omega) < \m_\e(\omega)$ for some realization $\omega$, there are exactly $\m_\e(\omega) - \oO_\e(\omega)$ indices $k$ before the last error index $\m_\e$, where the error bound is satisfied, i.e., $\omega \notin E_k(\e)$. Further, the sum of error indices is invariant under a shift $n_0$ and renumbering of the indices, since $\sum_{n=n_0}^\infty \ind(E_n(\e)) = \sum_{n=0}^\infty \ind(E_{n_0+n}(\e)) = \sum_{n=0}^\infty \ind(E_{\si(n_0+n)}(\e))$ for any permutation of $\NN_0$. In the shifted version, we have $\m_\e=\sum_{n=n_0}^\infty \ind\Big(\bigcup_{m=n}^\infty E_m\Big)$. In other words, $\m_\e = \sup\{n-n_0+1~|~n\in\NN, \om \in E_n(\e_n)\}$ (with the convention $\sup\emptyset=0$), such that the maximal index, for which $\omega\in E_n$ (if such an index exists) is $n_0+\m_\e(\omega)-1$. Put differently, the minimal index from which on $\omega\notin E_n$ is given by $n_0+\m_\e(\omega)$. For any choice of $n_0\gqq 0$ this is a function of the events' ordering since moving the last occurrence of the an error event changes $\m_\e$ by definition. Therefore, $\m$ is not invariant for index permutations. \\
Note that whenever the conditions for the first Borel-Cantelli lemma are met, the supremum in the definition of $\m$ can be replaced a.s.~by a maximum, which is the notation used in \cite{EHS26}. The set over which the supremum is taken varies slightly in the various applications. This depends on the first index above which the events $E_n(\e_n)$ are well defined.

The benefit of this concept is fourfold: 
\begin{enumerate}
 \item[a.)] Above all, this manuscript implements an \textit{intuitive and widely applicable statistical quantification of many $\PP$-a.s.~approximations of sample paths}, which measures for each error level $\e$ the \textit{random frequency of level $\e$ deviations from the limit} until finally complying with the error. In particular, we study the deviation frequencies in the almost sure convergence of L\'evy's Haar basis construction of Brownian motion (Theorem~\ref{thm:asL}) with the help of an asymptotic version of results in \cite{EH22} (Proposition~1) with particular focus on i.i.d.~Gaussian sequences (Example~\ref{ex:indepGauss}). The findings obtained this way are applied to prove $\PP$-a.s.~rates of convergence in terms of \textit{mean deviation frequencies} for L\'evy's modulus of continuity of Brownian motion (Theorem~\ref{thm:Levymodulus}), the Paley-Wiener-Zygmund theorem (Theorem~\ref{thm:PWZ}), the quantitative loss of path monotonicity (Theorem~\ref{thm: monotonicity}), the a.s.~convergence to the quadratic variation (Theorem~\ref{thm:QV1} and \ref{thm:QV2}), Khinchin's and Chung's ``other'' law of the iterated logarithm (Theorem~\ref{thm: lil} and \ref{thm:Chung}) and Strassen's functional law of the iterated logarithm (Theorem~\ref{thm:Strassen}) among others. These robustness results are of interest in itself but have also natural applications in simulation and numerical analysis.  
 \item[b.)] In some cases the concept of deviation frequencies has the potential to yield sharpened classical results, such as improved rates of $L^2$-convergence of L\'evy's construction of Brownian motion (Corollary \ref{cor:L2}). 
In addition, we study the Kolmogorov test for the asymptotics of Brownian motion close to $0$, which is classically stated as a probability $0 / 1$ law, and does not allow to distinguish between different asymptotics. In Theorem~\ref{thm: Koltest} we give a precise quantitative distinction of the a.s.~asymptotics in terms of different mean deviation frequencies for different benchmark functions. 
  \item[c.)] The use of the moment estimates in the first Borel-Cantelli lemma in \cite[Lemma~1]{EHS26} illustrates the following tradeoff 
  between a sequence of a.s.~error tolerances, $(\e_n)_{n\in \NN}$, and the error incidence (or deviation frequency) until the random variables finally comply with $(\e_n)_{n\in \NN}$. 
  For any positive and non-increasing sequence $(\e_n)_{n\in \NN}$ as $n\ra\infty$,
  we obtain a sequence of error probabilities $(p_n)_{n\in \NN}$ given with the help of \eqref{e:fehlerereignis} by 
 \begin{equation}
 \PP(E_n(\e_n)) =: p_n. 
 \end{equation}
 If $T_a := \sum_{n=1}^\infty a_n \sum_{m=n}^\infty p_m< \infty$ for some positive weights $a = (a_n)_{n\in \NN_0}$, we find by Theorem~1 in \cite{EH22} and Markov's inequality (which are also the main results in \cite{EHS26}) that for the (discrete) antiderivative $\sS_a(N) = \sum_{n=1}^N a_n$ with $\sS_a(0) = 0$ we have the estimate
 \begin{equation}\label{e:tradeoff}
 \PP(\oO_{\e_n}\gqq k)\lqq \PP(\m_{\e_n}\gqq k)\lqq \sS_a^{-1}(k)\cdot T_a, \qquad k\gqq 1. 
 \end{equation}
 There are two extremal cases: (i) if $\e_n = \e>0$ fixed, we have that $T_a$ is finite for 
 a ``maximally growing'' sequence of weights $a$, which yields (by monotonicity) an equally ``strongest decrease'' on the right-hand side of \eqref{e:tradeoff} of the deviation frequencies $k$. On the other hand, (ii) if $\e_n$ decreases ``maximally'' in the sense that $(p_n)_{n\in \NN}$ is barely summable, we obtain by the usual first Borel-Cantelli lemma and Markov's inequality that the deviation frequencies decay only linearly. Virtually, all mixed regimes between (i) and (ii) can be obtained with the same technique. 
 \hfill\\
 This tradeoff can be stated informally as follows: 
 ``The faster an almost sure error tolerance $\e_n$ descends to $0$, as $n\ra\infty$,  
  the slower decays (in a nonlinear sense) the respective probability for the last index of error
  $\PP(\sup\{n\in \NN~|~ \om \in E_n(\e_n)\}\gqq k)$ in $k$.''\\ 
  The statement remains true if the word ``faster'' and ``slower'' are mutually exchanged. 
   \item[d.)] The deviation count and last deviation index approach in the Borel-Cantelli lemma of \cite{EH22} and \cite{EHS26} can be adapted to many more settings, such as, for instance, the numerical solutions of stochastic (partial) differential equations, L\'evy and additive processes, among others. Example~\ref{ex:BM}-\ref{ex:QVdyadic2} and~\ref{ex:KolmoTest}, and Remark~\ref{rem:QVLevy} illustrate how to implement our findings. See also for instance \cite[Subsection 3.2.2]{EH22}. 
\end{enumerate}

The article is organized as follows: in Section~\ref{s:mainresults} we present the main results and examples, organized in five subsections: \ref{ss:mainratesLevy}, the almost sure convergence results of L\'evy's construction of Brownian motion, \ref{ss:maincontinuity}, continuous versions of Brownian motion, stochastic processes and stochastic fields, 
\ref{ss:mainfine}, fine properties of Brownian paths, \ref{ss:mainlil}, the laws of the iterated logarithm and \ref{ss:BCmain} and \ref{ss:BCmain1} with the asymptotic overlap statistics and an asymptotic quantitative Borel-Cantelli lemma. The proofs are organized - in the respective order of the statements - in the appendices A, B, C, D and E. 

\bigskip


\section{\textbf{The main results}}\label{s:mainresults} 

Our main results are grouped in five subsections with deviation quantification results: L\'evy's construction (Subsec. \ref{ss:mainratesLevy}), Kolmogorov-Chentsov type results (Subsec.~\ref{ss:maincontinuity}), several fine continuity and non-differentiability results of Brownian paths (Subsec.~\ref{ss:mainfine}) and several laws of the iterated logarithm for Brownian paths (Subsec.~\ref{ss:mainlil}). In Subsection \ref{ss:BCmain} and \ref{ss:BCmain1} 
we present - somewhat independently - the asymptotic mean deviation estimates, on which Theorem \ref{thm:asL} relies.

\subsection{\textbf{Rates of almost sure convergence in L\'evy's construction}}\label{ss:mainratesLevy}\hfill\\

\noindent This first subsection is dedicated to {\it P.~L\'evy}'s approximation procedure to obtain a Brownian motion on $[0,1]$ as $\PP$-a.s. uniform limit of a sequence of random functions $L^J_t$, which are linear combinations of Haar basis functions, weighted by i.i.d.~Gaussian random variables (see e.g. \cite{MP10}). 
Theorem \ref{thm:asL}, the main theorem of this subsection, will state that for the a.s.~uniform limit $W:=\lim\limits_{J\to\infty } L^J$, we obtain among others 
\begin{itemize}
\item for all $\alpha>0$ and $J\in \NN$ explicitly known random variables $\Lambda_J(\alpha)>1$ and a deterministic exponential rate $R(J) \searrow 0$ as $J\ra\infty$, such that 
\begin{align*}
\|L^J-W\|_\infty\lqq \sqrt{1+\alpha}\cdot \Lambda_J(\alpha) \cdot R(J),\quad\PP\text{-.a.s.,}\\[-5mm]
\end{align*}
where $\Lambda_J(\alpha)$ has Gaussian moments $\EE [\exp(q \Lambda_J^2(\alpha))]<\infty$ for some $q>0$ (Theorem~\ref{thm:asL}~a),
\item for the deviation frequency $$\oO_\varepsilon:=\sum\limits_{J=0}^\infty\bI\{\|L^J-W\|_\infty>\e\}\quad\text{and}\quad\m_\e(\omega) := \sup\{J+1~|~\|L^J(\om)-W(\om)\|_\infty>\varepsilon\},$$ 
defined for all $\varepsilon>0$, we get 
\begin{align*}
\EE[\exp(p\oO_\varepsilon)]\lqq \EE[\exp(p\m_\varepsilon)]<\infty \qquad \mbox{ for some }p>0
 \qquad \mbox{(Theorem~\ref{thm:asL}~d)} 
\end{align*}
\item rates of exponential type given in Theorem~\ref{thm:asL}\, d, 
$$\PP(\oO_\varepsilon\gqq k) \lqq \PP(\m_\varepsilon\gqq k)\lqq C_1  \cdot 2^{-\frac{k}{2} ~+ \mbox{[lower order terms]}
} \cdot (1+ k^\frac{3}{2})$$
for an explicitly known constant $C_1=C_1(\varepsilon,p)$ and optimized lower order terms.
\end{itemize}
\noindent The full statement of Theorem \ref{thm:asL} contains more refined results, such as a characterization of the modulus of convergence (Theorem \ref{thm:asL} b) and a step-by-step construction error analysis (Theorem~\ref{thm:asL}~c). 
Before we go more into detail, in order to grasp the scope of our a.s.~approximation, we compare it to the historical -- and as well as to the optimal -- approximation in $L^2(\Omega\times[0,1])$ in the following subsection.\\

\noindent \textbf{Wiener's construction: } In his celebrated article \cite{Wie23} {\it N.~Wiener} constructed a Brownian motion over the interval $[0, 1]$ with respect to $L^2(\Omega \times [0,1])$ by the following Fourier series on a probability space which carries an i.i.d.~sequence $(Z_n)_{n\in \NN_0}$ such that $Z_0\sim N(0,1)$ and defined
\[
W^J_t := tZ_0+\sqrt{2}\sum_{k=1}^J \frac{\sin(k \pi t)}{\pi k}\cdot Z_k, \qquad J\in \NN, \quad t\in [0, 1].  
\]
Wiener's proof of convergence calculation is rather involved. 
A stronger statement than convergence in $L^2(\Omega \times [0, 1])$
is given for instance in \cite[p.~167, formula~$(9.21)$]{Ko07} and implies 
\begin{align*}
\EE\big[\|W^{2^{J+1}}-W^{2^{J}}\|_\infty^2\big]^{\frac{1}{2}} 
\lqq \frac{\sqrt{2}}{\pi}(1+\sqrt{2})\cdot 2^{-\frac{J}{4}} ,\quad J\in \NN,\\[-5mm]  
\end{align*}
where $\|\cdot \|_\infty$ is the supremum norm over $[0, 1]$. Hence, by Chebyshev's inequality, $\PP(\|W^{2^{J+1}}-W^{2^{J}}\|_\infty\gqq M)\lqq \frac{2}{\pi^2}\frac{(1+\sqrt{2})^22^{-\frac{J}{2}}}{M^2}$, which is summable. Also $\EE[\|W^{2^{J+1}}-W^{2^{J}}\|_\infty]\lqq \frac{\sqrt{2}}{\pi}(1+\sqrt{2})\cdot 2^{-\frac{J}{4}}$ and $\EE[\|W^{2^{J+1}}-W^{2^{J}}\|_\infty^2]\lqq \frac{2}{\pi^2}(1+\sqrt{2})^22^{-\frac{J}{2}}$ are summable.
Due to the independence of $(W^{2^{J+1}}-W^{2^{J}})_{J\in\NN}$ by construction, the application of Kolmogorov's three series theorem shows that $W:= \lim\limits_{J\ra\infty} W^{2^{J}}$ converges uniformly over $[0, 1]$ in $L^2(\PP)$ and $\PP$-a.s.

Furthermore, 
\begin{align}\label{e:L2infty}
&\EE[\|W^{2^{J+1}}-W\|_\infty^2] 
\!\lqq\!\bigg(\sum_{j=J+1}^\infty  \big(\EE[ \|W^{2^{j+1}}-W^{2^j}\|_\infty^2]\big)^\frac{1}{2}\bigg)^2\lqq \! \frac{2}{\pi^2}(1+\sqrt{2})^2 \Big( \sum_{j=J+1}^\infty 2^{-\frac{j}{4}} \Big)^2\!\nonumber\\
&= \frac{(1+\sqrt{2})^22^{7/4}}{\pi^2(2-2^{3/4})^2}2^{-\frac{J}{2}}.  
\end{align}
While $\|W^{2^{J}}-W\|_\infty \ra 0$ $\PP$-a.s., we cannot derive a meaningful upper bound for the $\PP$-a.s.~rate of convergence: Due to the symmetry of the standard normals it is necessary to establish the absolute convergence of the series in order to obtain a.s.~upper bounds. However, by elementary calculus it is well-known that for Lebesgue-almost all $t\in [0,1]$,
\begin{align*}
\lim_{J\ra\infty} \sum_{k=1}^J \frac{|\sin(k \pi t)|}{k} =\infty. 
\end{align*}
Due to the $L^\infty([0, 1]) \subseteq L^2([0,1])$ embedding, formula \eqref{e:L2infty} yields an upper bound of the original convergence in $L^2(\Omega \times [0,1])$ of order $2^{-J/2}$, which is not optimal among all possible choices of orthonormal bases.\\

\noindent\textbf{The Kosambi-Karhunen-Lo{\`e}ve construction: } It is known for a long time \cite{Ash65, KMR83} that the optimal choice of basis in this topology is given by the widely used \textit{Kosambi-Karhunen-Lo\`eve expansion} of Brownian motion, see \cite{BPS14, Ka47, K043, Lo78}:  
\[
K_t^N := \frac{\sqrt{2}}{\pi} \sum_{k=1}^{N} \frac{\sin((k-\frac{1}{2}) \pi t)}{k-\frac{1}{2}}\cdot Z_k, \qquad N\in \NN, t\gqq 0. 
\]
Still, it suffers the same obvious flaw of the lack of absolute convergence 
\begin{align*}
\lim_{J\ra\infty} \sum_{k=1}^J \frac{|\sin((k-\frac{1}{2}) \pi t)|}{k-\frac{1}{2}} = \infty,\quad \text{Lebesgue-a.e}.
\end{align*}

\noindent \textbf{The scope of almost sure estimates: }
The optimal rate of the Kosambi-Karhunen-Lo\`eve expansion in $L^2(\Omega \times [0, 1])$ 
is known and satisfies 
\begin{equation*}
 \EE\Big[\int_0^1 |K_s^N-W_s|^2 ds\Big]^{\frac{1}{2}} 
= \frac{\sqrt{2}}{\pi} \bigg(\sum_{k=N+1}^\infty \frac{1}{(k-\frac{1}{2})^2}\bigg)^{\frac{1}{2}}
\lqq  \frac{\sqrt{2}}{\pi} \frac{1}{\sqrt{N}}.
\end{equation*}
If we consider ``packages'' or ``generations'' of basis vectors of length $N = 2^J$ (as in L\'evy's construction, which we study below) we obtain the upper bound 
\begin{equation}\label{e:L2upperbound}
 \EE\Big[\int_0^1 |K_s^{2^J}-W_s|^2 ds\Big]^{\frac{1}{2}} 
\lqq \frac{\sqrt{2}}{\pi}\cdot 2^{-\frac{J}{2}}.
\end{equation}
Observe that by the $L^\infty[0,1] \subseteq L^2[0,1]$ embedding, any integrable, $\PP$-a.s. upper bound of $\|K_s^{2^J}-W_s\|_\infty$ is an upper bound of 
$ \EE[\int_0^1 |K_s^{2^J}-W_s|^2 ds]^{\frac{1}{2}}$. 
Therefore, we cannot expect better
rates of  $\PP$-a.s.~convergence than of order $2^{-\frac{J}{2}}$.\\

\noindent \textbf{L\'evy's construction: } Consider the Haar basis $(h_n)_{n\in \NN_0}$ of $L^2[0,1]$ \cite[Chapter 2.3]{KS98}, \cite{Ma97} 
and the Schauder functions given by 
\begin{equation}\label{e:Haarintegral}
\int_0^t h_n(s) ds = 
\int_0^t 2^{-\frac{j}{2}} h_1(2^{\frac{j}{2}} s -k)ds = 
2^{-\frac{j}{2}-1} H(2^j t -k), 
\end{equation}
where $n = 2^j + k$, $k=0, \dots, 2^{j}-1$, $n\geq 1$, and $H(t) := 2t \cdot \ind_{[0, \frac{1}{2}]}(t) + 2(1-t)\cdot \ind_{(\frac{1}{2}, 1]}(t)$, see \cite{Ci60} and \cite{SPB14}. 
Now, we define $H_n(t) := H(2^j t -k)$, $n\geq 1$ and $H_0(t):=t$. 
L\'evy's construction of Brownian motion is then formally given by 
\begin{equation}\label{d:BrownscheBewegung}
L^J_t := \sum_{n=0}^{2^J} Z_n \int_0^t h_n(s) ds = \sum_{n= 0}^{2^J} \la_n \cdot Z_n \cdot H_n(t), 
J\in \NN_0, 
\end{equation}
where $\la_n = 2^{-\lfloor \log_2(n)\rfloor /2-1}$ and an i.i.d.~sequence $(Z_n)_{n\in \NN}$, $Z_n\sim N(0,1)$. 
With the help of \eqref{e:Haarintegral} 
the $\PP$-a.s.~limit of \eqref{d:BrownscheBewegung} and in $L^2$ (see for instance \cite{MP10}) reads 
\begin{equation}\label{e:Levylimit}
W_t = \lim_{J\ra\infty} L^J_t.  
\end{equation}
For comparability of our results below on we define the ``tooth'' function $G_j$ of the $j$-th generation by 
\[
L^J_t = \sum_{j= 0}^J  G_{j}(t), \qquad \mbox{ where } \qquad G_{j}(t) :=  \sum_{k=2^{j-1}}^{2^j-1} \la_k \cdot Z_k \cdot H_k(t). 
\]

\noindent \textbf{Almost sure uniform convergence over $[0, 1]$ with mean deviation frequency: } Our first main result quantifies the almost sure uniform convergence in \eqref{e:Levylimit}. 

\begin{thm}[Rates of almost sure convergence of the L\'evy construction]~\label{thm:asL}\hfill
\begin{enumerate}
 \item[\textnormal{a)}] \textbf{Almost sure random upper bound: } For any $\alpha > 0$ 
 there is a sequence of nonnegative, $\PP$-a.s. non-increasing random variables $(\Lambda_J(\alpha))_{J\in \NN}$ such that\\ 
 \begin{align}\label{eq:zufaelligeschranke}
 \|L^J - W\|_\infty 
\lqq \sqrt{1+\alpha} \cdot  \max\{\Lambda_{J}(\alpha), 1\} 
 \cdot C_a \cdot \sqrt{J+1}\cdot 2^{-\frac{J}{2}}, \qquad \PP\mbox{-a.s. }\\[-3mm]\nonumber
 \end{align}
for all $J\gqq 1$, where $C_a=\sqrt{\tfrac{2}{\ln(2)}}\big(1 + \tfrac{1}{2\ln(2)}\big)\approx 2.9240.$\\
For each $\al>0$, $J\in \NN$ 
and $0 < q < (1+\al) J$ we have Gaussian moments for $\Lambda_J(\alpha)$\\
 \begin{align}\label{e:Lambdamoment}
&\EE\Big[2^{q [\max\{\Lambda_{J}^2(\alpha),\, 1\}-1]}-1\Big]\nonumber\\
&\qquad \lqq  \frac{q\big(1+\frac{1}{2\ln(2)\al J}\big)(2^{J+1}-1)\sqrt{\pi}\big(1-\Phi\big(\sqrt{2\ln(2)((1+\alpha)J-q)}\big)\big)}{2^{q-1}(1+\al)^{1/2}\al \sqrt{J}\sqrt{(1+\alpha)J-q}}
\end{align}  
where $\Phi$ is the cdf of a standard normal variable.\\

 \item[\textnormal{b)}] \textbf{Almost sure deterministic upper bound: }
 For all $\alpha>0$ there exists an $\NN_0$-valued random variable $\jJ(\al)$ such that $\PP$-a.s.\\ 
  \begin{align}\label{eq:thm1b}
 \|L^J - W\|_\infty \lqq \sqrt{1+\alpha}\cdot C_a \cdot \sqrt{2\ln(2)} \cdot \sqrt{J+1}\cdot 2^{-\frac{J}{2}}\qquad \mbox{ for all } J\gqq \jJ(\alpha),
 \end{align}
where $C_a \cdot \sqrt{2\ln(2)}\approx 3.4427$ and
\begin{align*}
\frac{2^{-\frac{2}{\alpha}}\big({2^{\alpha}}-(1-2^{-J})^{-\alpha}\big)}{\alpha\sqrt{2\ln(2)(1+\alpha)}}\cdot\frac{2^{-J\alpha}}{\sqrt{J}}\lqq\PP(\jJ(\alpha)=J)\lqq \Big(\frac{2^\alpha-1}{\alpha}+\frac{2^\alpha}{2^{J-1}}\Big)\cdot 2^{-\alpha J},
\end{align*}
where the upper bound holds for $J\gqq 1$ and the lower for $J\gqq 2$.

 \item[\textnormal{c)}] \textbf{Step-by-step adapted error deviation frequency: } For any $\alpha> 0$, \[\e_j := \sqrt{1+\al} \cdot \sqrt{2 \ln(2)}\cdot \sqrt{j} \cdot 2^{-\frac{j}{2}}, \qquad j\in \NN,\] 
 and  $J\in \NN_0, k\in \NN$ we have the following \textbf{deviation frequency quantification}. 
 
 For $$\oO_J := \#\{j\gqq J+1~\vert~\|L^j-L^{j-1}\|_\infty > \e_j\}$$ we have 
 \begin{align}
 \PP\big(\oO_J \gqq k\big)\lqq 4\exp\bigg(-\big(\frac{\alpha\ln(2)(J+k)-\ln\big(\frac{1}{1-2^{-\alpha}}\big)\big)^2}{4\alpha\ln(2)}\bigg).\label{e:mdfLevy}\\\nonumber
 \end{align}
 

 \item[\textnormal{d)}] \textbf{Fixed $\e$-error deviation frequency and last deviation: }
 We have 
 $\lim\limits_{J\ra\infty} L^J = W$ uniformly on $[0,1]$ 
 \textbf{almost surely with exponential mean deviation frequency and last deviation}, in the sense of \cite[Definition 1]{EHS26}: For any $\e>0$ the overlap statistic 
 \[
 \oO_{\e}:= \sum_{J=1}^\infty \ind\{\|L^J -W\|_\infty >\e\}\qquad \mbox{ and }\qquad 
 \m_{\e}:= \sum_{J=1}^\infty \ind\Big(\bigcup_{j\gqq J}  \{\|L^j -W\|_\infty >\e\}\Big)
 \] 
 satisfies for all $0 \lqq p < \frac{\ln(2)}{2}$ and $C := \frac{2-\sqrt{2}}{\ln(2)}\Big(1+\frac{1}{\ln(2)}\Big)\Big(\frac{1}{\sqrt{2}}+\sqrt{\pi}\Big)\cdot C_a\cdot c\approx 30.9704$ that 
\begin{align}
\EE[e^{p\oO_\e}]\lqq \EE[e^{p \m_\e}] \lqq 1+  \tfrac{C}{\e}\Big(\tfrac{\ln(2)}{2} - p\Big)^{-\frac{3}{2}},
\quad \label{e:LevOverlap}
\end{align} 
with $C_a$ given in item a) and $c$ in \eqref{e:c1}.
Furthermore, for all $k\gqq 3$ and $\e>0$ we have that\\ 
 \begin{align}
\PP(\oO_\e \gqq k)
 \lqq  \PP(\m_\e \gqq k)\lqq e\left(1+\frac{C}{\e} \cdot k^{\frac{3}{2}}\right)2^{-\frac{k}{2}}\,. \label{eq:minimizeplevy}
  \end{align}
 \item[\textnormal{e)}] \textbf{Almost sure close to optimal rate $\delta_J$ with linear decay of the deviation frequency: } For all $\theta>0$, $J\in \NN_0$ and 
 \begin{equation}\label{e:Levydeltarate}\delta_J := 2^{-\frac{J}{2}} (J+1)^{\frac{3}{2}}\ln(J+1)^{1+\theta} \end{equation} 
 we have for 
 \[
 \oO_\delta := \#\{J\in \NN_0~\vert~ \|L^J-W\|_\infty>\delta_J\}
 \]
 that for $k\geq 1$ a.s.
 \begin{align}\label{e:LevydeltaMDF}
\PP(\oO_\delta \gqq k) \lqq k^{-1}\cdot \frac{C_a\cdot c}{\ln(2)^\theta}\cdot \Big(\frac{1}{2\ln(2)} + \frac{1}{\theta}\Big).
\end{align}

\end{enumerate}
 \end{thm}
 \bigskip
\noindent The proof of Theorem~\ref{thm:asL} a) is given in Appendix~\ref{ss:randomupper} combined with Appendix~\ref{a:fastsicherextrem}. 
Theorem~\ref{thm:asL} c) is shown in Appendix~\ref{ss:stepbystep} using Proposition~\ref{prop: asymptotic}.  Theorem~\ref{thm:asL} b) is proved in Appendix~\ref{ss:detupperbound}, Theorem~\ref{thm:asL} d) in Appendix~\ref{ss:exmdf} and Theorem~\ref{thm:asL} e) in Appendix~\ref{a:fsoptimaleRate}. 

\begin{rem}
\begin{enumerate}
 \item To our knowledge inequalities of type \eqref{eq:zufaelligeschranke} have been known in the literature only asymptotically as $n\ra\infty$, that is, starting from a random index $N = N(\omega)$, see for instance \cite[Section 3.2, p.31, last display]{SPB14} or equivalently with an unspecified random upper bound, see \cite[Lemma 3.2]{Ste01}. We give a close to optimal numerical upper bound $C_a$ and the Gaussian integrability of the random part in \eqref{e:Lambdamoment}. For a more structural viewpoint on estimates of this type, we also refer to \cite[Section 2.4]{Ta14}. 
  \item After simulating $J$ generations of teeth, 
  the probability of having $k$ or more error steps of minimal size $\sqrt{1+\al} \cdot \sqrt{2 \ln(2)}\cdot \sqrt{j} \cdot 2^{-j/2}$ in item c), 
 decreases as fast as a Gaussian tail both in $J$ and $k$. 
 A fortiori, our result yields that the probability that the last index of such an error is more than $k$ steps ahead decays with the same rate.
 \item We stress the tradeoff between rates in $\delta_J$ and the mean deviation frequency and maximum.  
 For a fixed error bar $\delta_J = \e$ we obtain the fastest decay of this type of the deviation frequencies in \eqref{eq:minimizeplevy} of Theorem~\ref{thm:asL}, d). 
 On the other hand for a close to optimally small rate of almost sure convergence $\delta_J$ 
 in \eqref{e:Levydeltarate} of Theorem~\ref{thm:asL}, e) 
 we obtain with the barely linear decay of the deviation frequencies in \eqref{e:LevydeltaMDF}. \\
\medskip 
 Note that virtually all mixed regimes between suboptimal a.s.~rates of convergence $\delta_n \lqq \e_n$ and higher order mean deviation frequencies (resp.~the mean index of the last deviation) can be obtained with the same technique. This freedom is described as the ``tradeoff between the mean error tolerance and the mean error maximum (resp. frequency) in \cite{EHS26}''.
\end{enumerate}
\end{rem}

A simple consequence of Theorem~\ref{thm:asL} a) yields a consistently comparable result to Wiener's construction and the Kosambi-Karhunen-Lo{\`e}ve expansion in \eqref{e:L2upperbound} and shows near optimality (up to a factor $\sqrt{J+1}$) in $L^2(\Omega, L^\infty([0,1]))$. 

\begin{cor}\label{cor:L2}
Under the assumptions of Theorem~\ref{thm:asL}, we have for $J\in \NN$ 
\begin{align*}
\EE\Big[ \|L^J  - W \|_\infty^2\Big]^\frac{1}{2} \lqq C_a\cdot c \cdot \sqrt{J+1}\cdot  2^{-\frac{J}{2}}
\end{align*}
where \begin{equation}\label{e:c1}
c\approx 2.0693,
\end{equation}
and $C_a$ was given in Theorem \ref{thm:asL} a). An approximate value for the constant is $C_a\cdot c\approx 6.0505$.
 \end{cor}
\noindent The proof of Corollary~\ref{cor:L2} is given in Appendix~\ref{ss:randomupper}. 
Obvious extensions are valid for any $L^p$ distance, $p\gqq 1$.

\medskip 
 \subsection{\textbf{Deviation frequencies from continuity and H\"older continuity}}\label{ss:maincontinuity} \hfill\\ 
 
 \noindent Apart from direct constructions of stochastic processes, such as in L\'evy's construction of the preceding subsection, or the general theory of processes \cite{DM82, Pr04}, one of the standard tools to show the almost sure H\"older continuity of paths of stochastic processes (and in particular of Brownian motion) is the Kolmogorov-Chentsov theorem. Following \cite{CMY04}, the result was first presented orally in 1934 by Kolmogorov in the Seminar of the University of Moscow, later published by Slutsky \cite{Sl37} and extended by Kolmogorov and Chentsov in \cite{Ch56}. 
 The result is stated in classical monographs, 
 such as \cite{BS02, Ch82, GS04, IW81, Kal97, Kle08, KS98, KR18, Pr04, RW00, RY99} and in most of them used to show the continuity of the Brownian motion and the Brownian Bridge. In the sequel we give a quantitative version for processes  on $[0, 1]$ and for random fields with parameters in a bounded (open, connected, nonempty) domain $\dD\subseteq \RR^d$, which was given in \cite{Ku04} and goes back to the original article by Totoki~\cite{To61}.

 \subsubsection{\textbf{Almost sure continuity of Brownian paths (following J. Doob): }}

 We start with a quantitative version of the ad-hoc continuity result 
 for the special case of Brownian motion in \cite[Theorem, p.~577]{Do84}. 
 Similar explicit calculations are found in \cite{HWY92}.

  \begin{thm}\label{thm:Doob}
  Let $X = (X_t)_{t\in [0, 1]}$ be a scalar Brownian motion in law, that is, $X_0=0$ $\PP$-a.s., $X_{t}-X_{s}\sim N(0,t-s)$ for $s<t$, and $X$ has independent increments. Then there is a continuous version of $X$. In particular, we have the following \textbf{deviation maximum quantifications}: 
  \begin{enumerate}
  \item For all $k\gqq 1$, $\e>0$, and setting
  \[
  \oO_\e := \#\Bigg\{n\in \NN~\vert~ \sup_{\substack{r,s\in\QQ\cap [0,1]\\ |s-r|\lqq\frac{1}{n}} } |X_s-X_r| \gqq 2\e\Bigg\}, \quad \mbox{ and }
  \]
  \[\m_\e := \sup\Bigg\{n\in \NN~\vert~ \sup_{\substack{r,s\in\QQ\cap [0,1]\\ |s-r|\lqq \frac{1}{n}}} |X_s-X_r| \gqq 2\e\Bigg\},\]
we have 
 \begin{align}\label{e:Doob}
\PP\big(\oO_\e \gqq k\big)\lqq \PP\big(\m_\e \gqq k\big)\lqq e\bigg(1+K\frac{(\e^2+2) e^{\frac{\e^2}{4}}(e^{\frac{\e^2}{4}}-1)}{\e^5} \cdot k^{\frac{3}{2}}\bigg)\cdot e^{-\frac{\e^2}{4}k}
\end{align}
for $k> \frac{4}{\e^2}$ where $K= 16 \big(\tfrac{1}{\sqrt{2}}+\sqrt{\pi}\big) \approx 39.6730$. 

\item For the scale $\epsilon = (\e_n)_{n\in \NN}$ with $\e_n := 2 \sqrt{ \tfrac{\theta\ln(n+1)}{n}}$, $\theta>2$ 
and 
\begin{align}\label{e:Doob2}
\oO_\epsilon := \# \Bigg\{n\in \NN~\vert~ \sup_{\substack{r,s\in\QQ\cap [0,1]\\ |s-r|\lqq \frac{1}{n}}} |X_s-X_r| \gqq 2\e_n\Bigg\} 
\end{align}
we get a barely linear decay of the tail of the mean deviation frequency  
 \begin{align}\label{e:Doob3}
\PP\big(\oO_\epsilon \gqq k\big)\lqq k^{-1}\cdot\frac{4 \zeta(\theta-1)}{\sqrt{\theta\ln(2)}}, 
\end{align}
where $\zeta(t) = \sum_{n=1}^\infty n^{-t}$ is Riemann's zeta function. 
Further, for $-1<p<\theta-3$, we get
\begin{align}\label{e:Doob4}
\PP\big(\oO_\epsilon \gqq k\big)\lqq\PP(\m_{\epsilon}\gqq k)\lqq k^{-(p+1)}\cdot \left(\frac{4}{\sqrt{\theta \ln(2)}}(\theta-1) \zeta(\theta-2-p)+\ind(\{p<0\})\right).
\end{align}
\end{enumerate}
 \end{thm}
 \noindent 
The proof is given in Appendix~\ref{ss:Doob}. 
    
   \begin{rem}
 Note that \eqref{e:Doob2} yields a rather precise rate of convergence (as compared to L\'evy's modulus of continuity in Theorem~\ref{thm:Levymodulus}) up to a multiplicative logarithm and the constant $\theta$ vs.~$2$. However, the mean deviation frequency in \eqref{e:Doob3} is barely linear, in contrast to the exponential and Gumbel decay of the mean deviation frequencies in Theorem~\ref{thm:Levymodulus}. 

 This can be modified by the tradeoff explained in \cite[Lemma~2]{EHS26}, we could improve the rate of the tail decay of $\m_\e$ by going over to suboptimal a.s.~rates $0 < \tilde \e_n < \e_n$. Note that in \eqref{e:Doob3} $\m_\epsilon$ is replaced by the deviation frequency $\oO_\epsilon$.
 \medskip      
 
We also refer to \cite[Section 1.2.6), p.12]{Fr83} with additional calculations concerning the modulus of continuity discussed in Appendix~\ref{ss:Lmc}.  
   \end{rem}

 \medskip

 \subsubsection{\textbf{H\"older continuous versions of a stochastic process (Kolmogorov, Chentsov): }}
 
\begin{thm}[Kolmogorov, Chentsov]\label{thm:KC} 
Consider a stochastic process $X = (X_t)_{t\in [0, T]}$ 
with values in a separable normed space $(B, \|\cdot \|)$ on a given 
probability space $(\Omega, \aA, \PP)$. Let $X$ satisfy the following 
moment condition. There are positive constants $\alpha, \beta$ and $C$ such that 
\[
\EE[\|X_t - X_s\|^{\alpha}] \lqq C\,|t-s|^{1+\beta}, \qquad \mbox{ for all }s,t\in [0,T]
.
\]
\begin{enumerate}
 \item Then there exists a continuous modification $\ti X = (\ti X_t)_{t \in [0, T]}$ of $X$ which has locally H\"older continuous paths for any H\"older exponent $\gamma \in (0, \frac{\beta}{\alpha})$. 
\item We have the \textbf{last deviation quantification}: 

For 
$$
\m(\om) := \sup\Big\{n+1~|~n\in \NN_0, \max_{\ell\in \{1, \dots, 2^n\}} \|X_{T\frac{\ell}{2^n}}(\om)-X_{T\frac{\ell-1}{2^n}}(\om)\|> T^\gamma2^{-\gamma n}\Big\}
$$ we have a.s.~for all $n\gqq \m$ that 
\[
\max_{\ell\in \{1, \dots, 2^n\}} \|X_{T\frac{\ell}{2^n}}-X_{T\frac{\ell-1}{2^n}}\|\lqq T^\gamma2^{-\gamma n}, 
\]
and for all $k\gqq 1$ 
 \begin{align}
 \PP\big(\m
 \gqq k\big)
 &\lqq e^{\frac{5}{4}} 
 \cdot (2k+1)\Big(\frac{CT^{1+\beta-\gamma\alpha}(2^{(\beta -\alpha\gamma)}-1)}{1-2^{-(\beta - \alpha \gamma)}} + 1\Big) \cdot 2^{-k(\beta-\alpha\gamma)}.  
 \label{e:Chentsovbound2}
 \end{align}
\end{enumerate}
\end{thm}
\noindent The proof is given in Appendix~\ref{ss:KC}. Again we refer for possible refinements in the tradeoff between the possible values $\e_n$ replacing $T^\gamma2^{-\gamma n}$ and different kinds of tail decays of $\m$ to \cite[Lemma~2]{EHS26}. Clearly, the $\max$ in the definition of $\m$ can be directly replaced by the weaker $\#$ operator, which yields estimates for the respective deviation frequency $\oO$.

\begin{exm}\label{ex:BM}
For $X$ being a scalar Brownian motion, 
for any $\al>2$, we set $\beta := \frac{\al-2}{2}$ and $D_\al := \EE[|\nN|^\al] = \frac{2^{\frac{\alpha}{2}}}{\sqrt{\pi}} \Gamma(\frac{\al+1}{2})$ for $\nN\sim N(0,1)$. 
It is not hard to see that 
\begin{align*}
\EE[|X_t - X_s|^{\al}] =D_\al |t-s|^{1+ \beta} \qquad \mbox{ for all } s, t\in [0,1],
\end{align*}
and the tradeoff statements of Theorem \ref{thm:Doob}, and Theorem \ref{thm:KC} for any $\gamma\in (0,\tfrac{1}{2}-\tfrac{1}{\alpha})$ apply.
 \end{exm}

\begin{exm}\label{ex:fBM}
It is obvious how to generalize the preceding example to fractional Brownian motion with Hurst index $H\in (0,1)$. 
For $X$ being such a fractional Brownian motion, we take $\alpha>\frac{1}{H}$, set $\beta:=H\alpha-1$  and take the same constant $D_\alpha$ as before. Then,  again,\begin{align*}
\EE[|X_t - X_s|^{\al}] =D_{\al} |t-s|^{1+ \beta} \mbox{ for all } s, t\in [0,1],
\end{align*}
and the tradeoff statements of Theorem \ref{thm:Doob}, and Theorem \ref{thm:KC} for any $\gamma\in (0,H-\tfrac{1}{\alpha})$ apply.
 \end{exm}

\medskip 
\subsubsection{\textbf{H\"older continuous versions of a random field (Kolmogorov, Totoki): }}\hfill\\[-3mm]

\noindent  We present a quantitative version of the result in \cite[Theorem 4.1]{Ku04}. Consider the lattice 
\[
\lL_n := \Big\{\Big(\frac{i_1}{2^n}, \dots, \frac{i_d}{2^n}\Big)~\big|~i_1, \dots, i_d\in \ZZ\Big\}, \quad \mbox{its (dense) union, } \quad 
\lL := \bigcup_{n\in \NN_0} \lL_n,
\]
and a bounded domain $\dD\subseteq \RR^d$ and a normed space $(B, \|\cdot\|)$. For any $f: \lL \cap \dD \ra B$ define 
$\Delta_n(f) := \max_{\substack{x,y \in \lL_n \cap \dD\\|x-y| = 2^{-n}}} \|f(x) - f(y)\|$ and 
$\Delta_n^\gamma(f) := 2^{n\gamma}\cdot \Delta_n(f)$. 

\begin{thm}[Kolmogorov, Totoki]\label{thm:KolmoToki}
Consider a bounded domain $\dD\subseteq \RR^d$ and a random field $X = (X(x))_{x\in \dD}$ with values in a normed space $(B, \|\cdot\|)$ over a given probability space $(\Omega, \aA, \PP)$. Assume that there exists positive constants $\beta, C$ and $\alpha$ such that 
\begin{align}\label{e:KolmoToki}
\EE\Big[\|X(x)-X(y)\|^{\alpha}\Big] \lqq C |x-y|^{d+\beta},\qquad x, y\in \dD.  
\end{align}
Then $X$ has a $\gamma$-H\"older continuous modification $\ti X$ for any $\gamma \in (0, \frac{\beta}{\alpha})$. Moreover, we have the following 
\textbf{deviation frequency quantification}: 
For any $\delta>0$ such that $\beta -\alpha\gamma-\delta\alpha>0$ we have 
for $\m_{\gamma}(\omega) := \sup\{n+1~\vert~n\in \NN, \Delta_n^\gamma(X)> 2^{-\delta n}  \}$ that 
\[
\Delta_n^\gamma(X)\lqq 2^{-\delta n} \qquad \mbox{ for all } n\gqq\m_\gamma \qquad \PP\mbox{-a.s.} 
\]
At the same time 
\begin{align*}
\PP(\m_\gamma\gqq k)
 &\lqq e^{\frac{5}{4}} \cdot (2k+1)(2^{d +\beta -\alpha\gamma-\delta\alpha} \text{\emph{vol}}(\operatorname{conv}(\dD))+1)\cdot 2^{-k(\beta -\alpha\gamma-\delta\alpha)}, \quad k\gqq 1,
\end{align*}
where $\operatorname{conv}(\dD)$ is the convex envelope of $\dD$.
\end{thm}
\noindent The proof is given in Appendix~\ref{ss:KT} combined with Lemma~\ref{lem:optimal} for $M = 2^{d+\beta -\alpha \gamma-\delta \alpha}\, \mbox{vol}(\operatorname{conv}(\dD))$ in Appendix~\ref{a:optimal}. 
Possible refinements in the tradeoff between the possible values $\e_n$ replacing $2^{-\gamma n}$ and different kinds of tail decays of $\m$ can be read off from \cite[Lemma~2]{EHS26}. There it is shown that the $\max$ in the definition of $\m$ can be directly replaced by the weaker $\#$ operator, which yields estimates for the respective deviation frequency $\oO$. 

\begin{exm}[Brownian sheet]\label{ex:BS}
Let $(X_{(t,s)})_{t,s\in [0,1]}$ be a Brownian sheet. Then, for $\alpha>4$, we get
\begin{align*}
\EE\left[\left\|X_{(t,s)}-X_{(t',s')}\right\|^\alpha\right]=D_\alpha|ts-2\min(t,t')\min(s,s')+t's'|^\frac{\alpha}{2}.
\end{align*}
Since a constant $c$ can be found such that $|ts-2\min(t,t')\min(s,s')+t's'|\lqq c|(t,s)-(t',s')|$, we get that

\begin{align*}
\EE\left[\left\|X_{(t,s)}-X_{(t',s')}\right\|^\alpha\right]\lqq D_\alpha c^{\frac{\alpha}{2}}|(t,t')-(s,s')|^{2+\frac{\alpha-4}{2}}.
\end{align*}
Hence we can proceed again as in the case for Brownian motion and we obtain, setting $\beta=\frac{\alpha-4}{2}$, the H\"older continuity of $X$ on $[0,1]$ for all exponents $\gamma\in (0,\frac{1}{2})$ (since the limit $\lim\limits_{\alpha\to\infty}\frac{\beta}{\alpha}=\lim\limits_{\alpha\to\infty}\frac{\alpha-4}{2\alpha}=\frac{1}{2}$) as well as the other subsequent results.

\end{exm}

\medskip 

\subsection{\textbf{Deviation frequencies of fine continuity properties}}\label{ss:mainfine}
\subsubsection{\textbf{L\'evy's modulus of continuity }}\label{ss:Lmc}\hfill\\ 

\noindent One of the fascinating features of Brownian sample paths is the precise knowledge of their continuity properties, such as its (global) modulus of continuity, established in \cite{Lev37}. Being stated like that in textbooks such as 
\cite{KS98, MP10, BS02, RY99, IMcK74, McK69, Kn81, Kle08, Kal97, SPB14} or \cite{RW00}, one might naively assume that the lower and the upper bound might behave somewhat symmetric. This, however, is completely wrong. A noteworthy exception is \cite[Chapter 5.1, Theorem 5.3]{Ste01}, where the author stresses this asymmetry in particular. We refer to \cite{IMcK74} for different types of expansions. 

We see below, that the number of upward infringements of the asymptotic rate are extremely more unlikely than downward infringements and essentially exhibits a doubly exponentially Gumbel type decay. Such a behavior is finally not surprising in the light of extreme-value distributions, while downward infringements die out with merely an exponential decay, coming from the independence of increments. This asymmetry is hidden in the original statement, let alone being quantified.

\begin{thm}\label{thm:Levymodulus}
For $\mu: (0, 1]\ra (0, \infty)$ given by $\mu(\delta) := \sqrt{2\delta \ln(1/\delta)}$, $\delta>0$ we have 
\begin{align*}
\limsup_{\delta\searrow 0} \max\limits_{\substack{0\lqq s < t\lqq 1\\|t-s| \lqq \delta}}|W_s -W_t| \cdot \mu(\delta)^{-1} = 1\qquad \PP\mbox{-a.s.} 
\end{align*}
In addition, we have the following quantitative statements: 
\begin{enumerate}
 \item For any $\theta\in (0,1)$, $0 < \eta < \theta$ and any $0 \lqq p < \frac{1}{e^{\eta} \sqrt{4\pi (1-\theta)}}$ and 
 \[
 \m_{\theta} := \sup\Big\{n -\lceil\tfrac{1}{2(1-\theta)}\rceil+1~\vert~n\in \NN_0, \frac{\max_{1\lqq j\lqq \lfloor e^{n}\rfloor} |W_{\frac{j}{e^n}}- W_{\frac{j-1}{e^n}}|}{\mu(e^{-n})} \lqq \sqrt{1-\theta}\Big\}, 
 \]
 and 
 \[\oO_{\theta} := \#\Big\{n\gqq \tfrac{1}{2(1-\theta)}~\vert~\frac{\max_{1\lqq j\lqq \lfloor e^n\rfloor} |W_{\frac{j}{e^n}}- W_{\frac{j-1}{e^n}}|}{\mu(e^{-n})} \lqq \sqrt{1-\theta}\Big\},\]
for a.a.~$\omega\in \Omega$ we have that  $n\gqq \lceil\tfrac{1}{2(1-\theta)}\rceil+\m_{\theta}(\omega)$ implies 
 \[
 \frac{\max_{1\lqq j\lqq \lfloor e^{n}\rfloor} |W_{\frac{j}{e^n}}- W_{\frac{j-1}{e^n}}|}{\mu(e^{-n})} > \sqrt{1-\theta}. 
 \]
 In addition, there is a constant $\kK = \kK_{\theta,\eta, p}>0$ such that for all $k\gqq 1$ 
\begin{align*}
\PP\big(\oO_{\theta} \gqq k\big)\lqq \PP\big(\m_{\theta} \gqq k\big) \lqq \kK\cdot \exp( -p \exp(\eta k)).
\end{align*}
\medskip 
 \item For any $\theta \in (0,1)$ and $\e > \frac{1+\theta}{1-\theta}-1$ and $\rho = (1-\theta)(1+\e)^2-(1+\theta)$ 
there is a positive constant $K_{\e, \theta}>0$ such that for 
\[
 \m_{\e, \theta} :=\sup\Big\{n -\lceil\tfrac{1+\ln(1+\theta)}{1-\theta}\rceil+1~|~n\in \NN, \max_{\substack{0\lqq i < j \lqq \lfloor e^n\rfloor \\ 1\lqq j-i \lqq \lceil e^{n\theta}\rceil }} 
\frac{|W_{\frac{j}{e^n}}(\omega) -W_{\frac{i}{e^n}}(\omega)|}{\mu((j-i)e^{-n})} \gqq 1+ \e\Big\}
\]
and 
\[ \oO_{\e, \theta}:= \#\Big\{n\gqq \tfrac{1+\ln(1+\theta)}{1-\theta}~|~ \max_{\substack{0\lqq i < j \lqq \lfloor e^n\rfloor \\ 1\lqq j-i \lqq \lceil e^{n\theta}\rceil }} 
\frac{|W_{\frac{j}{e^n}}(\omega) -W_{\frac{i}{e^n}}(\omega)|}{\mu((j-i)e^{-n})} \gqq 1+ \e\Big\},\]
we have for a.a.~$\omega\in \Omega$ that for all $n\gqq  \lceil\tfrac{1+\ln(1+\theta)}{1-\theta}\rceil+\m_{\e, \theta}(\omega),$ 
\begin{align*}
 \max_{\substack{0\lqq i < j \lqq \lfloor e^n\rfloor \\ 1\lqq j-i \lqq \lceil e^{n\theta}\rceil }} 
\frac{|W_{\frac{j}{e^n}} -W_{\frac{i}{e^n}}|}{\mu((j-i)e^{-n})} \lqq 1+ \e,
\end{align*}
moreover, it holds  
\begin{align}
\PP\big( \oO_{\e, \theta} \gqq k\big) &\lqq \PP\big( \m_{\e, \theta} \gqq k\big)\nonumber\\
&\lqq  
e^{\frac{5}{4}} \cdot (2k+1) \Big(\frac{K_{\e, \theta}}{1-e^{-\rho}} e^{-\rho(\lceil \frac{1+\ln(1+\theta)}{1-\theta}\rceil-1)} +1\Big) \cdot e^{-\rho k}, \qquad k\gqq 1.
\label{e:Lmclower}
\end{align}
\end{enumerate}

\end{thm}
\noindent The proof is given in Appendix~\ref{ss:Lmcproof}. 

\begin{rem}
Note that the convergence result depends strongly on the approximation. In this case the grid is chosen as $(i/e^n)_{i=0, \dots, \lfloor e^n\rfloor}$ in order to expose the extreme value Gumbel type law, as shown for instance in \cite[Example 3.5.4, p.174]{KEM98}.
\end{rem}

\medskip 
\subsubsection{\textbf{The quantitative blow up of Brownian secant slopes (Paley, Wiener, Zygmund)}}\hfill\\
    
\noindent In the original paper \cite{PWZ33}, its authors studied consequences of the finite variation properties 
applied to Fourier series, in particular, Wiener's construction of Brownian motion. A modern proof of this result is given for instance in \cite{KS98} or \cite{Ko07}. The support of Brownian motion on a subset of nowhere differentiable paths can be quantified in terms of an average secant blow up for a dyadic approximation. Note that $D^+f$ and $D^- f$ mean the right and left upper Dini derivative of a function $f$ \cite{KS98}. $D_+f$ and $D_- f$ denote the right and left lower Dini derivative, respectively. 

\begin{thm}[Paley, Wiener, Zygmund]\label{thm:PWZ}
The event 
\begin{align*}
\{\om \in \Omega~\vert~\mbox{ for each }t\in [0, 1] \mbox{ either }D^+ W_t(\omega) =  \infty  \mbox{ or }D_+ W_t(\omega) =  -\infty\} 
\end{align*}
contains an event $E \in \aA$ with $\PP(E) = 1$. 
Moreover, we have the following \textbf{deviation frequency and last deviation quantification}: For 
$c_\pi:= \frac{2^{10}}{\pi^2}$ and any $R\in (1, 2^{1/4})$ we have 
for 
\[
\m_R := \sup\Big\{n+1~|~ ~\vert~\exists\,s\in [0, 1]:~\sup_{t\in [s-2^{-n}, s+2^{-n}]\cap [0,1]} \frac{|W_s-W_t|}{2^{-n}} \lqq   R^n\Big\}, 
\]
and 
\[\oO_R := \#\Big\{n\in \NN_0~|~ ~\vert~\exists\,s\in [0, 1]:~\sup_{t\in [s-2^{-n}, s+2^{-n}]\cap [0,1]} \frac{|W_s-W_t|}{2^{-n}} \lqq   R^n\Big\},\] 
that for all $\omega\in E$ and $n\gqq \m_R(\omega)$ we have 
\[
\forall\,s\in [0, 1]:~\sup_{t\in [s-2^{-n}, s+2^{-n}]\cap [0,1]} \frac{|W_s-W_t|}{2^{-n}} > R^n
\]
and for all $k\in \NN$ 
\begin{align}
\PP\big(\oO_R\gqq k)\lqq \PP\big(\m_R\gqq k) \lqq  e^{\frac{5}{4}} \cdot (2k+1)\Big(\frac{2c_\pi}{R^4}+1\Big)\cdot \Big(\frac{R^4}{2}\Big)^{ k}, \quad k\gqq 1.\label{e:PWZ}
\end{align}
\end{thm}
\noindent The proof is given in Appendix~\ref{ss:PWZ}. 

\begin{rem}
Again we refer to \cite[Lemma~2]{EHS26} for different tradeoffs between the a.s.~explosion rate of the secant and the integrability of the modulus of convergence.

The rate in terms of $R^4<2$ seems to be close to optimality since in the original proofs it is crucial to compare two neighboring intervals of the approximation and its respective left and right neighbor. 

The constant $c_\pi = \frac{1024}{\pi^2} \approx 103.74389$ appears naturally as the power $4$ of the constant $\frac{8}{\sqrt{2\pi}}$ coming from a Gaussian tail approximation of Brownian increments. 
\end{rem}

\medskip
\subsubsection{\textbf{The quantitative loss of monotonicity of Brownian sample paths}}\hfill\\ 

\noindent One of the characteristic features of a Brownian path is its roughness in the sense of non-monotonicity in any interval. Even stronger, paths exhibiting any point of increase or decrease only appear with probability zero. This goes back to the independence of the increments. See for instance \cite{KS98,MP10}.

\begin{thm}\label{thm: monotonicity}
For $\PP$-almost all $\om \in \Omega$ the path $W_\cdot(\om)$ is monotone in no subinterval of $[0, 1]$. 
In addition, we have the following \textbf{deviation frequency and last deviation quantification}. 
For 
\[
\m(\om) := \sup\Big\{n\in \NN ~\vert~ \bigcap_{i=0}^{n-1 }\{\om \in \Omega~\vert~W_{\frac{i+1}{n}}(\omega) -  W_{\frac{i}{n}}(\omega)\gqq 0\Big\}
\]
and 
\[\oO(\om) := \#\Big\{n\in \NN ~\vert~ \bigcap_{i=0}^{n-1 }\{\om \in \Omega~\vert~W_{\frac{i+1}{n}}(\omega) -  W_{\frac{i}{n}}(\omega)\gqq 0\Big\}\] 
we have 
\begin{equation}\label{e:monotonicity}
\PP\big(\oO\gqq k)\lqq \PP(\m\gqq k) 
\lqq e^{\frac{5}{4}}\cdot 2(2k+1)\cdot 2^{-k} \qquad \mbox{ for all }
\qquad k\gqq 1. 
\end{equation}
\end{thm}
\noindent The proof is given in Appendix~\ref{ss:QLM}. 

\begin{rem}
Note that the proof - including the same rates - remains the same for all processes starting in $0$ with independent increments and a continuous symmetric increment distribution. Examples are symmetric $\alpha$-stable processes, such as the Cauchy process. 
\end{rem}

\medskip

\subsubsection{\textbf{The a.s.~convergence to the quadratic variation}}\hfill\\

\noindent The finiteness of the quadratic variation and hence infinite path length 
\cite[Ch.1,Sec.9]{Lev54} and on a much deeper level its linear characterization 
\cite{Du10, MV11, Pr04} are principal features of Brownian motion. In the sequel we give deviation frequency quantifications of \cite[9.3 Theorem]{SPB14}, which goes back to \cite{Lev48}, and an improved version given in \cite[9.4 Theorem]{SPB14}, which uses the exponential integrability of the Gaussian increments and which goes back to \cite{Du73}.

We adopt the notation of \cite[Section 9.2]{SPB14}. For any $t>0$ consider a squence of finite partitions on $[0, t]$ as $(\Pi_n(t))_{n\in \NN}$ as follows: 
there is a squence $(k_n)_{n\in \NN}$, $k_n\in \NN$ and $k_n\nearrow \infty$ as $n\ra\infty$,   
\[
\Pi_n(t) := \{(t_0, \dots, t_n)~\vert~t_0=0 < t_1 < \dots < t_{k_n} = t\}
\]
and $|\Pi_n(t)| := \sup_{i=1, \dots, k_n} (t_{i}-t_{i-1})$. 

\begin{thm}\label{thm:QV1}
Given a scalar Brownian motion $(W_t)_{t\gqq 0}$ and, for any $t>0$, 
a sequence $(\Pi_n(t))_{n\in \NN}$ of finite partitions of $[0, t]$.  
If $\sum_{n=1}^\infty |\Pi_n(t)| <\infty$ then  
\begin{align*}
\mbox{Var}_2(W; t) = \lim_{n\ra\infty } \sum_{t_i \in \Pi_n(t)} (W_{t_i}-W_{t_{i-1}})^2  = t\qquad \PP\mbox{-a.s.}
\end{align*}
In particular, we have the following \textbf{deviation frequency and last deviation quantifications}: 
\begin{enumerate}
\item  For any $t>0$, any positive sequence $(a_n)_{n\in \NN}$ such that 
\[
K_1(t):= \sum_{n=1}^\infty a_n \sum_{m=n}^\infty |\Pi_m(t)|  < \infty  
\]
and any $\e>0$ we have for $\sS(N) = \sum_{n=1}^{N} a_n$ with $\sS(0) = 0$ that
\begin{align*}
\EE[\sS(\oO_\e(t))] \lqq \EE[\sS(\m_\e(t))]\lqq \frac{2t}{\e^2} \sum_{n=1}^\infty a_n \sum_{m=n}^\infty |\Pi_m(t)|\\
\qquad \mbox{ and }\qquad \PP(\oO_\e(t)\gqq k) \lqq \PP(\m_\e(t)\gqq k) \lqq\frac{K_1(t)}{\sS(k)}, \qquad k\in \NN, 
\end{align*}
where 
\[
\m_\e(t)(\omega) := \sup\Big\{n\in\NN~\Big|~\sum_{i=0}^{k_n-1} (W_{t_{i+1}}(\omega)-W_{t_i}(\omega))^2- t\Big|>\e\Big\}
\] 
and 
\[\oO_\e(t)(\omega) :=\#\Big\{n\in \NN~\Big|\sum_{i=0}^{k_n-1} (W_{t_{i+1}}(\omega)-W_{t_i}(\omega))^2- t\Big|>\e\Big\}.\] 
\item For any $\theta>1$ and for $\epsilon =  (\e_n)_{n\in \NN}$ with 
$\e_n = \sqrt{2t n^\theta |\Pi_n(t)|}$  decreasing, we get 
for all $k\gqq 1$, 
\[
\PP(\oO_{\epsilon}(t) \gqq k) 
\lqq k^{-1} \cdot \zeta(\theta),  
\]
where 
$\oO_\epsilon(t)(\omega) := \#\Big\{n\in \NN~|~\Big|\sum_{i=0}^{k_n-1} (W_{t_{i+1}}(\omega)-W_{t_i}(\omega))^2- t\Big|>\e_n\Big\}$. 

Further, for
$\m_\epsilon(t)(\omega) := \sup\Big\{n\in\NN~|~ \Big|\sum_{i=0}^{k_n-1} (W_{t_{i+1}}(\omega)-W_{t_i}(\omega))^2- t\Big|>\e_n\Big\}$ and $-1<p<\theta-2$, we have for all $k\geq 1$,
\begin{align}\label{e:QVbarely}
\PP(\oO_{\epsilon}(t) \gqq k) \lqq \PP(\m_{\epsilon}(t) \gqq k) \leq k^{-(p+1)}\Big(\theta\zeta(\theta-p-1)+\ind(\{p<0\})\Big).
\end{align}
\end{enumerate}
\end{thm}

\noindent The proof is given in Appendix~\ref{ss:QV}. Note that intermediate regimes for some $\delta_n < \e_n$ can be achieved via \cite[Lemma~2]{EHS26}.

\begin{exm}\label{ex:QVdyadic}
For $t>0$, $k_n = 2^{n}$ and equidistant $t_i = it 2^{-n}$ we consider for each $\e>0$ the family of events 
\[
A_n := \bigg\{\Big|\sum_{i=0}^{k_n-1} (W_{t_{i+1}}-W_{t_i})^2- t\Big|>\e\bigg\}, \qquad n\in \NN.
\] 
Hence for $\m_\e(\omega) =:  \sup\Big\{n\in\NN~\vert~\Big|\sum_{i=0}^{2^n-1}(W_{t_{i+1}}(\omega)-W_{t_i}(\omega))^2- t\Big|>\e\Big\}$ and
\[\oO_\e(\omega) := \#\Big\{n\in \NN~\vert~\Big|\sum_{i=0}^{2^n-1}(W_{t_{i+1}}(\omega)-W_{t_i}(\omega))^2- t\Big|>\e\Big\}\]
 and by Chebyshev's inequality we have 
\begin{align*}
\PP\bigg( \Big|\sum_{i=0}^{k_n-1} (W_{t_{i+1}}-W_{t_i})^2- t\Big|>\e\bigg) \lqq\frac{2t^2}{\e^2}\cdot 2^{-n},
\end{align*}
and for all $0 \lqq p < \ln(2)$ we have by Lemma~\ref{lem:Alexp} 
\begin{align*}
\EE[e^{p {\oO_\e}}]\lqq \EE[e^{p {\m_\e}}] 
\lqq 1+\frac{2t^2}{\e^2} \frac{1}{1- e^{p}/2},
\end{align*}
and with the help of Lemma~\ref{lem:optimal} for $M = \max\{\frac{2t^2}{\e^2},1\}$ and $b = \frac{1}{2}$ in Appendix~\ref{a:optimal} we obtain 
\begin{align*}
\PP\big(\oO_\e\gqq k\big)\lqq\PP\big(\m_\e\gqq k\big)
\lqq e^{\frac{5}{4}} \cdot (2k+1)\big(\max\big\{\tfrac{2t^2}{\e^2},1\big\}+1\big) \cdot 2^{-k}
\end{align*}
for all $k\gqq  1$. 

On the other hand, replacing the constant sequence $\e = (\e)_{n\in \NN}$ by $\epsilon = (\e_n)_{n\in \NN}$, 
$\e_n := t \sqrt{2  n^\theta 2^{-n}}$, for $\theta>2$, we have with the help of the usual first Borel-Cantelli lemma for a.a.~$\om$ and $n\gqq \m_\theta(\omega)$ the estimate 
\[
 \Big|\sum_{i=0}^{k_n-1} (W_{t_{i+1}}(\om)-W_{t_i}(\om))^2- t\Big| \lqq  \e_n,\]
however, we have, by Example \ref{ex:Riemann}, with the low order mean deviation maximum $-1 < p < \theta-2$, 
\[
\PP\big(\oO_\theta \gqq k\big)\lqq \PP\big(\m_\theta \gqq k\big)
\lqq k^{-(1+p)}\big(\theta\zeta(\theta-p-1)+\ind(\{p<0\})\big), \qquad k\gqq 1, 
\]
where $\m_{\theta}(\omega) := \max\{n\in\NN\in \NN_0~|~\omega\in A_n\}$ and $\oO_{\theta}:= \#\{n\in \NN~|~\omega\in A_n\}$. 
For $\theta\in (1,2]$ we can always apply the usual Borel-Cantelli lemma, which yields an upper bound only for $\EE[\oO_\theta]$ and a linear decay of $\PP(\oO_\theta\gqq k)$. For a similar reasoning see Appendix~\ref{a:fsoptimaleRate}.
\end{exm}


\begin{thm}\label{thm:QV2} 
Given a scalar Brownian motion $(W_t)_{t\gqq 0}$ and for any $t>0$ 
a sequence $(\Pi_n(t))_{n\in \NN}$ of finite partitions of $[0, t]$.  
If $|\Pi_n(t)| = o(\frac{1}{\ln(n)})_{n\ra\infty}$, then  
\begin{align*}
\mbox{Var}_2(B; t) = \lim_{n\ra\infty } \sum_{t_i \in \Pi_n(t)} (W_{t_i}-W_{t_{i-1}})^2  = t\qquad \PP\mbox{-a.s.}
\end{align*}
In particular, we have the following \textbf{deviation frequency and last deviation quantification}. 
For any $t>0$, $\e>0$ and any $\la \in (0, \frac{1}{2})$ 
and any positive sequence $(a_n)_{n\in \NN}$ such that 
\begin{align*}
K_2(t, \e, \la):= \sum_{n=1}^\infty a_n \sum_{m=n}^\infty 2\exp\Big(-\frac{\e \la}{2 |\Pi_n(t)|}\Big) < \infty. 
\end{align*}
we have for $\sS(N) = \sum_{n=1}^N a_n$ with $\sS(0) = 0$ that $\EE[\sS(\oO_\e(t))]\lqq \EE[\sS(\m_\e(t))] \lqq K_2(t, \e, \la)$ and 
\begin{align*}
\PP(\oO_\e(t)\gqq k) \lqq \PP(\m_\e(t)\gqq k)\lqq \frac{K_2(t, \e, \la)}{\sS(k)}, 
\end{align*}
where 
\begin{align}\label{d:overlap}
\m_\e(t) :&= \sup\bigg\{n\in\NN~|~ \Big|\sum_{i=0}^{k_n-1} (W_{t_{i+1}}-W_{t_i})^2- t\Big|>\e\bigg\}  \mbox{ and }\nonumber\\ \oO_\e(t) :&= \#\bigg\{n\in \NN~|~\Big|\sum_{i=0}^{k_n-1} (W_{t_{i+1}}-W_{t_i})^2- t\Big|>\e\bigg\}. 
\end{align}
\end{thm}

\noindent The proof is given in Appendix~\ref{ss:QV}. 

\begin{exm}(Example~\ref{ex:QVdyadic} improved)\label{ex:QVdyadic2}
For the equidistant partition $t_i = it2^{-n}$ with $|\Pi_n(t)| = t\,2^{-n}$ we consider for each $\e>0$ the family of events 
\[\bigg\{\Big|\sum_{i=0}^{k_n-1} (W_{t_{i+1}}-W_{t_i})^2- t\Big|>\e\bigg\}, \qquad n\in \NN.\] 
Hence we can strenghten the result in Theorem~\ref{thm:QV2} in that 
\[
\EE[\sS(\oO_\e)] \lqq \EE[\sS(\m_\e)] \lqq \sum_{n=1}^\infty a_n \sum_{m=n}^\infty 
\,\PP\bigg( \Big|\sum_{i=0}^{k_m-1} (W_{t_{i+1}}-W_{t_i})^2- t\Big|>\e\bigg). 
\]
The upper bound given in \cite[p.~142]{SPB14} then yields a Gumbel type decay. That is, for any fixed $\la \in (0,2)$, $\e>0$ and $n\in \NN$, we have 
\begin{align*}
\PP\bigg( \Big|\sum_{t_i\in \Pi_n(t)} (W_{t_{i+1}}-W_{t_i})^2- t\Big|>\e\bigg) 
&\lqq 2 \exp{\Big(- \tfrac{\e \la}{2 |\Pi_n(t)|}\Big)}\lqq 2 \exp\Big(- \tfrac{\e \la}{2t}\cdot 2^{n}\Big).
\end{align*}
Hence the choice $a_n = \exp(\frac{\e \ti \la}{2t} 2^{n})$ for some $0 < \ti \la < \la$ yields 
\begin{align*}
\sS(N) := \sum_{n=1}^N a_n = \sum_{n=1}^N \exp{\Big(\tfrac{\e \ti \la}{2t}\cdot 2^{n}\Big)}\gqq \exp\Big(\tfrac{\e \ti \la}{2t}\cdot 2^{N}\Big).
\end{align*}
With the same notation for $\oO_\e(t)$ and $\m_\e(t)$ of \eqref{d:overlap}, \cite[Lemma 1]{EHS26} yields 
\begin{align*}
\EE\Big[\exp\Big(\tfrac{\e \ti \la}{2t} 2^{\oO_\e(t)}\Big)\Big]
&\lqq \EE\Big[\exp\Big(\tfrac{\e \ti \la}{2t} 2^{\m_\e(t)}\Big)\Big]\\
&\lqq \sum_{n=1}^\infty \exp{\Big(\tfrac{\e \ti \la}{2t} \cdot2^{n}\Big)} \sum_{m=n}^\infty 2 \exp\Big(- \tfrac{\e \la}{2t}\cdot 2^{m}\Big)\\
&\lqq  2 \sum_{n=1}^\infty \exp{\Big(- \tfrac{\e (\la-\ti \la)}{2t}\cdot 2^{n}\Big)} 
\sum_{m=n}^\infty \exp\Big(- \big(\tfrac{\e \la}{2t}\cdot 2^{m}-\tfrac{\e \la}{2t}\cdot 2^{n}\big)\Big)
=: K_3(t, \e,\lambda,\tilde{\lambda})<\infty.
\end{align*}
Markov's inequality then yields the much better Gumbel type decay for the mean deviation maximum and the mean deviation frequency tail 
\begin{align*}
\PP(\oO_\e(t)\gqq k)\lqq \PP(\m_\e(t)\gqq k)\lqq  K_3(t, \e,\lambda,\tilde{\lambda}) \exp\Big(- \tfrac{\e \ti \la}{2t} \cdot 2^{k}\Big), \qquad k\in \NN. 
\end{align*}
On the other hand, we obtain for the much smaller scale $\delta_n = \frac{2t\theta}{\la} \ln\big(2^\frac{1}{\theta} n\big)\cdot 2^{-n}\ll_n \e_n$ ($=t \sqrt{2 n^\theta 2^{-n}}$ from Example \ref{ex:QVdyadic}), for any $\theta>1$ with the help of the usual Borel-Cantelli lemma,
\[
\limsup_{n\ra\infty}  \Big|\sum_{i=0}^{k_n-1} (W_{t_{i+1}}-W_{t_i})^2- t\Big| \cdot \delta_n^{-1} \lqq 1, \qquad \PP\mbox{-a.s.},
\]
and the merely linear deviation frequency decay 
\[
\PP\big(\oO_\e(t)\gqq k\big)
\lqq k^{-1}\cdot \zeta(\theta), \qquad k\gqq 1. 
\]
Further obvious extensions and tradeoffs following \cite{EHS26} are possible but not written out here in detail.
\end{exm}

\begin{rem}\label{rem:QVLevy}
Similar results are straightforward to implement for other classes of processes, such as for L\'evy martingales 
\begin{align*}
L_t := \int_0^t \int_{\RR\setminus \{0\}} z \tilde{\mathrm{N}}(ds,dz), \qquad 
\tilde{\mathrm{N}}([0, t]\times A) =  \mathrm{N}([0, t]\times A)- t \nu(A), \qquad A\in \bB(\RR^d), 0 \notin \bar A,  
\end{align*}
where $\mathrm{N}$ is a Poisson random measure on $[0, \infty)\times \RR^d$ with respect to the intensity measure $ds\otimes \nu$, where $\nu$ is a sigma finite measure on $(\RR^d, \bB(\RR^d))$ with the integrability conditions 
\begin{align*}
\int_{\RR^d} (\min\{1, |z|^2\}) \nu(dz) <\infty\qquad \mbox{ and }\qquad \nu(\{0\}) = 0.  
\end{align*}
\end{rem}

\bigskip

\subsection{\textbf{Deviation frequencies in the laws of the iterated  logarithm and related}}\label{ss:mainlil}
\subsubsection{\textbf{Khinchin's law of the iterated logarithm }}\label{expo:LIL}\hfill\\

We follow the exposition in \cite{KS98}, see also \cite[11.1 Theorem and 11.2 Corollary]{SPB14}

\noindent Consider a real valued standard Brownian motion $(W_t)_{t\in [0, 1]}$. 
 \begin{thm}\label{thm: lil}
For $\PP$-a.a.~$\om \in \Omega$ we have 
\begin{align}\label{e:lilobben}
\limsup_{t\ra 0+} \frac{W_t}{\sqrt{2 t \ln(\ln(1/t))}} \lqq 1.  
\end{align}
Furthermore, we have the following \textbf{deviation frequency and last deviation quantification}: For all $\delta >1, \theta \in (0,1)$ and $p\in (-1,\delta-1)$ we have for 
\[
\m_{\delta, \theta}:=  \sup\bigg\{n\in \NN~\vert~\sup_{\theta^{n+1} < s\lqq \theta^n} \frac{W_s}{\sqrt{2 s \ln(\ln(1/s))}}
 >  \Big(1+\frac{\delta}{2}\Big) \theta^{\frac{1}{2}}\bigg\}\]
 and
 \[ \oO_{\delta, \theta} := \#\bigg\{n\in \NN~\vert~\sup_{\theta^{n+1} < s\lqq \theta^n} \frac{W_s}{\sqrt{2 s \ln(\ln(1/s))}}
 >  \Big(1+\frac{\delta}{2}\Big) \theta^{\frac{1}{2}}\bigg\}\]
and for a.a.~$\omega$ that for $n > \m_{\delta, \theta}(\omega),$ 
\[
 \sup_{\theta^{n+1} < s\lqq \theta^n} \frac{W_s(\omega)}{\sqrt{2 s \ln(\ln(1/s))}}
 \lqq  \Big(1+\frac{\delta}{2}\Big) \theta^{\frac{1}{2}}
\]
and 
\begin{align}
\PP\big( \oO_{\delta, \theta} \gqq k\big)\lqq \PP\big( \m_{\delta, \theta} \gqq k\big) \lqq 
 \min\Big\{\frac{\zeta(1+\delta)}{\ln(1/\theta)^{1+\delta}k},\frac{(1+\delta) \zeta(\delta-p)}{\ln(1/\theta)^{1+\delta}k^{p+1}}+\frac{\ind(\{p<0\})}{k^{p+1}}\Big\},\label{e:lilquant}
\end{align} 
for all $k\gqq 1$ where $\zeta(s) = \sum_{n=1}^\infty n^{-s}$ is Riemann's zeta function, and
\begin{align*}&\PP\big( \oO_{\delta, \theta} \gqq k\big)\lqq \PP\big( \m_{\delta, \theta} \gqq k\big) \\
&\quad\lqq 
 \frac{1+\delta}{\ln(1/\theta)^{1+\delta}e^{\gamma\delta}}\cdot\frac{\ln(k)+1+\gamma}{k^{\delta}}+\frac{\ind(\{k<e^{\frac{1}{\delta-1}-\gamma}\}\cup\{\delta<1\})}{k^\delta},\quad k\gqq e^{\frac{1}{\delta}-\gamma},
\end{align*}
where $\gamma\approx0.5772$ is the Euler-Mascheroni constant. 
In the case of $\delta\in (0,1]$, we also have the following result for the deviation frequency only,
\begin{align*}
\PP\big( \oO_{\delta, \theta} \gqq k\big)  \lqq 
 \frac{1}{\ln(1/\theta)^{1+\delta}} \frac{\zeta(1+\delta)}{k}.
\end{align*} 
 \end{thm}
 \noindent The proof is given in Appendix~\ref{ss:LIL}. 
 
 \begin{exm}\label{ex: lil}
Under the assumptions of Theorem~\ref{thm: lil}
we have for all $p\in (-1,1)$ and $k\in\NN,$ 
\begin{align}\label{e:lilquant0}
\PP\big(\m_{2,\frac{1}{3}} \gqq k)
\lqq \min\Big\{\frac{\zeta(3)}{{\ln(3)^3}\cdot k}, \frac{3 \zeta(2-p)}{{\ln(3)^3}\cdot k^{1+p}}+\frac{\ind(\{p<0\})}{k^{1+p}}\Big\} 
\end{align}
and
\begin{align*}
\PP\big(\m_{2,\frac{1}{3}} \gqq k)\lqq \frac{3}{{\ln(3)^3}e^{2\gamma}}\frac{\ln(k)+1-\gamma}{k^2}+\ind(\{k=1\})\quad\text{for}\quad k\gqq \lceil e^{\frac{1}{2}-\gamma}\rceil =1.
\end{align*} 
\end{exm}

\begin{rem}
\begin{enumerate}
 \item In comparison to the upper bound of L\'evy's modulus of continuity, for instance, this result is rather weak, with a merely low order polynomial decay of the upcrossing frequencies even for exponentially small times.  
 \item Theorem~\ref{thm: lil} is consistent with an independent quantification 
of the law of the iterated logarithm along a diverging sequence studied in \cite[Subsection 3.2.1]{EH22}. 
\item Similar results can be obtained for $\alpha$-stable processes, see \cite[Chapter 8, Section 2]{Be98}. 
\item It is possible to adapt the results to other sequences of interest $t_n\searrow 0$, instead of $\theta^n$ at the price of a higher technical effort. 
\end{enumerate}
\end{rem}

\medskip

\subsubsection{\textbf{Chung's ``other'' law of the iterated logarithm }}\label{expo:Chung}\hfill\\

\noindent The result goes back to \cite{Ch48}. We follow the exposition in \cite{SPB14}. 

\begin{thm}\label{thm:Chung}
Let $(W_t)_{t\gqq 0}$ be a scalar Brownian motion. 
Then 
\begin{align*}
\liminf\limits_{t\ra\infty} \frac{\sup_{s\in [0, t]} |W_s| }{\sqrt{\frac{t}{\ln(\ln(t))}}} = \frac{\pi}{\sqrt{8}} \qquad \PP\mbox{-a.s.}
\end{align*}
In particular, we have the following frequency \textbf{deviation and last deviation quantification}: 
For any $q>1$, $\e>0$, $-1<p< \frac{1}{(1-\e)^2}-2$ and 
\begin{align*}
&\m_{q, \e} := \sup\Big\{n-n_q+1~\vert~\frac{\sup_{s\in [0, q^n]} |W_s|}{\sqrt{\frac{q^n}{\ln(\ln(q^n))}}} < (1-\e) \frac{\pi}{\sqrt{8}} \Big\},\\
&\text{where }n_q:=\inf\{n\in \NN~|~\ln(\ln(q))>0\},\text{ and }\\
&\oO_{q, \e}:= \# \Big\{n\gqq n_q~\vert~\frac{\sup_{s\in [0, q^n]} |W_s|}{\sqrt{\frac{q^n}{\ln(\ln(q^n))}}} < (1-\e) \frac{\pi}{\sqrt{8}} \Big\},
\end{align*}
we have for a.a.~$\omega$ that $n> \m_{q, \e}(\omega)$ implies 
\[
\frac{\sup_{s\in [0, q^n]} |W_s(\omega)|}{\sqrt{\frac{q^n}{\ln(\ln(q^n))}}} \gqq (1-\e) \frac{\pi}{\sqrt{8}} 
\]
and 
\begin{align*}
&\PP\big(\oO_{q, \e}\gqq k\big)\lqq \PP\big(\m_{q, \e}\gqq k\big)\\
&\ \lqq \max\Big\{\frac{4}{\pi \ln(q)^{\frac{1}{(1-\e)^2}}}\!\cdot\!\frac{\zeta\big(\frac{1}{(1-\e)^2}-p-1;n_q\big)}{(1-\e)^2 k^{p+1}}\!+\!\frac{\ind(\{p<0\})}{k^{p+1}}, \frac{4}{\pi \ln(q)^{\frac{1}{(1-\e)^2}}}\!\cdot\!\frac{\zeta\big(\frac{1}{(1-\e)^2};n_q\big)}{k} \! \Big\},
 \end{align*}
 where $\zeta(s) = \sum_{n=n_q}^\infty n^{-s}$ is the Hurwitz zeta function  (with $\zeta(s)=\infty$ for $s\lqq 1$). 
 Further, for $q>e$, $k\gqq e^{\frac{(1-\e)^2}{1-(1-\e)^2}-\gamma}$, we have (since $n_q=1$ then),
 \begin{align*}
 \PP\big(\oO_{q, \e}\gqq k\big)\lqq \PP\big(\m_{q, \e}\gqq k\big)\lqq &\frac{4}{\pi\ln^{\frac{1}{(1-\e)^2}}(q)(1-\e)^2e^{\gamma(\frac{1}{(1-\e)^2}-1)}} \frac{\ln(k)+1-\gamma}{k^{\frac{1}{(1-\e)^2}-1}}\\
 &+\frac{\ind\Big(\{k < e^{\frac{(1-\e)^2}{1-2(1-\e)^2}-\gamma}\}\cup\{\frac{1}{(1-\varepsilon)^2}<2\}\Big)}{k^{\frac{1}{(1-\e)^2}-1}},
 \end{align*}
 where $\gamma\approx 0.5772$ is the Euler-Mascheroni constant.
\end{thm}
\noindent The proof is given in Appendix \ref{ss:theotherLIL}. 

\begin{rem}
Similarly to Khinchin's law of the iterated logarithm, we have that the deviation frequency decays rather weakly.  
\end{rem}

 \begin{exm}\label{ex: theotherLIL}
Under the assumptions of Theorem~\ref{thm:Chung} we have for all $k\in \NN$
\begin{align}\label{e:theotherlilquant}
\PP\big(\oO_{4, \frac{1}{2}} \gqq k\big)\lqq \PP\big(\m_{4, \frac{1}{2}} \gqq k\big)\lqq
\min\Big\{\frac{1 }{\pi \ln^4(2)}\cdot\frac{\zeta(3-p)}{k^{1+p}}+\frac{\ind(\{p<0\})}{k^{1+p}}, \frac{1}{\pi \ln^4(2)}\cdot\frac{\zeta(4)}{k}\Big\}, 
\end{align}
for $-1<p<2$, and
\begin{align*}
\PP\big(\oO_{4, \frac{1}{2}} \gqq k\big)\lqq \PP\big(\m_{4, \frac{1}{2}} \gqq k\big)\lqq \frac{1}{\pi\ln^4(2)e^{3\gamma}}\frac{\ln(k)+1-\gamma}{k^4}
\end{align*}
for $k\gqq \lceil e^{\frac{1}{2}-\gamma} \rceil=1$.
\end{exm}

\medskip

\subsubsection{\textbf{Quantifying the Kolmogorov test}}\label{expo:Kolmogorov}\hfill\\

\noindent While the preceding law of the iterated logarithm presents the precise asymptotics of Brownian motion, Kolmogorov's test \cite[p.~34]{IMcK74} yields a coarser measure of the asymptotics in $0$, resulting in a $0$-$1$ law, which can now be quantified by its deviation frequencies. 

\begin{thm}\label{thm: Koltest}
Consider a real valued standard Brownian motion $(W_t)_{t\in [0, 1]}$ and a function $h:[0, \infty)  \ra [0, \infty)$ such that  
\begin{align}\label{e:monotonie}
t\mapsto h(t) \mbox{ is increasing,} \qquad  \mbox{ and }\qquad  t\mapsto h(t) / \sqrt{t} \mbox{ is decreasing. } 
\end{align}
Then the finiteness 
\begin{align*}
\lim\limits_{s\ra 0+} \phi(s) < \infty\qquad \mbox{ for }\qquad \phi(s) := \int_{0+}^s \frac{h(t)}{t^{3/2}} \cdot e^{-\frac{h(t)^2}{2 t}} dt, \quad s>0  
\end{align*}
implies 
\begin{align*}
\PP\Big(\lim\limits_{t\ra 0+} \frac{W_t}{h(t)} \lqq 1\Big) = 1.
\end{align*}
In particular, we have the following \textbf{deviation frequency and last deviation quantification}: 
For any positive $(b_n)_{n\in \NN}$ with $b_n \searrow 0$ as $n\ra\infty$ and 
$\oO:= \#\{n\in \NN~|~\sup_{t\in (0, b_n)} \frac{W_t}{h(t)}> 1\}$ and $\m := \sup\{n\in\NN~|~ \sup_{t\in (0, b_n)} \frac{W_t}{h(t)}> 1\}$
we have that $\oO = \m$, a.s.~and  
for all positive sequences $(a_k)_{k\in \NN}$ and $\sS_a(N) := \sum_{n=1}^N a_n$ with $\sS_a(0) = 0$ 
\begin{align*}
\EE[\sS_a(\oO)] \lqq \sum_{n=1}^\infty a_n \phi(b_n), \qquad \mbox{ whenever the right-hand side is finite} 
\end{align*}
and in this case for all $k\gqq 1$ 
\begin{align*}
\PP\big(\oO\gqq k\big)\lqq \sS_a^{-1}(k)\cdot \sum_{n=1}^\infty a_n \phi(b_n).
\end{align*}
\end{thm}
\noindent The proof is a straightforward combination of \cite[inequality 5), p.~34]{IMcK74}, the observation that the sequence of events $\{\sup_{t\in (0, b_n)} \frac{W_t}{h(t)}\gqq 1\}$ is nested as a function of $n$ and \cite[Proposition 1]{EH22}. 

\begin{rem}
 The application of the usual Borel-Cantelli lemma does not allow to distinguish the different behaviors between the functions given by $h(t)=\sqrt{2(1+\al) t \ln(\ln(1/t))}$ for different $\al>0$, since both hold almost surely. Theorem~\ref{thm: Koltest}, however, allows to distinguish in terms of different deviation frequencies, as can be seen in the following example. 
\end{rem}

\begin{exm}\label{ex:KolmoTest}
The function $h(t) = \sqrt{2(1+\al) t \ln(\ln(1/t))} > \sqrt{2 t \ln(\ln(1/t))}$ for all $t\in (0, t_0)$, $t_0>0$ small enough and $\alpha>1$, obviously satisfies \eqref{e:monotonie}. 
We calculate 
\begin{align*}
\phi(s) 
&= \int_{0+}^s \frac{h(t)}{t^{3/2}} e^{-\frac{h(t)^2}{2}} dt= 
\int_{0+}^s \frac{\sqrt{2(1+\al)\ln(\ln(1/t))}}{t} e^{-(1+\al)\ln(\ln(1/t))}dt\\
&=\sqrt{2(1+\al)}\int_{\ln(\ln(1/s))}^\infty\sqrt{u}e^{-\al u} du=\frac{\sqrt{2(1+\al)}}{\al^{3/2}}\Gamma\Big(\frac{3}{2},\al\ln(\ln(1/s))\Big), 
\end{align*}
which is smaller than $\frac{\sqrt{2(1+\al)}}{\al^{3/2}}\Big(1+\frac{1}{2\al\ln(\ln(1/s))}\Big)\frac{\sqrt{\al\ln(\ln(1/s))}}{\ln(1/s)^\al}$ by Corollary \ref{cor:Gamma} in Appendix \ref{a:Gamma}.
Hence for $b_n = 4^{-n}$ (as for comparison with Example~\ref{ex: lil}), inserting yields
\begin{align*}
\varphi(4^{-n})=\frac{\sqrt{2(1+\al)}}{\al^{3/2}}\Big(1+\frac{1}{2\al\ln(\ln(4)n)}\Big)\frac{\sqrt{\al\ln(\ln(4)n)}}{\ln(4)^\al n^\al}
\end{align*}
we have for all $-1<p< \alpha-1$ that 
\begin{align*}
\sum_{n=1}^\infty n^p\varphi(4^{-n})\lqq \sum_{n=1}^\infty \frac{\sqrt{2(1+\al)}}{\ln(4)^\al\al}\Big(1+\frac{1}{2\al\ln(\ln(4))}\Big)\frac{\sqrt{\ln(\ln(4)n)}}{ n^{\al-p}}<\infty. 
\end{align*}
So also for $\sS(N) = \sum_{n=1}^N n^p$, we get for $\m := \sup\Big\{n\in\NN~\vert~ \sup_{t\in (0, b_n)} \frac{W_t}{ \sqrt{2(1+\al) t \ln(\ln(1/t))}}>1\Big\}$ the finiteness of the quantities $\EE[\sS(\m)]=\EE[\sS(\oO)]<\infty$. Following the reasoning in Example \ref{ex:Riemann}, we obtain $\EE[\m^{1+p}]=\EE[\oO^{1+p}]<\infty.$
\end{exm}

\bigskip 
\subsubsection{\textbf{Quantifying Strassen's functional law of the iterated logarithm}} 

\noindent While the laws of the iterated logarithm in Subsection \ref{expo:LIL} - \ref{expo:Kolmogorov} are formulated for the marginals $t\mapsto W_t$, the following functional version of the law of the iterated logarithm treats all continuous functions (starting in $0$) $[0, 1]\ni s\mapsto (t\mapsto W_{s\cdot t})$ simultaneously. 
We follow the exposition \cite[Section 12.1 and 12.13]{SPB14}, while the original work goes back to \cite{Str64}. 

\begin{thm}(Strassen)\label{thm:Strassen}
Let $(W_t)_{t\gqq 0}$ be a scalar Brownian motion and 
\begin{align*}
Z_s(t, \omega):= \frac{W_{s\cdot t}(\omega)}{\sqrt{2s \ln(\ln(s))}}, \qquad t\in [0, 1].  
\end{align*}
\begin{enumerate}
\item Then for almost all $\omega \in \Omega$ we have that 
\begin{align*}
\{Z_s(\cdot, \omega)~\vert~s>e\} 
\end{align*}
is relatively compact in the Banach space $(\cC_0[0,1], \|\cdot \|_\infty)$ of continuous functions $f:[0, 1]\ra\RR$ with $f(0) = 0$ equipped with the supremum norm $\|\cdot \|_\infty$ and the set of almost sure limit points is given by $\kK(\frac{1}{2})$, where 
\begin{align*}
\kK(r) = \bigg\{w\in \cC_0[0,1]~\vert~ w\mbox{ is absolutely continuous and } \frac{1}{2} \int_0^1 |w'(s)|^2 ds \lqq r\bigg\}.
\end{align*}
Furthermore, we have the following \textbf{deviation frequency quantification} of the almost sure convergence to $\kK$. Define 
\[
\m_{q, \eta, \eps} := \max\{n\in\NN~\vert~ d(Z_{q^n}(\cdot, \cdot),\kK(\tfrac{1}{2}+\eta))>\eps\}\]
and
\[ \oO_{q, \eta, \eps} := \#\{n\in \NN~\vert~d(Z_{q^n}(\cdot, \cdot),\kK(\tfrac{1}{2}+\eta))>\eps\}. 
\]

\item \textbf{Arbitrary ``energy excess'': }Fix $\eta>0$. 
For any $\eps>0$, $q>1$ and $0 < \vartheta < \eta$ there is a positive constant $a = a(\eta, \vartheta, q, \e)$ such that 
for all $k\in \NN$ 

\begin{equation}\label{e:Strassen0}
\PP(\oO_{q, \eta, \eps} \gqq k) \lqq k^{-1}\cdot \frac{a\,\zeta(1+2\vartheta)}{\ln(q)^{1+2\vartheta}},  \qquad k\gqq 1,   
\end{equation}

\item \textbf{Large ``energy excess'': }Fix $\eta>\tfrac{1}{2}$. 
\begin{enumerate}
 \item For any $\eps>0$, $q>1$ and $\tfrac{1}{2} < \vartheta < \eta$ we have 
for all $p>-1$ which satisfy $p < 2 \vartheta-1$ that for a.a.~$\om$ and $n>\m_{q, \eta, \eps}(\om)$ it follows 
\[
d(Z_{q^n}(\cdot, \cdot),\kK(\tfrac{1}{2}+\eta)) \lqq \eps  
\]
and for all $k\in \NN$ 
\begin{align}
&\PP(\oO_{q, \eta, \eps} \gqq k) \lqq \PP(\m_{q, \eta, \eps} \gqq k) \nonumber\\
&\quad\lqq k^{-(1+p)} \cdot \frac{(p+1)(1+2\vartheta)}{2\vartheta}\frac{b\,\zeta(2\vartheta-p)}{\ln(q)^{1+2\vartheta}}+k^{-(1+p)}\cdot\ind(\{p<0\})\label{e:Strassen}\\
&\quad\lqq k^{-(1+p)} \cdot (1+2\vartheta)\frac{b\,\zeta(2\vartheta-p)}{\ln(q)^{1+2\vartheta}}+k^{-(1+p)}\cdot\ind(\{p<0\})\qquad k\gqq 1,\nonumber  
\end{align}
where $d(w, A) = \inf_{v\in A}\|w-v\|_{\infty}$ and $b = b(\eta, \e) = 2e^2e^{\frac{4+8\eta}{\e^2}}.$

\item Additionally, optimizing \eqref{e:Strassen} in $p$ 
and then taking $\vartheta\to\eta$, we find in fact that
 \begin{align*}
 &\PP(\oO_{q, \eta, \eps} \gqq k) \lqq \PP(\m_{q, \eta, \eps} \gqq k)\\ 
 &\quad \lqq k^{-2\eta}\cdot \zeta\big(1+\tfrac{1}{\ln(k)+\gamma}\big)\cdot k^{\frac{1}{\ln(k)+\gamma}}  \cdot\frac{b}{\ln(q)^{1+2\eta}}\Big(2\eta -\frac{1}{\ln(k)+\gamma}\Big)\Big(1+\frac{1}{2\eta}\Big)\\
 &\qquad+k^{-2\eta}\ind\Big(\big\{k<e^{\frac{1}{2\eta-1}-\gamma}\big\}\Big),  
 \end{align*}
 for $k\gqq \max\Big\{\frac{1}{\ln(q)},e^{\frac{1}{2\eta}-\gamma}\Big\}$, where $\gamma\approx0.5772$ is the Euler-Mascheroni constant. The right hand side is asymptotically equal (as $k\to\infty$) to $k^{-2\eta}\cdot (2\gamma +\ln(k))\cdot \tfrac{e \cdot b(2\eta+1)}{\ln(q)^{1+2\eta}}$.
\item For any $\theta>0, \eps>0$, $q>1$ and $\frac{1}{2} < \vartheta < \eta$ for all $k\in \NN$ we have \\ 
\[
\limsup_{n\ra\infty} \,d\big(Z_{q^n}(\cdot, \cdot),\kK(\tfrac{1}{2}+\eta)\big) \cdot \e_n^{-1} \lqq 1 \qquad \PP\mbox{-a.s.},
\]
where 
\[
\e_n := \sqrt{\frac{4 +8\eta}{\ln(\tfrac{\ln(q)^{1+2\vartheta}}{2e^2})+ \ln(\tfrac{n^{2\vartheta}}{\ln(n+1)^{1+\theta}})}}
\]
which is of order $\ln\bigg(\frac{n}{\ln(n+1)^{\frac{1+\theta}{2\vartheta}}}\bigg)^{-\frac{1}{2}}$ and for $k\gqq 1$ we have 
\begin{equation}\label{e:Strassen01}
\PP(\#\{n\in \NN~\vert~d(Z_{q^n}(\cdot, \cdot),\kK(\tfrac{1}{2}+\eta))>\eps_n\}\gqq k) \lqq k^{-1}\cdot \sum_{n=1}^\infty \frac{1}{n\ln(n+1)^{1+\theta}}.  
\end{equation}
\end{enumerate}
 \end{enumerate}
\end{thm}
We note that this result (c) can be still strengthened in terms of $\m_{q, \eta, \epsilon}$ for the respective sequence $\epsilon$. The proof is found in Appendix~\ref{ss:Strassen}. \\

\bigskip 
\subsection{\textbf{The underlying quantitative version of the first Borel-Cantelli lemma}}
\subsubsection{\textbf{The Borel-Cantelli lemma
}}\label{ss:BCmain}\hfill\\ 

We recall the following special case of \cite[Lemma 1]{EHS26}:

\begin{lem}\label{lem:BC1}
On a given probability space $(\Omega, \aA, \PP)$ consider a sequence of events $(A_n)_{n\in \NN}$ 
such that $\sum_{n=1}^\infty \PP(A_n) <\infty$ and define 
\[
\oO := \sum_{n=1}^\infty \ind(A_n), \qquad \mbox{ and } \qquad \m:= \sum_{n=1}^\infty \ind\Big(\bigcup_{m\gqq n} A_m\Big).
\]
If, in addition, for some sequence of non-negative weights $a = (a_n)_{n\in \NN}$ we have 
\[
K_a := \sum_{n=1}^\infty a_n \sum_{m=n}^\infty \PP(A_m) < \infty, 
\]
then for the function $\sS_a(N) = \sum_{n=1}^N a_n$ with the convention $\sS_a(0) = 0$ we have 
\[
\EE[\sS_a(\oO)] \lqq \EE[\sS_a(\m)] \lqq K_a.  
\]
\end{lem}

\begin{exm}\label{ex:Riemann}
For $\PP(A_n)\lqq c n^{-q}$, $n\in \NN$, for some $c>0$ and $q>1$  we have for all $-1 < p < q-2$ that 
\[
\EE[\oO^{p+1}]\lqq \EE[\m^{p+1}]\lqq cq \zeta(q-p-1) 
\]
Choosing $a_n=n^p$, we see that for $p\gqq 0$,
\[
\sS_{a}(N) = \sum_{n=1}^N n^p \gqq \int_1^{N+1} (x-1)^p dx = \int_0^N x^p dx = \frac{N^{p+1}}{p+1},  
\]
and for $-1<p<0$,
\[
\sS_{a}(N) = \sum_{n=1}^N n^p \gqq \int_1^{N+1} x^p dx = \frac{(N+1)^{p+1}-1}{p+1}. 
\]
Further,
\begin{align*}
&\sum_{n=1}^\infty n^p \sum_{m=n}^\infty \frac{c}{m^q} \lqq c\sum_{n=1}^\infty n^{p-q} + c\sum_{n=1}^\infty n^p \int_{n}^\infty x^{-q} dx \\
&= c\zeta(q-p) + \frac{c}{q-1} \zeta(q-p-1) \lqq \frac{cq}{q-1} \zeta(q-p-1).
\end{align*}
Thus, for $0\leq p< q-2$ we have 
\[
\EE[\oO^{p+1}]\lqq \EE[\m^{p+1}]\lqq \frac{cq(p+1)}{q-1} \zeta(q-p-1)\lqq cq \zeta(q-p-1), 
\]
and for $-1<p<\min\{0,q-2\}$ we obtain
\[
\EE[\oO^{p+1}]\lqq \EE[\m^{p+1}]\lqq \frac{cq(p+1)}{q-1} \zeta(q-p-1)+1\lqq cq \zeta(q-p-1)+1, 
\]

This condition coincides with \cite[Example 1]{EH22}, except for the incorrect prefactor $\frac{c}{q-1}$ there, instead of $cq$ in front of the zeta function. The conditions for finiteness are identical. 
In addition, we have
\[
\PP(\oO\gqq k)\lqq \PP(\m\gqq k)\lqq \begin{cases}k^{-(p+1)}  cq \zeta(q-p-1),& 0\leq p\leq \max\{q-2,0\},\\
 k^{-(p+1)}  (cq \zeta(q-p-1)+1), & -1<p<\min\{q-2,0\}.\end{cases}
\]
\end{exm}

\begin{rem}
\noindent Note that for sums starting in $n=0$ we have that $\sS_a(N) = \sum_{n=0}^{N-1}a_n$ with the convention that $\sS_a(0) = 0$. 

\noindent This definition of $\sS_a(N)$ corrects an off-by-one error in the definition of $\sS_a(N)$ in Proposition~1 and Theorem~1 of \cite{EH22}. 
\end{rem}

\begin{lem}\label{lem:Alexp}
Let $b\in(0,1)$, $0\lqq p<-\ln(b)$ and $\sS(N)=\sum_{n=n_0}^{N+n_0-1}e^{pn}$ and $\PP (A_m)\lqq M b^m$ for some $M>0$ and all $m\gqq n_0$.
Then, 
$$\EE [e^{p\oO}]\lqq \EE [e^{p\m}]\lqq 1+Mb^{n_0-1}\frac{1}{1-e^pb}.$$
\end{lem}
\begin{proof}
Evaluating a geometric series, we obtain that 
$\sS(N)=e^{pn_0}\frac{e^{pN}-1}{e^p-1}$ and hence
\begin{align*}
e^{pN}= 1+\frac{e^{p}-1}{e^{pn_0}}\sS(N).
\end{align*}
Therefore, using Lemma~\ref{lem:BC1}, we obtain
\begin{align*}
\EE[e^{p\oO}]\lqq \EE[e^{p\m}]\lqq 1+\frac{e^{p}-1}{e^{pn_0}}\EE[\sS(N)]\lqq 1+\frac{e^{p}-1}{e^{pn_0}}\sum_{n=n_0}^\infty e^{pn}\sum_{m=n}^\infty Mb^m.
\end{align*}
We calculate the last geometric series, apply $e^p\lqq \frac{1}{b}$, which is true by assumption 
and continue with
\begin{align*}
&1+\frac{e^{p}-1}{e^{pn_0}}\sum_{n=n_0}^\infty e^{pn}\sum_{m=n}^\infty Mb^m =1+\frac{e^{p}-1}{e^{pn_0}}\frac{M}{1-b}\sum_{n=n_0}^\infty e^{(p+\ln(b))n}\\
&= 1+(e^{p}-1)\frac{Mb^{n_0}}{1-b}\frac{1}{1-e^pb} \lqq 1 + \frac{1-b}{b} \frac{Mb^{n_0}}{1-b}\frac{1}{1-e^pb}.
\end{align*}
Hence we end up with
$$\EE[e^{pO}]\lqq \EE[e^{p\m}]\lqq 1+\frac{Mb^{n_0-1}}{1-e^pb}.$$
\end{proof}

\begin{cor}\label{cor:Alexp}
With the above assumptions and $M\gqq 1$ we have 
\begin{equation}\label{e:expdecay}
\PP(\oO\gqq k)\lqq \PP(\m\gqq k)\lqq  e^{\frac{5}{4}}\cdot (2k+1)(Mb^{n_0-1}+1)\cdot b^{k}, \qquad k\gqq 1.
\end{equation}
\end{cor}
\begin{proof}
By Markov's inequality and the above theorem we get for all $0\lqq p<-\ln(b)$,
\begin{align*}
\PP(\oO\gqq k)\lqq \PP(\m\gqq k)\lqq e^{-pk}\EE[e^{p\oO}]\lqq e^{-pk}\bigg(1+\frac{Mb^{n_0-1}}{1-e^pb}\bigg).
\end{align*}
Optimizing in $p$ with the help of Lemma~\ref{lem:optimal} yields \eqref{e:expdecay}.
\end{proof}

\subsubsection{\textbf{Moment asymptotics of the Borel-Cantelli overlap count
}}\label{ss:BCmain1}\hfill\\ 

 \noindent The proofs of Theorem~\ref{thm:asL}, items b) and c), rely on the idea of quantifying the overlap statistics as developed in \cite{EHS26} (see also \cite{EH22}). However, the overlap statistic used in item c), does not start at $n=1$, but at some large value $n=N+1$ since we deal with the remainder of a convergent series up to $n=N$. For convenience, and since it is not in the literature, we present the respective asymptotic results for the remainder of the overlap statistics starting at $N$ below for independent increments. This is a generalization of Corollary 3 in \cite{EH22}. We point out that such an improvement of integrability is not possible in general for $\m$.

\begin{prop}\label{prop: asymptotic}
Given a probability space $(\Omega, \aA, \PP)$, consider an independent family of events $(E_n)_{n\in \NN_0}$. We define for $N\in \NN_0$ 
\[
\oO_{N} := \sum_{n=N+1}^\infty \ind(E_n), 
\qquad \mbox{ and } \qquad C_{N} := \sum_{n=N}^\infty \PP(E_n),    
\]
and assume the existence of a continuous, decreasing, invertible function $L: (0, \infty)\ra (0, \infty)$ satisfying $L(m) = C_m$, $m\in \NN_0$. 
Then for all $N\in \NN_0$ we have for any $\delta>1$, and all $r>0$,
\begin{align*}
\EE[e^{r\oO_N}-1] &\lqq \frac{\delta}{\delta-1}  \exp\Big(r \big(L^{-1}(e^{-r}/\delta)-N\big)\Big),
\end{align*}
and for all $k\in \NN_0$ we obtain 
\begin{align*}
\PP(\oO_N\gqq k) \lqq \frac{\delta}{\delta-1}\inf_{r>0} \frac{\exp\Big(r \big(L^{-1}(e^{-r}/\delta)-N\big)\Big)}{e^{rk}-1}.
\end{align*}
and

\begin{align*}
\PP(\oO_N\gqq k) &\lqq \inf_{r>0} \frac{\frac{\delta}{\delta-1}\exp\Big(r \big(L^{-1}(e^{-r}/\delta)-N\big)\Big)+1}{e^{rk}}\\
&\lqq \frac{\delta}{\delta-1}\exp\Big(-F_\delta^*(N+k)\Big)+\exp(-Rk),
\end{align*}
where $F_\delta^*(r^*)$ is the Fenchel-Legendre transform of the function $r\mapsto rL^{-1}(e^{-r}/\delta)$, i.e. 
\begin{align*}
F_\delta^*(r^*):=\sup_{r>0}\left(r\,r^*-rL^{-1}(e^{-r}/\delta)\right),
\end{align*}
and $R$ is the maximizer of this supremum.
\end{prop}
\noindent The proof is given in Appendix~\ref{ss:stepbystep}. 

\begin{exm}[Polynomial decay] Let the assumptions of Proposition~\ref{prop: asymptotic} be satisfied.
    \begin{enumerate}
    \item The general case: Consider $\PP(A_n) \lqq \frac{c}{n^p}$ for some $c>0$ and $p>1$.
    Then $\sum_{m\gqq n}\PP(A_m) \lqq c\Big(\frac{1}{n^p} + \int_n^\infty s^{-p} ds\Big) \lqq c (1+ \frac{1}{p-1}) n^{-(p-1)}$.
    For $L(s)=  \frac{cp}{p-1} s^{-(p-1)}$ we have that
    $L^{-1}(t)=  \Big(\frac{p-1}{cp} t\Big)^{-\frac{1}{p-1}}$. For $\delta = 2$ this means
    \[
    \EE[e^{r \oO_N}-1] \lqq 2 \exp\Big(-r N\Big) \cdot \exp\Big( C r e^{\frac{r}{p-1}}\Big),
    \]
    where $C = \Big(2\frac{cp}{p-1}\Big)^\frac{1}{p-1}$ such that for $r_* = (p-1) \ln(k/(C \ln(k)))$ we have due to Markov's inequality
    \begin{align*}
    \PP(\oO_N\gqq k)
    &\lqq (e^{r_* k}-1)^{-1}\bigg(2 \exp\Big(-r_* N\Big) \cdot \exp\Big( C r_* e^{\frac{r_*}{p-1}}\Big)\bigg).
    \end{align*}
    Since for $k_* := \min\{k\gqq 3~|~C\ln(j) >1, \frac{C \ln(j)}{j}<1, \ln(C\ln(j))>0 \mbox{ for all }j\gqq k\}$ we have
    \begin{align*}
    &(e^{r_* k}-1)^{-1} =  \bigg(\Big(\frac{k}{C \ln(k)}\Big)^{(p-1)k}-1\bigg)^{-1} = \Big(\frac{k}{C \ln(k)}\Big)^{-(p-1)k}\bigg(1-\Big(\frac{k}{C \ln(k)}\Big)^{-(p-1)k}\bigg)^{-1}\\
    &\qquad\qquad\quad\lqq \Big(1-\Big(\frac{C \ln(k_*)}{k_*}\Big)^{p-1}\Big)^{-1} e^{-(p-1)k (\ln(k) - \ln(C \ln(k)))},\ \text{and}\\
&\exp\Big(-r_* N\Big) = e^{-(p-1) N} = e^{-(p-1) N (\ln(k)- \ln(C \ln(k))},\ \text{and further,}\\
&\exp\Big( C r_* e^{\frac{r_*}{p-1}}\Big) = \exp\Big( C (p-1) (\ln(k) - \ln(C\ln(k))) \frac{k}{C\ln(k)}\Big)\\
    &= \exp\Big( (p-1) k - C (p-1) k\frac{\ln(C\ln(k))}{C \ln(k)} \Big) \lqq \exp\Big( (p-1) k\Big).
    \end{align*}
    we obtain a constant $K_p = \Big(1-\Big(\frac{C \ln(k_*)}{k_*}\Big)^{p-1}\Big)^{-1}$ such that for $k\gqq k_*$
    \begin{align*}
    \PP(\oO_N\gqq k)
    &\lqq K_p e^{-(p-1)k \ln(k) (1- \frac{\ln(C \ln(k))+1}{\ln(k)} + \frac{\ln(C\ln(k))}{ \ln(k)^2}+ \frac{N}{k}(1-\frac{\ln(C\ln(k))}{\ln(k)})}.
    \end{align*}
    \item The particular case of Lemma~\ref{lem:fastsicherextrem}: For an i.i.d. sequence $(X_n)_{n\in \NN}$ of standard
    normals, for $\alpha>0$ and $N\in \NN_0$, we consider the overlap count of the events
    $E_n = \{X_n > \sqrt{2(1 + \alpha) \ln(n + 1)}\}$. In the proof of Lemma 3 we see that $\PP(E_n)\lqq \frac{1}{n^{\alpha+1}}$ such that for the respective overlap count we have $p -1=\alpha$, $C_\alpha := \Big(\frac{2(1+\alpha)}{\alpha}\Big)^{\frac{1}{\alpha}}$,
\[
k_* := \min\Big\{k\gqq 3 ~\Big|~ C_\alpha \ln(k)\gqq 1 \text{ and } \frac{C_\alpha \ln(j)}{j}<1 \text{ for all } j\gqq k\Big\}\]
and $K_\alpha := 2\bigg(1-\Big(\frac{C_\alpha \ln(k_*)}{k_*}\Big)^{\alpha}\bigg)^{-1}$.
Then for all $N\in \NN_0$ and $k\gqq k_*$ we have
\[
\PP(\oO_N\gqq k) \lqq K_\alpha\, e^{\alpha k}\Big(\frac{C_\alpha \ln(k)}{k}\Big)^{\alpha(k+N)}.
\]
\end{enumerate}
\end{exm}

\begin{exm}[Independent events with exponential decay]\label{ex: indepexp}
Let the assumptions of Proposition~\ref{prop: asymptotic} be satisfied. For $\PP(E_n ) \lqq c · b^n$, $n \in \NN$, for some $c > 0$ and $b \in (0, 1)$. Then $\sum_{m=n}^\infty \PP(E_m) \lqq \frac{c}{1-b} b^n$, consequently choose $L(s) = \frac{c}{1-b} b^s$ and $L^{-1}(t) = \frac{\ln\left(\frac{c}{(1-b)t} \right)}{|\ln(b)|}$ and $L^{-1}\left(\frac{e^{-r}}{2}\right) = \frac{\ln\left(\frac{c}{2(1-b)} \right)+r}{|\ln(b)|}$. By Markov's inequality,
for $r>0$,
\[
\PP(\oO_N\gqq k) \lqq \frac{2}{1-e^{-rk}}\exp\Big(\frac{r^2}{|\ln(b)|} - r\Big(N+k-\frac{\ln\big(\frac{2c}{1-b}\big)}{|\ln(b)|}\Big)\Big).
\]
Choosing $r_* = \frac{|\ln(b)|(N+k)-\ln\big(\frac{2c}{1-b}\big)}{2}$, we obtain for all $N\in\NN_0$ and $k\in\NN$
with $|\ln(b)|(N+k) - \ln\big(\frac{2c}{1-b}\big) \gqq 2\ln(2)$ that
\[
\PP(\oO_N\gqq k) \lqq
4\exp\bigg(-\frac{\big(|\ln(b)|(N+k)-\ln\big(\frac{2c}{1-b}\big)\big)^2}{4|\ln(b)|}\bigg).
\]
In particular, we obtain the Gaussian type decay $\PP(\oO_N\gqq k) \lqq b^{\frac{k^2}{4}(1+o(1))}$ as $k\to\infty$.

\end{exm}

\begin{exm}[Independent events with Gaussian decay]\label{ex:indepGauss}
Let the assumptions of Proposition 1 be satisfied and assume $\PP(E_n) \lqq b^{n^2}$, $n\in\NN$, for some $b \in (0,1)$.
Then $\sum_{m=n}^\infty \PP(E_m) \lqq b^{n^2}\sum_{j=0}^\infty b^{2j} = \frac{b^{n^2}}{1-b^2}$, so we define
$L(s) = \frac{1}{1-b^2}\, b^{s^2}$. Hence
\[
L^{-1}(t) = \sqrt{\frac{\ln\big(\frac{1}{(1-b^2)\,t}\big)}{|\ln(b)|}}
\qquad\text{such that}\qquad
L^{-1}\Big(\frac{e^{-r}}{2}\Big) = \sqrt{\frac{r + \ln\left(\frac{2}{1-b^2}\right)}{|\ln(b)|}}.
\]
and
\begin{equation}\label{e:expMomentGauss}
\EE[e^{r \oO_N}-1] \lqq 2\exp\big(-r(N+1)\big)\exp\bigg(\frac{\sqrt{r^3 + r^2\ln\left(\frac{2}{1-b^2}\right)}}{\sqrt{|\ln(b)|}}\bigg).
\end{equation}
Claim: For all $N, k\gqq 1$ we have
\begin{align}\label{e:expkubisch}
\mathbb{P}(\mathcal{O}_N \geqslant k)& \leqslant 2\inf_{r>0}\frac{\exp\big(-r(N+1)\big)\exp\Big(\frac{\sqrt{r^3+r^2\ln\left(\frac{2}{1-b^2}\right)}}{\sqrt{|\ln(b)|}}\Big)}{e^{rk}-1}
\\&\lqq \frac{2e}{e-1}\Big(\frac{2}{1-b^2}\Big)^{\frac{1}{3}\sqrt{\frac{\ln\left(\frac{2}{1-b^2}\right)}{|\ln(b)|}}}\, b^{\frac{1}{9}(N+k+1)^3}.\nonumber
\end{align}
 The optimization on the right-hand side of \eqref{e:expkubisch} is given in  Appendix \ref{sss:BC}. 
\end{exm}

\begin{rem}
It is natural to ask whether for any sequence $E_n$ such that $\PP(E_n)\lqq b^{n^\ell}$ for $b\in (0,1)$ and $\ell\in \NN$, there are constants $C_1, C_2>0$ such that 
\[
\PP(\oO_N\gqq k) \lqq C_1 e^{-C_2 (N+k+1)^{\ell+1}}, \qquad \mbox{ for all }k,N\in \NN.   
\]
\end{rem}

\appendix 
\section{\textbf{Proof of: Rates of almost sure convergence in L\'evy's construction}}

\subsection{\textbf{The random upper bound}}\label{ss:randomupper}\hfill\\

\noindent In the sequel we give an almost sure upper bound on $\max_{n\gqq J} |Z_n|$ for some $J\in \NN$ and a sequence of i.i.d.~standard normals $(Z_n)_{n\in \NN}$. Of course, this is a standard topic in extreme value theory, where many  particularly fine results on the convergence in law and almost sure convergence are derived, see in particular \cite[Example 3.5.4, p.~174]{KEM98}. While the rates obtained there are stronger, our results yield an exponentially integrable prefactor, which converges exponentially fast to $1$. Our main focus are the a.s.~rates of convergence in $J$. The subsequent result can be considered an asymptotic (in $J$) quantified version of \cite[Lemma 3.2]{Ste01}. 

\begin{lem}\label{lem:fastsicherextrem}
Consider an i.i.d.~sequence  $(Z_n)_{n\in \NN}$ of standard normal random variables.  
Then for all $N\in \NN$, $N\gqq 2$ and $\al>0$ there exists a nonnegative random variable $\Gamma_{\al, N}$ such that $\PP$-a.s.,
\begin{equation}\label{e:direktfs}
|Z_n|\lqq \sqrt{1+\al}\cdot \max\{\Gamma_{\al, N}, \,1\} \cdot \sqrt{2\ln(n)}, \quad \mbox{ for all } n\gqq N+1. 
\end{equation}
 Then for all $\alpha>0$ and $q>0$ 
 and $N\in \NN$ such that $(1+\al)\ln(N) >q$ we have 
\begin{align*}
 \EE\Big[e^{q \cdot[\max\{\Gamma_{\al, N}^2, \,1\}-1]}-1\Big]\lqq \frac{2q\big(1+\frac{1}{2\al\ln(N)}\big)N\sqrt{\pi}\big(1-\Phi\big(\sqrt{2((1+\alpha)\ln(N)-q)}\big)\big)}{e^q(1+\al)^{1/2}\al \sqrt{\ln(N)}\sqrt{(1+\alpha)\ln(N)-q}}.
\end{align*}    
\end{lem}
\noindent It is obvious, that the almost sure inequality \eqref{e:direktfs} is intimately linked to \cite[Proposition 2.4.16]{Ta14} and its structural insights. 
For the convenience of the reader we give an elementary proof in Appendix~\ref{a:fastsicherextrem}.
\medskip

\begin{proof}\textbf{(of Theorem~\ref{thm:asL} a)}
The disjoint support (without boundary) of the $H_n$,
$\|H_n\|_\infty \lqq 1$ and setting \[\kappa^*(t, j) := \mbox{argmax}_{0\lqq k\lqq 2^{j}-1} H_{2^j+k}(t), \qquad t\in [0, 1], j\in \NN\] yield that \eqref{d:BrownscheBewegung} and \eqref{e:Levylimit} as well as Lemma~\ref{lem:fastsicherextrem} for $n=2^{J+1}-1$ imply $\PP$-a.s. 
\begin{align*}
|W_t - L^J_t| 
 &\lqq \sum_{j= J+1}^\infty \sum_{k=0}^{2^{j}-1} 
  2^{-j/2-1} \cdot |Z_{2^j+k}| \cdot H_{2^j+k}(t)
   \lqq \sum_{j= J+1}^\infty 
  2^{-j/2-1} \cdot |Z_{2^j+\kappa^*(t, j)}| \\  
 &\lqq \sqrt{1+\al}\cdot \max\{\Lambda_{J}, 1\}  \sum_{j= J+1}^\infty 
  2^{-j/2-1} \cdot \sqrt{2 \ln(2^j+\kappa^*(t, j))}  \\
 &\lqq \sqrt{1+\al}\cdot \max\{\Lambda_{J}, 1\} \sqrt{2\ln(2)} \sum_{j= J+1}^\infty 
  2^{-j/2-1} \cdot \sqrt{j+1},
\end{align*}
where $\Lambda_J$ is the random variable $\Gamma_{\alpha,2^{J+1}-1}$ from Lemma \ref{lem:fastsicherextrem}. Its integrability follows from there (replacing $q$ by $q\ln(2)$).
Thus, going over to the supremum, we find
\begin{align*}
\|W_t - L^J_t\|_\infty\lqq \sqrt{1+\al}\cdot \max\{\Lambda_{J}, 1\} \sqrt{2\ln(2)} \sum_{j= J+1}^\infty 
  2^{-j/2-1} \cdot \sqrt{j+1}.
\end{align*}
As the summands on the right hand side are monotonically decreasing, we may use the integral criterion. Applying also the asymptotic expansion of the incomplete Gamma function (see Appendix~\ref{a:Gamma}, Corollary~\ref{cor:Gamma}), 
we have for all $J\gqq 1$,
\begin{align*} 
&\sum_{j= J+1}^\infty 2^{-j/2-1} \cdot \sqrt{j+1}  
\lqq \frac{1}{\sqrt{2}}\int_J^\infty 2^{-(x+1)/2} \cdot \sqrt{x+1}d x  =  \frac{1}{\sqrt{2}}\int_{J+1}^\infty e^{-\frac{\ln(2)}{2}x} \cdot \sqrt{x}d x  \\  
&= \frac{1}{\sqrt{2}} \left(\frac{2}{\ln(2)}\right)^\frac{3}{2} \int_{\frac{\ln(2)}{2}(J+1)}^\infty e^{-y} \cdot \sqrt{y}d y\lqq \frac{\sqrt{2}}{\ln(2)} \Big(1 + \frac{1}{\ln(2)(J+1)}\Big) \cdot 
e^{-\frac{\ln(2)}{2}(J+1)} \cdot \sqrt{J+1}\\
&\lqq \frac{1}{\ln(2)} \Big(1 + \frac{1}{\ln(2)2}\Big)\sqrt{J+1}\cdot 2^{-\frac{J}{2}}.
 \end{align*}
Combining the previous inequalities we conclude inequality \eqref{eq:zufaelligeschranke} in Theorem~\ref{thm:asL}. 
\end{proof} 
 
 Taking the expectation we obtain the statement of Corollary~\ref{cor:L2}:
\begin{proof}\textbf{(of Corollary~\ref{cor:L2})} Taking the $L^2$-norm with respect to $\PP$ in  \eqref{eq:zufaelligeschranke} for $\al>0$ and minimizing, we have
 \begin{align*}
 \EE[ \|L^J  - W \|_\infty^2]^\frac{1}{2} 
 \lqq \sqrt{J+1}\cdot C_a\cdot 2^{-J/2}\cdot  \inf_{\al>0} \sqrt{1+\alpha}\cdot  \Big(\EE[\max\{\Lambda_J^2(\alpha), 1\}]\Big)^\frac{1}{2}.
\end{align*}
To estimate the norm of $\Lambda_J$, we get 
\begin{align*}
\EE[\max\{\Lambda_J^2(\alpha), 1\}] = 1+\int_1^\infty 2t\PP(\Lambda_J>t)dt,
\end{align*}
and $\PP(\Lambda>t)$ has been estimated in \eqref{eq:gamma-prob-ineq} in the proof of Lemma \ref{lem:fastsicherextrem} such that we get, following the same way as \eqref{eq:gamma-int} for $N=2^{J+1}-1$,
\begin{align*}
 &1+\int_1^\infty 2t\PP(\Lambda_J>t)dt\\
 &\lqq 1+\frac{2\big(1+\frac{1}{2\al\ln(N)}\big)N\sqrt{\pi}\big(1-\Phi\big(\sqrt{2(1+\alpha)\ln(N)}\big)\big)}{(1+\al)^{1/2}\al \sqrt{\ln(N)}\sqrt{(1+\alpha)\ln(N)}}.\\
 &\lqq 1+\frac{2\big(1+\frac{1}{2\ln(2)\al J}\big)(2^{J+1}-1)\sqrt{\pi}\big(1-\Phi\big(\sqrt{2\ln(2)(1+\alpha)J}\big)\big)}{(1+\al)^{1/2}\al \sqrt{J\ln(2)}\sqrt{(1+\alpha)J\ln(2)}}.
\end{align*}
By the Börjesson-Sundberg estimate \eqref{eq:BSW-bound}, we get 
$$1-\Phi\big(\sqrt{2\ln(2)(1+\alpha)J}\big)<\frac{2^{-\frac{(1+\alpha)J}{2}}}{\sqrt{2\pi}\sqrt{2\ln(2)(1+\alpha)J}},$$
which leads to
\begin{align*}
 &1+\int_1^\infty 2t\PP(\Lambda_J>t)dt\lqq 1+\frac{\sqrt{2}\big(1+\frac{1}{2\ln(2)\al J}\big)2^{J\big(\frac{1-\alpha}{2}\big)+1}}{(1+\al)^{3/2}\al J^{3/2}\ln(2)^{3/2}}.
\end{align*}
Hence, we have to minimize 
\begin{align}\label{eq:alpha-min}
1+\alpha+\frac{\sqrt{2}\big(1+\frac{1}{2\ln(2)\al J}\big)2^{J\big(\frac{1-\alpha}{2}\big)+1}}{(1+\al)^{1/2}\al J^{3/2}\ln(2)^{3/2}}.
\end{align}
Since for $\alpha>1$, the expression is decreasing in $J$, also its minimized value for $\alpha\in(1,\infty)$ is decreasing in $J$. The largest value minimized like this is the one where we set $J=1$. It is attained at $\alpha^*\approx 2.313381784655773$ (which actually is the minimizer in $(0,\infty)$) and measures about $4.281911033086089$.
Taking square roots yields the constant $c$ and the statement.
\end{proof}

\medskip 
\subsection{\textbf{The random frequency of step by step deviations}}\label{ss:stepbystep}\hfill\\

\noindent In the sequel we improve the a.s.~quantification in Theorem \ref{thm:asL} a)  
with the help of moment results on the overlap statistic studied in \cite{EH22}. 
For this purpose we show a general parametrized version of Theorem 3 in \cite{EH22}. Note that the case of independent events is not covered in \cite{EHS26}.

\medskip

\subsubsection{\textbf{Moment asymptotic of the Borel-Cantelli overlap count for independent events}}\label{sss:BC}\hfill
\begin{proof}\textbf{(of Proposition~\ref{prop: asymptotic} in Subsection~\ref{ss:BCmain1})} We first consider the case of $C_{N+1}< 1$.\\ 
\noindent \textbf{Claim: } For any $N\in \NN_0$ and $r< |\ln(C_{N+1})|$ we have 
\[
\EE[e^{p\oO_N}] \lqq 1 + \frac{C_{N+1} e^r}{1-C_{N+1} e^r}. 
\]
The proof of \cite[Theorem 3]{EH22} remains untouched, 
except of the replacement of $C_1$ by $C_{N+1}$. 
For $N\lqq M$, $N, M\in \NN_0$ we define $\oO_{N, M}:= \sum_{n= N+1}^M \ind_{E_n}$ and $\oO_N := \lim_{M\ra\infty} \oO_{N, M}$. In addition, set $G_k^{N, M}:=\{\oO_{N, M} = k\}$.  
For any $0 \lqq k \lqq M-N$ we have the Schuette-Nesbitt formula \cite{G79} 
\[
\sum_{k=0}^{M-N} a_k \PP(G_k^{N, M}) = \sum_{n=0}^{M-N} \qQ^{N, M}_n b_n,  
\quad \mbox{ where } \quad \qQ^{N, M}_n = \sum_{\substack{J\subseteq \{N+1, \dots, M\}\\|J|= n}} \PP \Big(\bigcap_{j\in J} E_j\Big),
\]
and $b_n = \sum_{j=0}^n \binom{n}{j} (-1)^{n-j} a_j$. For the choice $a_k = e^{rk}$, for some $r>0$, we obtain the values $b_n=(e^{r}-1)^n\leq e^{rn}$.

The independence of the events yields that 
\begin{align*}
\qQ^{N, M}_n =  \sum_{\substack{J\subseteq \{N+1, \dots, M\}\\|J|= n}} \prod_{j\in J} \PP(E_j) 
\lqq \Big(\sum_{i=N+1}^M \PP(E_i)\Big)^{n} \lqq C_{N+1}^n.
\end{align*}
Following the remaining steps of the proof of \cite[Theorem 3]{EH22}, sending $M\ra\infty$, 
we conclude for $r< |\ln(C_{N+1})|$ that  
\begin{align*}
\sum_{k=0}^\infty \PP(G_k^N) e^{rk} \lqq \frac{1}{1-C_{N+1} e^r}, \qquad \mbox{ where } G_k^N = \{\lim\limits_{M\ra\infty} \oO_{N, M} = k\}.
\end{align*}
The proof of the claim is complete. 

\noindent We now show the statement. 
For any $\delta>1$ and $m\gqq N+1$ such that $C_{N+1+m} < e^{-r}/\delta$ we write 
  \begin{align*}
 \EE[e^{r\oO_N}-1] 
 &\lqq \EE[e^{r (m+\oO_{N+m})}-1] 
 = e^{rm} \EE[e^{r \oO_{N+1+m}}-e^{-rm}] 
 \lqq \frac{e^{rm}}{1-C_{N+1+m} e^r} = \frac{e^{r (\ell-(N+1))}}{1-C_{\ell} e^r}, 
\end{align*}
where $\ell = N+1+ m$. 
If we define $\Lambda(r, \delta):= \inf\{\ell\gqq 1~\vert~ C_\ell < e^{-r}/\delta\}$ we obtain 
 \begin{align*}
\EE[e^{r\oO_N}-1] &\lqq e^{-r(N+1)} \frac{\delta}{\delta-1}e^{r \Lambda(r, \delta)}
\end{align*}
such that for any continuous, decreasing, invertible function $L: (0, \infty) \ra (0, \infty)$ satisfying $L(m) = C_m$ and $\delta >1$ we have 
\begin{align*}
\EE[e^{r\oO_N}-1] 
&\lqq \frac{\delta}{\delta-1} e^{-r (N+1)} e^{r L^{-1}(e^{-r}/\delta)}.
\end{align*}
The remaining inequalities having $\PP(\oO_N\gqq k)$ on the left hand side follow from applying Markov's inequality for the functions $x\mapsto e^{rx}-1$ and $x\mapsto e^{rx}$.
Finally, we observe that
\begin{align*}
&\inf_{r>0} \frac{\frac{\delta}{\delta-1}\exp\Big(r \big(L^{-1}(e^{-r}/\delta)-(N+1)\big)\Big)+1}{e^{rk}}\\
&\qquad \lqq \frac{\delta}{\delta-1}\exp\Big(R \big(L^{-1}(e^{-R}/\delta)-(N+1-k)\big)\Big)+e^{-Rk},
\end{align*}
where $R$ minimizes $r\mapsto \frac{\delta}{\delta-1}\exp\Big(r \big(L^{-1}(e^{-r}/\delta)-(N+1-k)\big)\Big)$, proving the last inequality.
\end{proof}

\begin{proof}[Proof (of the Claim in Example~\ref{ex:indepGauss})]
This is seen by optimizing the exponents in \eqref{e:expMomentGauss}. By Markov's inequality and unifying bases we have for any $r_*>0$
\begin{align*}
\PP(\oO_N \gqq k)
&\lqq \inf_{r>0} 2\exp\Bigg(\frac{\sqrt{r^3+r^2\ln\left(\frac{2}{1-b^2}\right)}}{\sqrt{|\ln(b)|}} - (N+1)r - \ln(e^{rk}-1)\Bigg)\\
&\lqq \frac{2e^{r_* k}}{e^{r_* k}-1}\inf_{r>r_*}\exp\Bigg(\frac{\sqrt{r^3+r^2\ln\left(\frac{2}{1-b^2}\right)}}{\sqrt{|\ln(b)|}} - (N+1+k)r\Bigg).
\end{align*}
In the sequel, we minimize in $(r_*,\infty)$
\[
r\mapsto Q(r) := \frac{\sqrt{r^3+r^2\ln\left(\frac{2}{1-b^2}\right)}}{\sqrt{|\ln(b)|}} - (N+1+k)r.
\]
Note that
\begin{align*}
Q'(r) &= \frac{3r^2+2\ln\left(\frac{2}{1-b^2}\right)r}{2\sqrt{|\ln(b)|}\,r\sqrt{r+\ln\left(\frac{2}{1-b^2}\right)}} - (N+1+k)
= \frac{3r+2\ln\left(\frac{2}{1-b^2}\right)}{2\sqrt{|\ln(b)|}\sqrt{r+\ln\left(\frac{2}{1-b^2}\right)}} - (N+1+k)\\
&= \frac{3r+2\ln\left(\frac{2}{1-b^2}\right) - 2(N+1+k)\sqrt{|\ln(b)|}\sqrt{r+\ln\left(\frac{2}{1-b^2}\right)}}{2\sqrt{|\ln(b)|}\sqrt{r+\ln\left(\frac{2}{1-b^2}\right)}}.
\end{align*}
Hence $0 = r + \frac{2\ln\left(\frac{2}{1-b^2}\right)}{3} - \frac{2(N+1+k)\sqrt{|\ln(b)|}}{3}\sqrt{r+\ln\left(\frac{2}{1-b^2}\right)}$ implies $Q'(r)=0$. The optimizer is given for $A = \frac{2\ln\left(\frac{2}{1-b^2}\right)}{3}$, $B = \frac{2(N+1+k)\sqrt{|\ln(b)|}}{3}$ and $C=\ln\left(\frac{2}{1-b^2}\right)$ by
\begin{align*}
r_0 &= \frac{1}{2}\Big(B^2 - 2A + B\sqrt{B^2+4(C-A)}\Big)\\
&= \frac{2(N+1+k)^2|\ln(b)|}{9} - \frac{2\ln\left(\frac{2}{1-b^2}\right)}{3} \\
&\quad+ \frac{(N+k+1)\sqrt{|\ln(b)|}}{3}\sqrt{\frac{4(N+1+k)^2|\ln(b)|}{9} + \frac{4\ln\left(\frac{2}{1-b^2}\right)}{3}}\\
&= \frac{2}{9}(N+1+k)^2|\ln(b)| - \frac{2\ln\left(\frac{2}{1-b^2}\right)}{3} \\
&\quad+ \frac{2}{9}(N+k+1)\sqrt{|\ln(b)|}\sqrt{(N+1+k)^2|\ln(b)| + 3\ln\left(\frac{2}{1-b^2}\right)}.
\end{align*}
Since for large values of $N,k$, the optimizer is of order $r_0 \approx r_1 := \frac{4}{9}(N+k+1)^2|\ln(b)|$, we calculate, provided that $r_1 > r_*$,
\begin{align*}
Q(r_0) \lqq Q(r_1) &= \frac{r_1\sqrt{r_1+\ln\left(\frac{2}{1-b^2}\right)}}{\sqrt{|\ln(b)|}} - (N+1+k)r_1\\
&= \frac{4}{9}(N+k+1)^2|\ln(b)|\sqrt{\frac{4}{9}(N+k+1)^2 + \frac{\ln\left(\frac{2}{1-b^2}\right)}{|\ln(b)|}} - \frac{4}{9}(N+k+1)^3|\ln(b)|.
\end{align*}
Using basic calculus, it is easy to see that the function $x\mapsto -\frac{9}{27}x^3 + \frac{4}{9}\sqrt{\frac{4}{9}x^2 + \frac{\ln\left(\frac{2}{1-b^2}\right)}{|\ln(b)|}}\,x^2$ is bounded by
\[
\Big(\frac{4}{9}D^2\sqrt{\frac{4}{9}D^2+1} - \frac{9}{27}D^3\Big)\Bigg(\frac{\ln\left(\frac{2}{1-b^2}\right)}{|\ln(b)|}\Bigg)^{3/2} \approx 0.3054\Bigg(\frac{\ln\left(\frac{2}{1-b^2}\right)}{|\ln(b)|}\Bigg)^{3/2} \lqq \frac{1}{3}\Bigg(\frac{\ln\left(\frac{2}{1-b^2}\right)}{|\ln(b)|}\Bigg)^{3/2},
\]
where $D = \sqrt{\frac{27\sqrt{217}+39}{136}} \approx 1.792$. This implies that
\begin{align*}
&\frac{4}{9}(N+k+1)^2|\ln(b)|\sqrt{\frac{4}{9}(N+k+1)^2 + \frac{\ln\left(\frac{2}{1-b^2}\right)}{|\ln(b)|}} - \frac{4}{9}(N+k+1)^3|\ln(b)|\\
&\qquad\lqq -\frac{1}{9}(N+k+1)^3|\ln(b)| + \frac{\ln\left(\frac{2}{1-b^2}\right)^{3/2}}{3\sqrt{|\ln(b)|}}.
\end{align*}
Hence, we finally obtain, taking $r_* = 1$ and $k,N\gqq 1$ with $\frac{4}{9}(N+k+1)^2|\ln(b)| > 1$,
\[
\PP(\oO_N \gqq k) \lqq \frac{2e}{e-1}\Big(\frac{2}{1-b^2}\Big)^{\frac{1}{3}\sqrt{\frac{\ln\left(\frac{2}{1-b^2}\right)}{|\ln(b)|}}}\, b^{\frac{1}{9}(N+k+1)^3}.
\]
\end{proof}

\medskip

\subsubsection{\textbf{Proof of Theorem~\ref{thm:asL} c): }}\hfill\\

\begin{proof}\textbf{(of Theorem~\ref{thm:asL} c)}
For $b_\al =  2^{-\al}$ and 
$\dD_j := \{k\cdot 2^{-j}~\vert~k=0, \dots, 2^{j}\}$, $j\in \NN_0$, we have\\ 
\begin{align*}
&\PP(~\exists~d\in \dD_j\setminus \dD_{j-1} \mbox{ with } |Z_d|\gqq \sqrt{1+\al}\sqrt{2\ln(2)} \sqrt{j})\\
&\qquad \lqq \sum_{d\in \dD_j\setminus \dD_{j-1} } \PP(|Z_d|\gqq \sqrt{1+\al}\sqrt{2\ln(2)}\sqrt{j})\lqq 2^{j-1} \exp\Big(-\frac{(1+\al) 2\ln(2)}{2} j\Big) = \frac{1}{2}\cdot b_\al^{j}, 
\end{align*}
with the help of Chernov's bound. In addition, the dyadics $\dD_j  \subseteq  \dD_{j+1}$, $j\in \NN_0$, are monotonic and the family of events 
\[
A_j:= \{\mbox{there exists }d\in \dD_j\setminus \dD_{j-1} \mbox{ with } |Z_d|\gqq \sqrt{1+\al} \sqrt{j}\}, \qquad j\in \NN 
\]
is independent. Note that by construction $A_j = \{\|G_n\|_\infty > \sqrt{1+\al}\sqrt{2\ln(2)} \sqrt{j}\cdot 2^{-(j+1)/2}\}$. 
We define 
\[
\oO_J:=  \sum_{j=J+1}^\infty \ind\{\|G_j\|_\infty > \sqrt{1+\al}\cdot \sqrt{2\ln(2)} \cdot \sqrt{j}\cdot 2^{-j/2}\}.
\]
By Example \ref{ex: indepexp} we have for each fixed $J\in \NN_0$ and $0 < r < \al \ln(2)$
\begin{align*}
\PP(\oO_J\gqq k) \lqq\frac{2}{1-e^{-rk}}\exp\bigg(\frac{r^2}{\alpha\ln(2)}-r\bigg(N+k-\frac{\ln\big(\frac{1}{1-2^{-\alpha}}\big)}{\alpha\ln(2)}\bigg)\bigg)
\end{align*}
and consequently, optimizing over $r>0$, we have 
\[
\PP(\oO_J\gqq k) \lqq 4\exp\bigg(-\big(\frac{\alpha\ln(2)(J+k)-\ln\big(\frac{1}{1-2^{-\alpha}}\big)\big)^2}{4\alpha\ln(2)}\bigg). 
\]
This shows \eqref{e:mdfLevy} and finishes the proof of Theorem~\ref{thm:asL} c).
\end{proof}

\subsection{\textbf{The deterministic upper bound with random modulus of convergence}}\label{ss:detupperbound}\hfill

\begin{lem}\label{lem:lastentry}
Consider  a probability space $(\Omega, \aA, \PP)$ and a sequence $(E_n)_{n\in \NN_0}$ of independent events and $U_n := \bigcup_{m=n}^\infty E_m$, $n\in \NN_0$. 
Then the random variable $\jJ(\om) := \sup\{n\in \NN_0~\vert~ \om \in U_n\}$ satisfies 
\[
\PP(\jJ = k) = \PP(E_k)  e^{\sum_{\ell=k+1}^\infty\ln(1-\PP(E_\ell))}, \qquad k\in \NN_0.   
\]
\end{lem}

\begin{proof} Note that $(U_n)_{n\in \NN_0}$ is a nested sequence by construction. 
The independence of $(E_n)_{n\in \NN_0}$ implies
\begin{align*}
&\PP(\jJ = k) 
= \PP\Big(U_k \cap \bigcap_{\ell = k+1}^\infty U_\ell^c\Big) = \PP\Big((E_k \cup U_{k+1}) \cap \bigcap_{\ell = k+1}^\infty U_\ell^c\Big) \\
&= \PP\Big(E_k  \cap \bigcap_{\ell = k+1}^\infty U_\ell^c\Big)= \PP\Big(E_k  \cap \bigcap_{\ell = k+1}^\infty \Big(\bigcup_{r =\ell}^\infty E_r\Big)^c\Big)= \PP\Big(E_k  \cap \bigcap_{\ell = k+1}^\infty \bigcap_{r =\ell}^\infty E_r^c\Big)= \PP\Big(E_k  \cap \bigcap_{\ell = k+1}^\infty E_\ell^c\Big)\\
&= \PP(E_k) \prod_{\ell=k+1}^\infty (1-\PP(E_\ell))= \PP(E_k) \prod_{\ell=k+1}^\infty e^{\ln(1-\PP(E_\ell))}= \PP(E_k)  \exp\Big(\sum_{\ell=k+1}^\infty\ln(1-\PP(E_\ell))\Big).
\end{align*}
\end{proof}

\medskip 
\begin{proof}\textbf{(of Theorem~\ref{thm:asL} b)} We apply Lemma~\ref{lem:fastsicherextrem} for some i.i.d.~family 
$(Z_{n})_{n\in \NN_0}$ with $Z_{n} \sim N(0,1)$, where for $n\gqq 1$, we have $n=2^j+k$, with $j=\lfloor\log_2(n)\rfloor$ and $k=n-2^j\in \{0, \dots, 2^{j}-1\}$. 
We set
\begin{align*}
E_n := \Big\{|Z_{2^j +k}| \gqq \sqrt{1+\al} \sqrt{2 \ln(2^j + k)}\Big\}. 
\end{align*}
Lemma \ref{lem:fastsicherextrem} then ensures that there is some random number $\tilde{N}\in \NN$ such that for all $n\gqq\tilde{N}$, we have $|Z_{2^j +k}| \lqq \sqrt{1+\al} \sqrt{2 \ln(2^j + k)}$, writing $n=2^j+k$. In particular this holds for all $n\gqq 2^\jJ(\alpha)$, where $\jJ(\alpha):=\lfloor\log_2(\tilde{N})\rfloor+1$. Then, by the same reasoning as in the proof of Theorem \ref{thm:asL} a), inequality \eqref{eq:thm1b} follows for $J\gqq\jJ(\alpha)$.

We now estimate probabilities for $\jJ(\alpha)$.
By the B\"orjesson-Sundberg bound \cite{BS79} (already appearing in \cite{Woz65} and is actually a bound on a normal variable's Mill's ratio \cite{Mills26}),
\begin{align*}
\frac{t}{1+t^2}e^{-\frac{t^2}{2}}<\PP(Z_{2^j +k}>t)<\frac{e^{-\frac{t^2}{2}}}{\sqrt{2\pi}\cdot t},\quad \text{for }t>0,
\end{align*}
which implies 
\begin{align}\label{eq:BSW-bound}
\frac{2t}{1+t^2}e^{-\frac{t^2}{2}}<\PP(|Z_{2^j +k}|>t)<e^{-\frac{t^2}{2}}\min\Big\{\frac{2}{\sqrt{2\pi}\cdot t},1\Big\},\quad t>0.
\end{align} 
Inserting for $t$ yields
\begin{align*}
\PP(E_n) &= \PP \Big(|Z_{2^j +k}|\gqq \sqrt{2(1+\al) \ln(2^j +k)}\Big)\\
&\lqq  \exp(-(1+\al) \ln(2^j +k))= (2^j +k)^{-(1+\al)}\lqq 2^{-j(1+\al)}
\end{align*}
for the upper bound. For the lower bound we get
\begin{align*}
\PP(E_n)\gqq \frac{2\sqrt{2(1+\al) \ln(2^j +k)}}{1+2(1+\al) \ln(2^j +k)}(2^j +k)^{-(1+\al)}
\end{align*}
We then see that 
\begin{align*}
\PP(\jJ(\alpha)=J)=\PP(\lfloor\log_2\tilde{N}\rfloor=J-1)=\PP(\tilde{N}\in[2^{J-1},2^{J}-1])=\sum_{n=2^{J-1}}^{2^J-1}\PP(\tilde{N}=n),
\end{align*}
and by Lemma \ref{lem:lastentry}, we get
\begin{align*}
\PP(\jJ(\alpha)=J)=\sum_{n=2^{J-1}}^{2^J-1}\PP(E_n)\exp\Big(\sum_{\ell=n+1}^\infty\ln(1-\PP(E_\ell))\Big).
\end{align*}
 The above bounds thus give
 \begin{align*}
 &\sum_{n=2^{J-1}}^{2^J-1}\frac{2\sqrt{2(1+\al) \ln(n)}}{1+2(1+\al) \ln(n)}n^{-(1+\al)}\exp\Big(\sum_{\ell=n+1}^\infty\ln\big(1-\ell^{-(1+\alpha)}\big)\Big)\lqq\PP(\jJ(\alpha)=J)\\
 &\lqq\sum_{n=2^{J-1}}^{2^J-1} n^{-(1+\al)}\exp\bigg(\sum_{\ell=n+1}^\infty\ln\Big(1-\frac{2\sqrt{2(1+\al) \ln(\ell)}}{1+2(1+\al) \ln(\ell)} \ell^{-(1+\alpha)}\Big)\bigg).
 \end{align*}
 We estimate the upper sum by 
 \begin{align*}
 &\sum_{n=2^{J-1}}^{2^J-1}n^{-(1+\alpha)}\lqq 2^{-(J-1)(1+\alpha)}+\int_{2^{J-1}}^{2^J-1}x^{-(1+\alpha)}dx\\
 &=2^{-(J-1)(1+\alpha)}+\frac{1}{\alpha}\big((2^{J-1})^{-\alpha}-(2^{J}-1)^{-\alpha}\big)\lqq \Big(\frac{2^\alpha}{2^{J-1}}+\frac{2^\alpha-1}{\alpha}\Big)2^{-\alpha J}.
 \end{align*}
 For the lower sum, note that $\ln(1-\ell^{-(1+\alpha)})\gqq -2\ln(2)\ell^{-(1+\alpha)}$. Thus 
 $$\sum_{\ell=n+1}^\infty\ln\big(1-\ell^{-(1+\alpha)}\big)\gqq -2\ln(2)\int_n^\infty x^{-(1+\alpha)}dx=-\frac{2\ln(2)}{\alpha}n^{-\alpha}$$
 and $\exp\Big(-\frac{2\ln(2)}{\alpha}n^{-\alpha}\Big)=2^{-\frac{2}{\alpha n^\alpha}}\gqq 2^{-\frac{2}{\alpha 2^{(J-1)\alpha}}}\gqq 2^{-\frac{2}{\alpha}}$ for $n\gqq 2^{J-1}$. Therefore,
 \begin{align*}
 &\sum_{n=2^{J-1}}^{2^J-1}\frac{2\sqrt{2(1+\al) \ln(n)}}{1+2(1+\al) \ln(n)}n^{-(1+\al)}\exp\Big(\sum_{\ell=n+1}^\infty\ln\big(1-\ell^{-(1+\alpha)}\big)\Big)\\
 &\gqq 2^{-\frac{2}{\alpha}}\sum_{n=2^{J-1}}^{2^J-1}\frac{2\sqrt{2(1+\al) \ln(n)}}{1+2(1+\al) \ln(n)}n^{-(1+\al)}.
 \end{align*}
 For $J>1$, we may estimate $\frac{2\sqrt{2(1+\al) \ln(n)}}{1+2(1+\al) \ln(n)}\gqq \frac{1}{\sqrt{2(1+\al) \ln(n)}}\gqq \frac{1}{\sqrt{2(1+\alpha)J\ln(2)}}$. So we continue the above estimate by
 \begin{align*}
 &2^{-\frac{2}{\alpha}}\sum_{n=2^{J-1}}^{2^J-1}\frac{2\sqrt{2(1+\al) \ln(n)}}{1+2(1+\al) \ln(n)}n^{-(1+\al)}\gqq \frac{2^{-\frac{2}{\alpha}}}{\sqrt{2(1+\alpha)J\ln(2)}}\sum_{n=2^{J-1}}^{2^J-1}n^{-(1+\al)}\\
 &\gqq \frac{2^{-\frac{2}{\alpha}}}{\sqrt{2(1+\alpha)J\ln(2)}}\int_{2^{J-1}}^{2^J-1}x^{-(1+\al)}dx=\frac{2^{-\frac{2}{\alpha}}}{\alpha\sqrt{2(1+\alpha)J\ln(2)}}\Big({2^{-(J-1)\alpha}}-(2^J-1)^{-\alpha}\Big)\\
 &=\frac{2^{-\frac{2}{\alpha}}}{\alpha\sqrt{2\ln(2)(1+\alpha)}}\Big({2^{\alpha}}-(1-2^{-J})^{-\alpha}\Big)\cdot\frac{2^{-J\alpha}}{\sqrt{J}}. 
 \end{align*}
 This proves the bounds for $\PP(\jJ(\alpha)=J)$.
\end{proof}

\medskip 

\subsection{\textbf{Almost sure convergence with exponential mean deviation frequency}}\label{ss:exmdf}\hfill\\

\begin{proof}\textbf{(of Theorem~\ref{thm:asL} d)}
For $\e>0$ recall  
\[
\oO_\e(\omega) = \sum_{J=1}^\infty \ind\{\|L^J(\omega)-W(\omega)\|_\infty >\e\}, \qquad 
\m_\e = \sup\{J\in\NN~|~\|L^J(\omega)-W(\omega)\|_\infty >\e\}.
\]
Then, by Corollary~\ref{cor:L2} and Markov's inequality, we have for all $J\gqq 1$
\begin{equation}\label{e:Levywahrsch}
\PP(\|L^J-W\|_\infty >\e) \lqq \e^{-1}\cdot C_a\cdot c \cdot  \sqrt{J+1}\cdot 2^{-\frac{J}{2}}.
\end{equation}
Hence similar to \cite[Example~2]{EHS26}
we have for all $0 \lqq p < \ln(2)/2$ that 
\begin{align*}
\EE[e^{p\oO_\e}]\lqq \EE[e^{p\m_\e}]&\lqq 1+ \e^{-1} (\sqrt{2}-1)\cdot C_a\cdot c  \sum_{n=1}^\infty e^{pn} \sum_{m=n}^\infty \sqrt{m+1}\cdot 2^{-\frac{m}{2}}\\
&=1+ \e^{-1} \cdot (2-\sqrt{2})\cdot C_a\cdot c \cdot \sum_{n=1}^\infty e^{pn} \sum_{m=n+1}^\infty \sqrt{m}\cdot 2^{-\frac{m}{2}}\\
&\lqq 1+ \e^{-1} \cdot (2-\sqrt{2})\cdot C_a\cdot c \cdot \sum_{n=1}^\infty e^{pn} \int_{n}^\infty\sqrt{x}\cdot 2^{-\frac{x}{2}}dx\\
&=1+ \e^{-1} \cdot \frac{4(\sqrt{2}-1)}{\ln(2)^{3/2}}\cdot C_a\cdot c\cdot  \sum_{n=1}^\infty e^{pn} \int_{\frac{\ln(2)n}{2}}^\infty\sqrt{y}\cdot e^{-y}dy.
\end{align*}
Corollary \ref{cor:Gamma} for $a=\frac{3}{2}$ then yields
\begin{align*}
\EE[e^{p\oO_\e}]&\lqq \EE[e^{p\m_\e}]\lqq1+ \e^{-1} \cdot \frac{4(\sqrt{2}-1)}{\sqrt{2}\ln(2)}\Big(1+\frac{1}{\ln(2)}\Big)\cdot C_a\cdot c \cdot  \sum_{n=1}^\infty\sqrt{n}e^{-n(\frac{\ln(2)}{2}-p)}\\
&=1+ \e^{-1} \cdot \frac{4(\sqrt{2}-1)}{\sqrt{2}\ln(2)}\Big(1+\frac{1}{\ln(2)}\Big)\cdot C_a\cdot c \cdot \left(\sum_{1\lqq n\lqq M}\sqrt{n}e^{-n(\frac{\ln(2)}{2}-p)}+\sum_{n>M}^\infty\sqrt{n}e^{-n(\frac{\ln(2)}{2}-p)}\right),
\end{align*}
where $M=\frac{1}{2}\frac{1}{\frac{\ln(2)}{2}-p}$ is chosen such that $x\mapsto\sqrt{x}e^{-x(\frac{\ln(2)}{2}-p)}$ is decreasing for $x\gqq M$. We estimate further,
\begin{align*}
\sum_{1\lqq n\lqq M}\sqrt{n}e^{-n(\frac{\ln(2)}{2}-p)}\lqq M^{3/2}.
\end{align*}
For the remaining sum, we use again the integral criterion, which yields
\begin{equation}\label{e:Gammamoments}
\begin{split}
&\EE[e^{p\m_\e}]\lqq 1+\e^{-1} \cdot \frac{4(\sqrt{2}-1)}{\sqrt{2}\ln(2)}\Big(1+\frac{1}{\ln(2)}\Big)\cdot C_a\cdot c \cdot\left(M^{3/2}+\int_{M}^\infty\sqrt{x}e^{-x(\frac{\ln(2)}{2}-p)}dx\right)\\
&=1+\e^{-1} \cdot \frac{4(\sqrt{2}-1)\cdot C_a\cdot c}{\sqrt{2}\ln(2)}\Big(1+\frac{1}{\ln(2)}\Big) \cdot\left(\frac{1}{2^{3/2}(\frac{\ln(2)}{2}-p)^{3/2}}+\frac{1}{(\frac{\ln(2)}{2}-p)^{3/2}}\int_{M(\frac{\ln(2)}{2}-p)}^\infty\sqrt{y}e^{-y}dy\right)\\
&\lqq 1+\e^{-1} \cdot \frac{4(\sqrt{2}-1)\cdot C_a\cdot c}{\sqrt{2}\ln(2)}\Big(1+\frac{1}{\ln(2)}\Big)\cdot\left(\frac{1}{2^{3/2}(\frac{\ln(2)}{2}-p)^{3/2}}+\frac{1}{(\frac{\ln(2)}{2}-p)^{3/2}}\Gamma(3/2)\right)\\
&=1+\e^{-1} \cdot \frac{4(\sqrt{2}-1)\cdot C_a\cdot c}{\sqrt{2}\ln(2)}\Big(1+\frac{1}{\ln(2)}\Big)\cdot\Big(\frac{1}{2\sqrt{2}}+\Gamma(3/2)\Big)\Big(\frac{\ln(2)}{2} - p\Big)^{-3/2}\\
&=1+\e^{-1} \cdot \frac{(2-\sqrt{2})\cdot C_a\cdot c}{\ln(2)}\Big(1+\frac{1}{\ln(2)}\Big)\Big(\frac{1}{\sqrt{2}}+\sqrt{\pi}\Big) \cdot\Big(\frac{\ln(2)}{2} - p\Big)^{-3/2}\\
&=:1+\e^{-1} \cdot C\Big(\frac{\ln(2)}{2} - p\Big)^{-3/2}.
\end{split}
\end{equation}
This shows \eqref{e:LevOverlap}. Consequently, Markov's inequality yields 
\[
\PP(\oO_\e \gqq k)\lqq \PP(\m_\e \gqq k)\lqq \inf_{p\in [0, \frac{\ln(2)}{2})} e^{-pk} \bigg(1+ \e^{-1} \cdot  C\cdot \Big(\frac{\ln(2)}{2} - p\bigg)^{-3/2}\Big), \qquad k\in \NN,
\]
which shows 
\begin{align}
\PP(\oO_\e \gqq k)\lqq \PP(\m_\e \gqq k)
&\lqq \inf_{p\in [0, \frac{\ln(2)}{2})} e^{-pk} \bigg(1+ \frac{C}{\e}\cdot \Big(\frac{\ln(2)}{2} - p\Big)^{-3/2}\bigg).\label{eq:minimizeplevy0}
 \end{align}
Taking the derivative and finding the zero, the exact minimizer $p^*(k):=\min\Big\{\frac{\ln(2)}{2}-{r(k)^2}, 0\Big\}$ can be found by $r(k)$, which is the smallest positive zero of the polynomial $x^5+\frac{C}{\e}x^2-\frac{3C}{2\e k}$. 
Inserting, we see that $\frac{1}{\sqrt{k}}\lqq r(k)\lqq \sqrt{\frac{3}{2k}}$ as long as $1\lqq \Big(\frac{C}{2\e}\Big)^{2/3} k$. 
We see that $\frac{\ln(2)}{2}- \frac{1}{k}$ is near to the optimal value $p^*(k)$ for $k\gqq 3$. Using this value in \eqref{eq:minimizeplevy0}
we obtain  
 \begin{align}
\PP(\oO_\e \gqq k)\lqq \PP(\m_\e \gqq k)\lqq e\left(1+\frac{C}{\e} \cdot k^{3/2}\right)e^{-\frac{\ln(2)}{2}k}, \qquad k \gqq 3\label{e:Gammamomentsexactmini}
\end{align}
follows. This shows \eqref{eq:minimizeplevy} and finishes the proof of Theorem~\ref{thm:asL} d). 
\end{proof}

\bigskip 
\subsection{\textbf{Almost sure convergence with close to optimal a.s. rate}}\label{a:fsoptimaleRate}
\begin{proof}[\textbf{(of Theorem~\ref{thm:asL} e)}] Going back to \eqref{e:Levywahrsch} we have 
for all $\theta>0$ and $\delta_J := 2^{-\frac{J}{2}}\cdot(J+1)^{\frac{3}{2}}\cdot  \ln(J+1)^{1+\theta}$ 
\begin{equation}\label{e:Levywahrsch2}
\PP(\|L^J-W\|_\infty >\delta_J) \lqq 
\delta_J^{-1}\cdot C_a\cdot c \cdot  \sqrt{J+1}\cdot 2^{-\frac{J}{2}} = \frac{C_a\cdot c}{(J+1) \ln(J+1)^{1+\theta}}, \qquad J\in\NN. 
\end{equation}
Summation over the preceding inequality yields by integral comparison 
for the deviation frequency $\oO_\delta:= \sum_{J=1}^\infty \ind\{\|L^J-W\|_\infty >\delta_J\}$ 
that with the help of the usual first Borel-Cantelli lemma we have 
\begin{align*}
\EE[\oO_\delta] &= \sum_{J=1}^\infty  \PP(\|L^J-W\|_\infty >\delta_J) 
\lqq C_a\cdot c\cdot \sum_{J=1}^\infty \frac{1}{(J+1) \ln(J+1)^{1+\theta}}\\ 
&\lqq C_a\cdot c\cdot \Big(\frac{1}{2\ln(2)^{1+\theta}}+\int_{2}^\infty \frac{1}{x \ln(x)^{1+\theta}} \Big) = C_a\cdot c\cdot \Big(\frac{1}{2\ln(2)^{1+\theta}} + \frac{1}{\theta\ln(2)^\theta}\Big)<\infty. 
\end{align*}
In particular, $\oO_\delta<\infty$, $\PP$-a.s. such that 
\[
\limsup_{J\ra\infty} \|L^J-W\|_\infty \cdot \delta_J^{-1} \lqq 1, \qquad \PP\mbox{-a.s.}
\]
Markov's inequality then yields for all $k\in \NN$ 
\[
\PP(\#\{J\in \NN~\vert~ \|L^J-W\|_\infty>\delta_J\} \gqq k) = \PP(\oO_\delta\gqq k) \lqq k^{-1}\cdot \frac{C_a\cdot c}{\ln(2)^\theta}\cdot \Big(\frac{1}{2\ln(2)} + \frac{1}{\theta}\Big). 
\] 
\end{proof}

\bigskip 

\section{\textbf{Proof of: Deviation frequencies from continuity and H\"older continuity}}

\subsection{\textbf{Doob's ad hoc proof of continuity}}\label{ss:Doob}\hfill\\ 

  \begin{proof}\textbf{(of Theorem~\ref{thm:Doob})}
  Formula (3.6) on p.~577 of \cite{Do84} reads for all $\e>0$ and $n\in \NN$ 
  \begin{align}\label{e:Doobadhoc}
  \PP(E_n(\e)) \lqq \frac{8}{\e} \cdot \sqrt{n} \cdot e^{-\frac{\e^2}{4} n}, 
  \qquad \mbox{ for all }\qquad E_n(\e)= \Bigg\{\sup_{\substack{r,s\in\QQ\cap [0,1]\\ |s-r|\lqq \frac{1}{n}}} |X_s-X_r| \gqq 2\e\Bigg\}.
  \end{align}
  The remaining part is treated almost the same way as in \eqref{e:Gammamoments} and \eqref{e:Gammamomentsexactmini}. 
  For $\oO_\e = \sum_{n=1}^\infty \ind(E_n)$ and $\m_\e(\omega)= \sup\{n\in \NN~|~\omega\in E_n\}$, 
    and all $p \in [0, \frac{\e^2}{4})$, we obtain the constant 
  $C_\e:=16(1+\tfrac{2}{\e^2})(\tfrac{1}{\sqrt{2}}+\sqrt{\pi})e^{\frac{\e^2}{4}}(e^{\frac{\e^2}{4}}-1)$, such that
  \begin{align*}
  \EE[e^{p \oO_\e}]\lqq \EE[e^{p \m_\e}] & \lqq 1 + \frac{ C_\e}{\e^3} \Big(\frac{\e^2}{4}-p\Big)^{-3/2}
  \end{align*}
such that 
  \begin{align}\label{e:Dooboptimization}
  \PP(\oO_\e \gqq k)\lqq   \PP(\m_\e \gqq k)
  &\lqq \inf_{p \in [0, \frac{\e^2}{4})} e^{-kp} \Big(1 + \frac{C_\e}{\e^3} \Big(\frac{\e^2}{4}-p\Big)^{-3/2}\Big).
  \end{align}
For $x = \sqrt{\frac{ \e^2}{4} - p}$, the minimizer is given as the smallest positive zero $x_+$ of $x^5 + \frac{C_\e}{\e^3} x^2 -  \frac{3C_\e}{2\e^3 k}$, which can be estimated  by
$$\sqrt{\frac{1}{k}}\lqq x_+\lqq \sqrt{\frac{3}{2k}},$$
as long as $1\lqq \Big(\frac{C_\e}{2\e^3}\Big)^{2/3}k\Leftrightarrow k\gqq \Big(\frac{2\e^3}{C_\e}\Big)^{2/3}$. Since it is readily checked that $$\Big(\frac{2\e^3}{C_\e}\Big)^{2/3}=\frac{2^{2/3}\cdot\e^2}{\Big(16(\frac{1}{\sqrt{2}}+\sqrt{\pi})(1+\frac{2}{\e^2})(e^{\frac{\e^2}{4}}-1)\Big)^{2/3}}e^{-\frac{\e^2}{6}}<1,$$ the bounds hold for all $k\gqq 1$.
Hence $p_* = p_*(k) = \max\{\frac{\e^2}{4} - x_+^2,0\Big\} \lqq \max\{\frac{\e^2}{4} - \frac{1}{k},0\}$ and $p_*(k)\gqq \max\{\frac{\e^2}{4} -\frac{3}{2k},0\}$. Plugging the upper bound of $p_*$ into \eqref{e:Dooboptimization} we obtain 
\begin{align*}
\PP(\oO_\e \gqq k)\lqq \PP(\m_\e \gqq k)\lqq e\left(1+\frac{C_\varepsilon}{\e^3} \cdot k^{\frac{3}{2}}\right)e^{-\frac{\e^2}{4}k}, \qquad k> \frac{4}{\e^2}  .
\end{align*}
This shows \eqref{e:Doob} and finishes the proof of the first part of Theorem~\ref{thm:Doob}. 

\noindent 
For the second statement we use \eqref{e:Doobadhoc} for $\e_n,$
\begin{align*}
  \PP(E_n(\e_n)) \lqq \frac{8}{\e_n} \cdot \sqrt{n} \cdot e^{-\frac{\e_n^2}{4} n}.
\end{align*}
Setting for some $\theta>0$ ,
\begin{align*}
\frac{\e_n^2}{4} n = \theta \ln(n+1)\qquad \Leftrightarrow \qquad \e_n = \sqrt{4 \theta \frac{\ln(n+1)}{n}}, 
\end{align*}
we have 
\begin{align*}
\frac{8}{\sqrt{4 \theta \frac{\ln(n+1)}{n}}} \cdot \sqrt{n} \cdot e^{-\frac{4 \theta \frac{\ln(n+1)}{n}}{4} n} 
&
\lqq \frac{4}{\sqrt{\theta \ln(n+1)}} \frac{1}{(n+1)^{\theta-1}}\lqq \frac{4}{\sqrt{\theta \ln(2)}} \frac{1}{(n+1)^{\theta-1}},
\end{align*}
which is summable for $\theta>2$. Therefore for $\oO_{(\e_n)_{n\in \NN}} := \sum_{n=1}^\infty \ind(E_n(\e_n))$ the usual first Borel-Cantelli lemma (analogously to Appendix \ref{a:fsoptimaleRate}) yields $\EE[\oO_{(\e_n)_{n\in \NN}}]\lqq \frac{4}{\sqrt{\theta \ln(2)}} \zeta(\theta-1)$ and in particular $\oO_{(\e_n)_{n\in \NN}}< \infty$, $\PP$-a.s. which shows \eqref{e:Doob2}. 
Markov's inequality yields 
\[
\PP(\oO_{(\e_n)_{n\in \NN}}\gqq k)\lqq k^{-1}\cdot \frac{4}{\sqrt{\theta \ln(2)}} \zeta(\theta-1).
\]
By the application of Example \ref{ex:Riemann}, we obtain for $-1<p<\theta-3$,
\[
\PP(\m_{(\e_n)_{n\in \NN}}\gqq k)\lqq k^{-(p+1)}\cdot \left(\frac{4}{\sqrt{\theta \ln(2)}}(\theta-2-p) \zeta(\theta-1)+\ind(\{p<0\})\right).
\]

This shows \eqref{e:Doob3} and \eqref{e:Doob4} and finishes the proof. 
\end{proof}

\medskip 

\subsection{\textbf{Locality and deviations from H\"older continuous paths}}

\subsubsection{\textbf{The Kolmogorov-Chentsov continuity theorem for stochastic processes}}\label{ss:KC}\hfill\\ 

\begin{proof}\textbf{(of Theorem~\ref{thm:KC})}
We follow the lines of the proof of \cite[2.8 Theorem]{KS98}. 
For convenience set $T = 1$. After establishing the continuity in probability 
\begin{align*}
\PP(\|X_t - X_s\|\gqq \e) \lqq C \e^{-\alpha} |t-s|^{1+\beta},
\end{align*}
on p.~54, the following discretization 
\[
t = \frac{\ell}{2^n}, \qquad s  = \frac{\ell-1}{2^n}, \qquad \mbox{ and }\qquad \e = 2^{-\gamma n}, \qquad n\in \NN_0, \quad \ell \in \{1, \dots, 2^{n}\}, 
\]
yields 
\begin{align*}
\PP(|X_{\ell/2^n}-X_{(\ell-1)/2^n}\|\gqq 2^{-\gamma n}) \lqq C 2^{-n(1+\beta -\alpha\gamma)}, 
\end{align*}
and a simple union bound estimate in \cite[$(2.9)$ on p.~54]{KS98} reads for the same constant $C$ as follows,
\begin{equation}
\PP\Big(\max_{1\lqq \ell\lqq 2^n} \|X_{\ell/2^n}-X_{(\ell-1)/2^n}\|\gqq 2^{-\gamma n}\Big) 
\lqq C 2^{-n(\beta - \alpha \gamma)}. \label{e:decay1}
\end{equation}
By the classical Borel-Cantelli lemma there is an event $\Omega^*\in \aA$ with $\PP(\Omega^*)= 1$ and a random variable $n^*:\Omega \ra\NN$ such that  
\begin{align}
\max_{1\lqq \ell\lqq 2^n} \|X_{\ell/2^n}-X_{(\ell-1)/2^n}\|< 2^{-\gamma n}\qquad \mbox{ for all }n \gqq n^*. 
\end{align}
The remainder of the proof of \cite[2.8 Theorem]{KS98} remains untouched and shows (1). 
To prove Theorem~\ref{thm:KC} (2), it remains to establish the stated moment estimates of 
\[
\m(\omega) = \sup\{n+1~|~n\in \NN_0, \omega \in A_n\},\quad A_n := \Big\{\max_{1\lqq \ell\lqq 2^n} \|X_{\ell/2^n}-X_{(\ell-1)/2^n}\|\gqq 2^{-\gamma n}\Big\},
\]
the total number and the last index of exceptions. Such moment estimates of $\m$ and $\oO= \sum_{n=0}^\infty \ind(A_n)$ have been studied in detail in \cite{EHS26}. We consider 
\[
\Big\{\max_{1\lqq \ell\lqq 2^n} \|X_{\ell/2^n}-X_{(\ell-1)/2^n}\| > 2^{-\gamma n}\Big\}, \qquad n\in \NN_0.  
\]
Therefore, we can apply the moment estimates on $\m$ in \cite{EHS26}. By Lemma~1 in \cite{EHS26} we have 
\[
\EE[\sS(\oO)]\lqq \EE[\sS(\m)] \lqq \sum_{n=0}^\infty a_n \sum_{m=n}^\infty \PP(A_m),  
\]
for any nonnegative sequence $(a_n)_{n\in \NN}$ and $\sS(N) = \sum_{n=0}^{N-1} a_n$. 
Now, for $a_n = e^{p n}$ we have $\sS(N) = \sum_{n=0}^{N-1} e^{p n} =  \frac{e^{pN} -1}{e^p-1}$. 
Therefore for $0 < p < (\beta -\alpha\gamma)\ln(2)$ we have 
\begin{align*}
\sum_{m=n}^\infty \PP(A_m) 
&\lqq C \sum_{m=n}^\infty 2^{-m(\beta - \alpha \gamma)} = \frac{C 2^{-n(\beta - \alpha \gamma)}}{1-2^{-(\beta - \alpha \gamma)}}, 
\end{align*}
such that 
\begin{align*}
\EE[\sS(\oO)] \lqq \EE[\sS(\m)] \lqq \frac{C}{1-2^{-(\beta - \alpha \gamma)}}\sum_{n=0}^\infty e^{n(p-(\beta-\alpha\gamma)\ln(2))} 
= \frac{C}{1-2^{-(\beta - \alpha \gamma)}}\frac{1}{1-e^{p-(\beta-\alpha\gamma)\ln(2)}}
\end{align*}
and finally 
\begin{align*}
\EE[e^{p \m}] 
& = \EE[\sS(\m)](e^p-1)+1\\
&\lqq  \frac{C}{1-2^{-(\beta - \alpha \gamma)}}\frac{e^p-1}{1-e^{(p-(\beta-\alpha\gamma)\ln(2))}}+1\\
&\lqq  \frac{C}{1-2^{-(\beta - \alpha \gamma)}}\frac{e^{(\beta -\alpha\gamma)\ln(2)}-1}{1-e^{(p-(\beta-\alpha\gamma)\ln(2))}}+1\\
&= \Big(\frac{M}{1-e^p b} +1\Big).
´
\end{align*}
The preceding inequality, together with Lemma~\ref{lem:optimal} from Appendix \ref{a:optimal} below for 
$M= \frac{C(2^{(\beta -\alpha\gamma)}-1)}{1-2^{-(\beta - \alpha \gamma)}}$ 
and $b = 2^{-(\beta -\alpha\gamma)}$ yields \eqref{e:Chentsovbound2} and finishes the proof of Theorem~\ref{thm:KC} 2). 
\end{proof}

\medskip

\subsubsection{\textbf{The Kolmogorov-Totoki continuity theorem for random fields}}\label{ss:KT}\hfill\\ 

\noindent We cite Lemma 4.2 in \cite{Ku04} and recall the notation $\Delta_n^\gamma(f) = 2^{n\gamma} \Delta_n(f)$, where $$\Delta_n(f) = \max_{\substack{x, y\in \lL_n\cap \dD\\ |x-y| = 2^{-n}}} \|f(x)-f(y)\|, \qquad n\in \NN_0.$$
\begin{lem}\label{lem:Kunita}
For any $f: \lL \cap \dD \ra B$ and any $\beta>0$ we have the inequality 
\begin{align*}
\|f(x) - f(y)\|\lqq 2^{d+1} \Big(\sum_{n=0}^\infty \Delta_n^\beta(f)\Big) |x-y|^\beta, \qquad x, y \in \lL\cap \dD.  
\end{align*}
\end{lem}
\noindent Hence, any $f: \lL \cap \dD\ra B$ such that $\sum_{n=0}^\infty \Delta_n^\beta(f)<\infty$ is globally $\beta$-H\"older continuous on $\lL \cap \dD$. Due to the density of $\lL\cap \dD$ in $\dD$ and the compactness of $\dD$ there is a unique uniformly continuous and $\beta$-H\"older continuous extension $\ti f:\dD \ra B$ of any such $f$ to $\bar \dD$, i.e. $\ti f(x) = f(x)$ for $x\in \lL\cap \dD$.  

\begin{proof}\textbf{(of Theorem~\ref{thm:KolmoToki})}
We only consider the case $\alpha\gqq 1$. Note that 
\begin{align*}
\big(\Delta_n^\gamma(X(\cdot))\big)^\alpha 
&\lqq \Big(\sup_{\substack{x, y \in \lL\cap \dD\\ |x-y| = \frac{1}{2^n}}} \|X(x)-X(y)\|2^{n\gamma}\Big)^\alpha 
\lqq \Big(\sum_{\substack{x', y' \in \lL\cap \dD\\ |x'-y'| = \frac{1}{2^n}}} \|X(x')-X(y')\|2^{n\gamma}\Big)^\alpha.
\end{align*}
Since $\dD\subseteq\operatorname{conv}(\dD)$, the number of summands in the preceding sum
is bounded by $\mbox{vol}(\operatorname{conv}(\dD)) / \mbox{vol}([0,1]^d) \cdot 2^{d(n+1)}$. 
Hence by \eqref{e:KolmoToki} we have 
\begin{align}
\EE\Big[\Delta_n^\gamma(X(\cdot))^\alpha\Big] 
&\lqq \mbox{vol}(\operatorname{conv}(\dD))\cdot 2^{(n+1) d}\cdot 2^{n\alpha\gamma} 
\sup_{\substack{x, y \in \lL\cap \dD\\ |x-y| = \frac{1}{2^n}}} \EE[\| X(x)-X(y)\|^\alpha] \nonumber\\
&\lqq \mbox{vol}(\operatorname{conv}(\dD))\cdot 2^{(n+1) d}\cdot 2^{n\alpha\gamma}\cdot 2^{-n(d+\beta)}
= 2^d \cdot \mbox{vol}(\operatorname{conv}(\dD)) \cdot 2^{n (\alpha\gamma -\beta)}.\label{e:Totokimoment}
\end{align}
Finally, combining $\gamma \in(0, \frac{\beta}{\alpha})$ with Lemma~\ref{lem:Kunita} we have 
\begin{align*}
\EE\Big[\big(\sum_{n=0}^\infty \Delta_n^\gamma(X(\cdot)\big)^\alpha]^\frac{1}{\alpha} 
&\lqq \sum_{n=0}^\infty \EE\Big[\big( \Delta_n^\gamma(X(\cdot)\big)^\alpha]^\frac{1}{\alpha}
\lqq \Big(2^d \mbox{vol}(\operatorname{conv}(\dD))\Big)^\frac{1}{\alpha} \cdot \sum_{n=0}^\infty 2^{-n (\frac{\beta}{\alpha} -\gamma)} <\infty. 
\end{align*}
Consequently 
\[
\sum_{n=0}^\infty \Delta_n^\beta(X(\cdot) < \infty \qquad \PP\mbox{-a.s.}
\]
hence is $X$ has a $\gamma$-H\"older continuous version $\ti X$. The Markov inequality combined with \eqref{e:Totokimoment} we have 
\begin{align*}
\PP(\Delta_n^\gamma(X) > 2^{-n\delta}) 
\lqq 2^{n\delta \alpha}\cdot \EE[(\Delta_n^\gamma(X))^\alpha]
\lqq 2^d \mbox{vol}(\operatorname{conv}(\dD)) \cdot 2^{-n (\beta -\alpha \gamma-\delta \alpha)}.
\end{align*}
Then, by Example~1 in \cite{EHS26} we have for all $0 \lqq p < (\beta -\alpha \gamma-\delta \alpha)\ln(2)$ \[
\EE[e^{p\oO_\gamma}]\lqq \EE[e^{p\m_\gamma}] \lqq K_p, \qquad 
K_p:= \frac{2^{d+\beta -\alpha \gamma-\delta \alpha} \mbox{vol}(\operatorname{conv}(\dD))}{1-e^p \cdot 2^{-(\beta -\alpha \gamma-\delta \alpha)}} +1 < \infty.  
\]
Consequently, 
\[
\PP(\oO_\gamma\gqq k) \lqq \PP(\m_\gamma\gqq k)\lqq \inf_{p\in [0, (\beta -\alpha \gamma-\delta \alpha)\ln(2))}K_p \cdot e^{-pk}
, \qquad k\in \NN. 
\]
An application of Lemma~\ref{lem:optimal} from Appendix \ref{a:optimal} below with $M =2^{d+\beta -\alpha \gamma-\delta \alpha}\, \mbox{vol}(\operatorname{conv}(\dD))$ finishes the proof. 
\end{proof}

\bigskip 

\section{\textbf{Proof of: Deviation frequencies of fine continuity properties}} 
 \subsection{\textbf{L\'evy's modulus of continuity}}\label{ss:Lmcproof} \hfill\\
 
\begin{proof}\textbf{(of Theorem~\ref{thm:Levymodulus})}
The proof of the first statement is given in \cite{KS98} and the upper bound boils down to the application of the classical Borel-Cantelli lemma to event probabilities 
\begin{align*}
\PP\big(A_n(\theta)\big)\lqq \exp(-\kappa_n  e^n),  \mbox{ for all }n\in \NN, 
\end{align*}
where 
\begin{align*}
&A_n(\theta) := \Big\{\max_{1\lqq j\lqq \lfloor e^{n}\rfloor} |W_{\frac{j}{e^n}}- W_{\frac{j-1}{e^n}}| \lqq \sqrt{1-\theta} \mu(e^{-n}) \Big\}, \qquad \theta \in (0,1), \qquad \mbox{ and }\\
&\kappa_n := \kappa_n(\theta) := 2 \PP\Big(e^{n/2} W_\frac{1}{e^n} > \sqrt{1-\theta} e^{n/2} \mu(e^{-n})\Big).
\end{align*}
With the help of Mill's ratio \cite{BS79} and $\frac{x}{1+x^2}\gqq \frac{1}{2x}$ for $x\gqq 1$, is easy to see that 
\begin{align*}
\kappa_n & \gqq 2\frac{e^{-x^2/2}}{\sqrt{2\pi} } \frac{x}{1+x^2}, \qquad \mbox{ for }\quad x = \sqrt{(1-\theta) 2  n}, (\text{which is satisfied for }n\gqq\tfrac{1}{2(1-\theta)}\text{)}
\end{align*}
that is, 
\begin{align*}
\kappa_n &\gqq 2 \frac{e^{-(1-\theta) n}}{\sqrt{2\pi}} \frac{\sqrt{(1-\theta) 2  n}}{(1-\theta) 2  n +1} 
\gqq \frac{e^{-(1-\theta) n}}{\sqrt{2\pi}} \frac{1}{\sqrt{(1-\theta) 2  n}}.
\end{align*}
Hence for 
$\alpha = \alpha_\theta = \tfrac{1}{\sqrt{4\pi (1-\theta)}}$ we have 
\begin{align*}
\PP\Big(A_n(\theta)\Big)\lqq \exp\Big(- \alpha \frac{e^{n\theta}}{\sqrt{n}}\Big), \qquad n\gqq\frac{1}{2(1-\theta)}. 
\end{align*}
Using the quantitative version of the Borel-Cantelli lemma given by  \cite[Theorem 1]{EH22}, instead of the original one, we calculate 
with the help of the integral test 
\begin{align*}
\sum_{m = n}^\infty \PP(A_m) 
&\lqq  \sum_{m = n}^\infty \exp\Big(-\al \frac{(e^{m})^\theta}{\sqrt{m}}\Big)\lqq  \int_{n-1}^\infty \exp\Big(-\al \frac{(e^{x})^\theta}{\sqrt{x}}\Big) dx = \int_{e^{n-1}}^\infty \exp\Big(-\al \frac{y^\theta}{\sqrt{\ln(y)}}\Big) \frac{dy}{y}.
\end{align*}
For any $0 < \eta <\theta$ there exists $M_{\theta,\eta}>0$ such that  for $y\gqq M_{\theta,\eta}$ we have the inequality $\exp\big(-\al \tfrac{y^\theta}{\sqrt{\ln(y)}}\big)\lqq\exp(-\al y^{\eta})$, and thus, for $n$ such that $e^{n-1}\gqq M_{\theta,\eta}$,
\begin{align*}
\sum_{m = n}^\infty \PP(A_m) 
&\lqq \int_{e^{n-1}}^\infty \exp\big(-\al y^{\eta}\big) \frac{dy}{y}.
\end{align*}
The substitution $t = \alpha y^{\eta}$ yields $y = \big(\frac{t}{\alpha}\big)^\frac{1}{\eta}$ and hence 
$\frac{dy}{dt} = \frac{d\big(\frac{t}{\alpha}\big)^\frac{1}{\eta}}{dt} = \frac{1}{\alpha^{1/\eta}} \frac{1}{\eta} t^{\frac{1}{\eta}-1}$ with the lower bound $t = \alpha e^{\eta (n-1)}$. Thus, we find a constant $\kK_{\eta}>0$ such that for all $n\gqq\frac{1}{2(1-\theta)}$ with $e^n \gqq M_{\theta, \eta}$ (or $n\gqq \ln(M_{\theta,\eta}) $) we have 
\begin{align*}
\sum_{m = n}^\infty \PP(A_m) 
&\lqq \int_{ e^{n-1}}^\infty \exp\Big(-\al y^{\eta}\Big) \frac{dy}{y} = \int_{\alpha e^{\eta (n-1)}}^\infty \big(\frac{t}{\alpha}\big)^{-\frac{1}{\eta}}\exp\Big(-t\Big)\frac{1}{\alpha^{1/\eta}} \frac{1}{\eta} t^{\frac{1}{\eta}-1} dt\\
&= \frac{1}{\eta} \int_{\alpha e^{\eta (n-1)}}^\infty \frac{1}{t} e^{-t} dt \lqq \frac{1}{\eta} \frac{e^{-\eta (n-1)}}{\al}  \exp\Big( - \al e^{\eta(n-1)}\Big)\\
&\lqq \kK_\eta \cdot \exp(-\eta n) \cdot 
\exp\Big( - \frac{\alpha}{e^{\eta}}\exp(\eta n)\Big). 
\end{align*} 
Setting 
\begin{align*}
\oO_{\theta} = \sum_{n=1}^\infty \ind(A_n(\theta)), \qquad \m_{\theta}(\omega)=\sup\Big\{n -\lceil\tfrac{1}{2(1-\theta)}\rceil+1~\vert~n\in \NN_0, \omega\in A_n(\theta)\},
\end{align*}
we have by Lemma~1 in \cite{EHS26}, that for all $0 \lqq p < \frac{\al}{e^{\eta}}$ there is a constant $L_{\theta, \eta, p}>0$ such that with $N_{\theta,\eta}:=\lceil \ln(M_{\theta, \eta})\vee \frac{1}{2(1-\theta)}\rceil$,
\begin{align*}
&\EE\Big[\exp\Big( p\exp(\eta(\oO_\theta))\Big)\Big]
\lqq \EE\Big[\exp\Big( p\exp(\eta(\m_\theta))\Big)\Big]\\
&\qquad \lqq \sum_{n=1}^\infty p
\exp(\eta n) \cdot 
\exp\big( p \exp(\eta n)\big) \sum_{m = n}^\infty \PP(A_m)\\
&\qquad \lqq L_{\theta, \eta, p} + \sum_{n=N_{\theta,\eta}}^\infty p
\exp(\eta n) \cdot 
\exp\big( p \exp(\eta n)\big) \sum_{m = n}^\infty \PP(A_m)\\
&\qquad \lqq L_{\theta, \eta, p} + p \kK_{\eta} \sum_{n=N_{\theta,\eta}}^\infty 
\exp(\theta n) \cdot 
\exp\big( p \exp(\eta n)\big)\cdot \exp(-\eta n) \cdot 
\exp\Big( - \frac{\alpha}{e^{\eta}}\exp(\eta n)\Big) \\
&\qquad = L_{\theta, \eta, p} +p \kK_{\eta} \sum_{n=N_{\theta,\eta}}^\infty 
\exp\Big( \big(p  - \frac{\alpha}{e^{\eta}}\big)\exp(\eta n)\Big) =:  \kK_{\theta, \eta, p} < \infty. 
\end{align*}
Consequently, for all $0 \lqq p < \frac{\al}{e^{\eta}}$ and $k\gqq 1$ we have 
\begin{align*}
\PP(\oO_{\theta}\gqq k) \lqq \PP(\m_{\theta}\gqq k)
&\lqq  \kK_{\theta,\eta, p}\cdot \exp( -p \exp(\eta k)).
\end{align*}
For the proof of the second statement we follow the lines of \cite{KS98},  p.~115. 
There it is shown, that for any $\theta \in (0,1)$ and $1+\e> \frac{1+\theta}{1-\theta}$ 
\begin{align*}
&\PP\Bigg(\max_{\substack{0\lqq i < j \lqq \lfloor e^n\rfloor \\ k= j-i \lqq \lceil e^{n\theta}\rceil }} 
\frac{|W_{\frac{j}{e^n}} -W_{\frac{i}{e^n}}|}{\mu(\frac{k}{e^n})} \gqq 1+ \e\Bigg)
\lqq \sum_{k=1}^{\lceil e^{\theta n}\rceil} \PP\Big(\max_{0\lqq i < i+k \lqq \lfloor e^n\rfloor} 
|W_{\frac{k+i}{e^n}} -W_{\frac{i}{e^n}}|\gqq \mu\Big(\frac{k}{e^n}\Big)\Big)\\
&\lqq \lfloor e^n\rfloor \sum_{k=1}^{\lceil e^{\theta n}\rceil} \PP\Bigg(\frac{|W_{\frac{k}{e^n}}|}{\sqrt{\frac{k}{e^n}}}\gqq (1+\e) \sqrt{\ln\Big(\frac{e^{2n}}{k^2}\Big)}\Bigg)
\lqq 2 \lfloor e^n\rfloor \sum_{k=1}^{\lceil e^{\theta n}\rceil} \PP\bigg(W_{1}\gqq (1+\e) \sqrt{\ln\Big(\frac{e^{2n}}{k^2}\Big)}\bigg)\\
&\lqq \frac{2}{\sqrt{2\pi}} \lfloor e^n\rfloor \sum_{k=1}^{\lceil e^{\theta n}\rceil}\frac{\exp\Big(-\frac{1}{2}\Big[(1+\e) \sqrt{\ln(\frac{e^{2n}}{k^2})}\Big]^2\Big)}{(1+\e) \sqrt{\ln(\frac{e^{2n}}{k^2})}}  
= \sqrt{\frac{2}{\pi}} \lfloor e^n\rfloor\sum_{k=1}^{\lceil e^{\theta n}\rceil}\frac{e^{-n(1+\e)^2}k^{(1+\e)^2}}{(1+\e)\sqrt{\ln(\frac{e^{2n}}{k^2})}}\\
&\lqq \frac{1}{\sqrt{\pi}} \lfloor e^n\rfloor \frac{e^{-n(1+\e)^2}}{1+\e}  \sum_{k=1}^{\lceil e^{\theta n}\rceil} k^{(1+\e)^2}\qquad \mbox{ for } e^n / \lceil e^{\theta n}\rceil > e.
\end{align*}
Now, $\frac{e^n}{\lceil e^{\theta n}\rceil} \gqq \frac{e^n}{e^{\theta n}(1+\theta)} = e^{(1-\theta)n-\ln(1+\theta)} \gqq e$, the last inequality is true, iff $n \gqq \frac{1+\ln(1+\theta)}{1-\theta}$.
By integral comparison it is easily seen that 
\begin{align*}
\sum_{k=1}^{\lceil e^{\theta n}\rceil} k^{(1+\e)^2} 
&\lqq \int_0^{\lceil e^{\theta n}\rceil+
1} x^{(1+\e)^2} dx \lqq \frac{(\lceil e^{\theta n}\rceil+1)^{1+ (1+\e)^2}}{1+ (1+\e)^2} \lqq \frac{2^{1+ (1+\e)^2}}{1+ (1+\e)^2} \lceil e^{\theta n}\rceil^{1+ (1+\e)^2}.  
\end{align*}
Therefore for all $n\gqq \frac{1+\ln(1+\theta)}{1-\theta}$ we have 
\begin{align*}
&\PP\Bigg(\max_{\substack{0\lqq i < j \lqq \lfloor e^n\rfloor \\ k= j-i \lqq \lceil e^{n\theta}\rceil }} 
\frac{|W_{\frac{j}{e^n}} -W_{\frac{i}{e^n}}|}{\mu(\frac{k}{e^n})} \gqq 1+ \e\Bigg) 
\lqq  \frac{1}{\sqrt{\pi}} \frac{2^{1+ (1+\e)^2}}{1+ (1+\e)^2} \lfloor e^n\rfloor \frac{e^{-n(1+\e)^2}}{1+\e}\lceil e^{\theta n}\rceil^{1+ (1+\e)^2}\\
&\lqq  \frac{1}{\sqrt{\pi}}  \frac{2^{1+ (1+\e)^2}}{1+ (1+\e)^2} (e^n+1) \frac{e^{-n(1+\e)^2}}{1+\e}( e^{\theta n}+1)^{1+ (1+\e)^2}
\lqq  \frac{1}{\sqrt{\pi}}  \frac{8^{1+ (1+\e)^2}}{1+ (1+\e)^2} e^n \frac{e^{-n(1+\e)^2}}{1+\e} e^{\theta (1+ (1+\e)^2) n}\\
&\lqq \Big(\frac{8^{1+ (1+\e)^2}}{\sqrt{\pi}(1+\e)(1+ (1+\e)^2)}\Big) e^{-((1-\theta)(1+\e)^2-(1+\theta)) n}, 
 \end{align*}
and hence for $K_{\e} := \frac{8^{1+ (1+\e)^2}}{\sqrt{\pi}(1+\e)(1+ (1+\e)^2)}$ and $\rho = (1-\theta)(1+\e)^2-(1+\theta)$ we have 
\begin{align*}
\PP\Bigg(\max_{\substack{0\lqq i < j \lqq \lfloor e^n\rfloor \\ 1\lqq k= j-i \lqq \lceil e^{n\theta}\rceil }} 
\frac{|W_{\frac{j}{e^n}} -W_{\frac{i}{e^n}}|}{\mu(k/e^n)} \gqq 1+ \e\Bigg) 
&\lqq K_{\e} e^{-\rho n}, \qquad n\in \NN.  
\end{align*}
Consequently, by hypothesis, for all $\theta \in (0,1)$, $\e>0$ is chosen such that $\rho>0$. So for any fixed $n\gqq \frac{1+\ln(1+\theta)}{1-\theta}$ we have with
\begin{align*}
\oO_{\theta, \e} &= \sum_{n = \lceil \frac{1+\ln(1+\theta)}{1-\theta}\rceil}^\infty \ind\bigg\{\max_{\substack{0\lqq i < j \lqq \lfloor e^n\rfloor \\ 1\lqq k= j-i \lqq \lceil e^{n\theta}\rceil }} 
\frac{|W_{\frac{j}{e^n}} -W_{\frac{i}{e^n}}|}{\mu(k/e^n)} \gqq 1+ \e\bigg\},\\
\m_{\theta, \e} &= \sup\bigg\{n-\Big\lceil \frac{1+\ln(1+\theta)}{1-\theta}\Big\rceil+1~|~ 
\max_{\substack{0\lqq i < j \lqq \lfloor e^n\rfloor \\ 1\lqq k= j-i \lqq \lceil e^{n\theta}\rceil }} 
\frac{|W_{\frac{j}{e^n}} -W_{\frac{i}{e^n}}|}{\mu(k/e^n)} \gqq 1+ \e\bigg\},
\end{align*}
that, by Lemma~\ref{lem:Alexp} we have for all $0 < p < \rho$,
\begin{align*}
\EE[e^{p \oO_{\theta, \e}}] \lqq \EE[e^{p \m_{\theta, \e}}] \lqq 
1+\frac{K_{\e} e^{-\rho (\lceil\frac{1+\ln(1+\theta)}{1-\theta}\rceil -1)}}{1-e^{-(\rho-p)}} <\infty 
\end{align*}
and thus by Corollary~\ref{cor:Alexp} we have 
\begin{align*}
\PP(\oO_{\theta, \e}\gqq k)\lqq \PP(\m_{\theta, \e}\gqq k)\lqq e^{\frac{5}{4}} \cdot (2k+1) \Big(\frac{K_{\e}}{1-e^{-\rho}} e^{-\rho(\lceil \frac{1+\ln(1+\theta)}{1-\theta}\rceil-1)} +1\Big) \cdot e^{-\rho k}, \qquad k\gqq 1. 
\end{align*} 
This shows \eqref{e:Lmclower} and finishes the proof.

\end{proof}

\medskip

\subsection{\textbf{The quantitative blow up of Brownian secant slopes: Paley, Wiener and Zygmund}}\label{ss:PWZ}\hfill\\

\begin{proof}\textbf{(of Theorem~\ref{thm:PWZ})} The proof of the first statement, \cite{KS98}, p. 110, remains intact. 
For the second part we follow the version by \cite{Ko07}, p. 169. For any $\la>0$ and $n\in \NN_0$ define 
\begin{align*}
E^n_\la := 
\Bigg\{\exists\,s\in [0, 1]~\bigg\vert~\sup_{t \in [s-2^{-n}, s+2^{-n}]\cap [0,1]} \frac{|W_s-W_t|}{2^{-n}} \lqq \lambda \Bigg\}.  
\end{align*}
Then it is shown there, combining formula (9.31) and (9.32), that 
\begin{align*}
\PP(E^n_\la) 
&\lqq 2^{n} \Big(\int_{-\la 2^{-n/2+2}}^{\la 2^{-n/2 +2}} \frac{e^{-x^2/2}}{\sqrt{2\pi}} dx\Big)^4
\lqq 2^{n} \Big(\frac{2}{\sqrt{2\pi}} \la 2^{-n/2 +2}\Big)^4 \\
&= 2^{n} \la^4 \Big(\frac{8}{\sqrt{2\pi}}\Big)^4  2^{-2n}
= \Big(\frac{8}{\sqrt{2\pi}}\Big)^4 \la^4 2^{-n} = \Big(\frac{1024}{\pi^2}\Big) \la^4 2^{-n}.
\end{align*}
Then for $\oO_\la := \sum_{n=0}^\infty \ind(E^n_{\la})$ and $\m_\la(\omega) := \sup\{n+1~|~\omega\in E^n_{\la}\}$,
and any $0 < r < \ln(2)$, Lemma~\ref{lem:Alexp} 
and Corollary~\ref{cor:Alexp} yield for $c_\pi := \frac{1024}{\pi^2}$ 
\begin{align*}
\EE[e^{r\oO_\la}] \lqq\EE[e^{r\m_\la}] \lqq 1+ \frac{2c_\pi \la^4 }{1-e^r \frac{1}{2}}
\end{align*}
and 
\begin{align*}
\PP(\oO_\la\gqq k) \lqq \PP(\m_\la\gqq k) 
\lqq e^{\frac{5}{4}} \cdot (2k+1) (2 c_\pi \lambda^4+1) \cdot 2^{-k}, \qquad k\gqq 1. 
\end{align*}
For the special case of $\la = \la_n = R^n$ for some $1<R<2^{1/4}$ we have 
\begin{align*}
\PP(E^n_\la) \lqq c_\pi(R^4/2)^{n}, \qquad n \in \NN_0, 
\end{align*}
and 
$\oO_R = \sum_{n=0}^\infty \ind(E^n_{\la_n})$ and $\m_R = \sup\{n+1~|~\om\in E^n_{\la_n}\}$, and for any $0 < r < \ln(2/R^4)$
\begin{align*}
\EE[e^{r \oO_R}] \lqq \EE[e^{r \m_R}]\lqq \frac{2}{R^4}\frac{c_\pi}{(1-e^r R^4/2)} +1,
\end{align*}
and for $k\in \NN$,
\begin{align*}
\PP(\oO_R\gqq k)\lqq \PP(\m_R\gqq k)
&\lqq e^{\frac{5}{4}}\cdot (2k+1) \Big(\frac{2 c_\pi}{R^4}+1\Big) \cdot \Big(\frac{R^4}{2}\Big)^{k}. 
\end{align*}
This finishes the proof. 
\end{proof}

\medskip

\subsection{\textbf{The quantitative loss of monotonicity in Brownian paths}}\label{ss:QLM}\hfill\\

\begin{proof}\textbf{(of Theorem~\ref{thm: monotonicity})} Instead of investigating an arbitrary interval, we will show the non-monotonicity on $[0,1]$. By the self-similarity of Brownian motion in distribution, this follows for all intervals. It also suffices to look at monotone increase only (as the case for the decrease works the same). We consider 
$E := \{\om \in \Omega~\vert~W (\omega) \mbox{ is nondecreasing on }[0, 1]\}$ and 
note that 
\begin{align*}
E = \bigcap_{n=1}^\infty E_{n}, \qquad \mbox{ for }\quad E_n = \bigcap_{i=0}^{n-1 }\{\om \in \Omega~\vert~W_{\frac{i+1}{n}} -  W_{\frac{i}{n}}\gqq 0\}.
\end{align*}
By the independence and stationarity of the increments we have $\PP(E_n) = 2^{-n}$. 
If we denote $\oO = \sum_{n=1}^\infty \ind_{E_n}$ and $\m(\omega) = \sup\{n-1~|~n\in \NN, \om\in E_n\}$, Lemma~\ref{lem:Alexp} implies 
for all $0 \lqq p < \ln(2)$ that 
\[
\EE[e^{p \oO}]\lqq \EE[e^{p \m}] \lqq  \frac{2}{2-e^p}+1.  
\]
Therefore Markov's inequality and Corollary~\ref{cor:Alexp} yield 
\begin{align*}
\PP(\oO\gqq k)\lqq \PP(\m\gqq k)\lqq  \inf_{p\in [0, \ln(2))} e^{-pk}\Big(\frac{1}{1-e^p/2}+1\Big) 
\lqq e^{\frac{5}{4}} \cdot 2(2k +1) \cdot 2^{-k}, \qquad k\in \NN.
\end{align*}
This finishes the proof. 
\end{proof}

\medskip 

\subsection{\textbf{The a.s. convergence to the quadratic variation}}\label{ss:QV}\hfill\\

\begin{proof}\textbf{(of Theorem~\ref{thm:QV1})} The first statement is shown in \cite[9.4 Theorem]{SPB14}. 
The first display of the proof in \cite[p.140]{SPB14} reads 
\begin{align*}
\EE\Big[\Big(\sum_{t_i\in \Pi_n(t)} (W_{t_i}-W_{t_{i-1}})^2 - t\Big)^2\Big] \lqq 2 |\Pi_n(t)|\,t.  
\end{align*}
Hence by Chebyshev's inequality we have that for all $\e>0$ and $m\in \NN$ 
\begin{equation}\label{e:Verteilungschwanz}
\sum_{n=m}^\infty \PP\Big(\Big|\sum_{t_i\in \Pi_n(t)} (W_{t_i}-W_{t_{i-1}})^2 - t\Big|>\e\Big) \lqq \sum_{n=m}^\infty \frac{2t}{\e^2} |\Pi_n(t)|.
\end{equation}
Hence by \cite[Lemma~1]{EHS26} we have that 
\begin{align*}
\EE[\sS(\oO_\e(t))] \lqq \EE[\sS(\m_\e(t))] \lqq\frac{2t}{\e^2} \sum_{m=1}^\infty a_m  \sum_{n=m}^\infty   |\Pi_n(t)|, 
\end{align*}
and its right-hand side is finite by assumption. 

For the second statement we use \eqref{e:Verteilungschwanz} for any $\theta>1$ and with $\e_n := 
\sqrt{2t n^\theta |\Pi_n(t)|}$ we obtain 
\begin{equation}\label{e:Verteilungschwanz0}
\PP\Big(\Big|\sum_{t_i\in \Pi_n(t)} (W_{t_i}-W_{t_{i-1}})^2 - t\Big|>\e_n\Big) \lqq \frac{2t}{\e_n^2} |\Pi_n(t)| \lqq  \frac{1}{n^\theta}
\end{equation}
Hence with the same reasoning as in Appendix~\ref{a:fsoptimaleRate} , 
the usual first Borel-Cantelli lemma combined with Markov's inequality yields 
\[
\limsup_{n\ra\infty} \Big|\sum_{t_i\in \Pi_n(t)} (W_{t_i}-W_{t_{i-1}})^2 - t\Big| \cdot \e_n^{-1}\lqq 1 \qquad \PP-\mbox{a.s.} 
\]
and for all $k\gqq 1$ we have 
\[
\PP\bigg(\#\Big\{n\in \NN~|~ \Big|\sum_{t_i\in \Pi_n(t)} (W_{t_i}-W_{t_{i-1}})^2 - t\Big|>\e_n\Big\}\gqq k\bigg)\lqq k^{-1} \cdot \zeta(\theta). 
\]
Finally, equation \eqref{e:QVbarely} follows from Example \ref{ex:Riemann}, with $q$ equalling $\theta$.
\end{proof}

\begin{proof}\textbf{(of Theorem~\ref{thm:QV2})}
In \cite[Proof of 9.4 Theorem, p.~141]{SPB14} the last display of the page reads as follows: For all $\e>0$, $0 < \la < \frac{1}{2}$ it follows that for all $n\in \NN$,
\begin{align*}
\PP\Big(\Big|\sum_{i=0}^{k_n-1} (W_{t_{i+1}}-W_{t_i})^2-t\Big|>\e\Big)\lqq 2\exp\Big(-\frac{\e \la}{2 |\Pi_n(t)|}\Big).     
\end{align*}
Since by assumption $K_2(t, \e, \la) < \infty$ \cite[Lemma 1]{EHS26} implies 
\begin{align*}
\EE[\sS(\oO_\e(t))] \lqq \EE[\sS(\m_\e(t))]\lqq K_2(t, \e, \la).  
\end{align*}
\end{proof}

\section{\textbf{Proof of: Deviation frequencies in the laws of the iterated Logaritm}}
\subsection{\textbf{Upcrossing frequencies in Khinchin's law of the iterated logarithm}}\label{ss:LIL}\hfill\\

 \begin{proof}\textbf{(of Theorem~\ref{thm: lil})}
 The proof of \eqref{e:lilobben} in \cite[p.~112]{KS98} remains untouched. We define 
$g(s) := \sqrt{2 s \ln(\ln(1/s))}$. Moreover, it is shown there, that for $\theta\in (0,1),$
 \begin{align*}
 \PP\Big(\max_{0\lqq s\lqq \theta^{n}} \Big(W_s - \frac{(1+\delta) \theta^{-n} g(\theta^n)s}{2}\Big) \gqq 
 \frac{1}{2} g(\theta^n)\Big) \lqq \frac{1}{(n \ln(1/\theta))^{1+\delta}}, \qquad n\gqq 1.  
 \end{align*}
The right-hand side is summable. For $\delta >0$, by Example~\ref{ex:Riemann} 
we have for all $-1 <p < \delta -1$ 
\begin{align*}
\PP(\oO_{\delta, \theta}\gqq k)\lqq \PP(\m_{\delta, \theta}\gqq k)
&\lqq \frac{1}{\ln(1/\theta)^{1+\delta}} \frac{(1+\delta) \zeta(\delta-p)}{k^{p+1}}+\frac{\ind(\{p<0\})}{k^{p+1}}, 
\end{align*}
where $\zeta(s) = \sum_{n=1}^\infty n^{-s}$ is Riemann's zeta function.  
Additionally, for $\delta>0$, the usual Borel-Cantelli lemma yields 
\[
\PP(\oO_{\delta, \theta}\gqq k)\lqq \frac{1}{\ln(1/\theta)^{1+\delta}} \frac{\zeta(1+\delta)}{k}.
\]
Optimizing the expression \eqref{e:lilquant} in $p\in(-1,\delta-1)$, following the lines of \cite[Example 1]{EHS26}, one obtains the statements last chain of inequalities.

\end{proof}

\medskip 
\subsection{\textbf{Downcrossing frequencies in Chung's ``other'' law of the iterated logarithm }}\label{ss:theotherLIL}

\begin{proof}\textbf{(of Theorem~\ref{thm:Chung})} 
In \cite[p.~169, second display from above]{SPB14} the authors obtain 
\begin{align*}
\PP{\Big(\sup_{s\in [0, 1]} |W_s| <x\Big)} &= \frac{4}{\pi} \sum_{k=0}^\infty 
\frac{(-1)^k}{2k+1} e^{-\frac{\pi^2 (2k+1)^2}{8x^2}}.  
\end{align*}
Such that 
\begin{align*}
A_n :=\Big\{\sup_{s\in [0, q^n]} |W_s| < (1-\e) \frac{\pi}{\sqrt{8}} \sqrt{\frac{q^n}{\ln(\ln(q^n))}}\Big\} 
\end{align*}
that by \cite[p.169, second formula display]{SPB14}
\begin{align*}
\PP(A_n) &= \PP\Big(  \sup_{s\in [0, q^n]} |W_s| < (1-\e) \frac{\pi}{\sqrt{8}} \sqrt{\frac{q^n}{\ln(\ln(q^n))}}\Big)\\
&= \PP{\bigg(  \sup_{s\in [0, 1]} |W_s| < (1-\e) \frac{\pi}{\sqrt{8}} \sqrt{\frac{1}{\ln(\ln(q^n))}}\bigg)} \\
&= \frac{4}{\pi} \sum_{k=0}^\infty 
\frac{(-1)^k}{2k+1} e^{-\frac{\pi^2 (2k+1)^2}{8} \frac{\ln(\ln(q^n))}{(1-\e)^2 (\frac{\pi}{\sqrt{8}})^2}}=\frac{4}{\pi} \sum_{k=0}^\infty 
\frac{(-1)^k}{2k+1} \Big(\frac{1}{n \ln(q)}\Big)^{\frac{(2k+1)^2}{(1-\e)^2}}.
\end{align*}

Since the series is decreasing, it is bounded by its initial term, showing $\PP(A_n)
\lqq\frac{4}{\pi}  \Big(\frac{1}{n \ln(q)}\Big)^{\frac{1}{(1-\e)^2}}$.
Consequently, by Example~\ref{ex:Riemann} we have for all $-1 < p < \frac{1}{(1-\e)^2}-2$, 
\begin{align*}
&\PP\bigg(\#\Big\{n\gqq n_q~|~\sup_{s\in [0, q^n]} |W_s| < (1-\e) \frac{\pi}{\sqrt{8}} \sqrt{\frac{q^n}{\ln(\ln(q^n))}}\Big\}\gqq k\bigg)\\
&\qquad \lqq \max\Big\{\frac{4}{\pi \ln(q)^{\frac{1}{(1-\e)^2}}}\!\cdot\!\frac{\zeta\Big(\frac{1}{(1-\e)^2}-p-1\Big)}{(1-\e)^2 k^{p+1}}\!+\!\frac{\ind(\{p<0\})}{k^{p+1}}, \frac{4}{\pi \ln(q)^{\frac{1}{(1-\e)^2}}}\!\cdot\!\frac{\zeta\Big(\frac{1}{(1-\e)^2}\Big)}{k}\Big\}. 
\end{align*}
In particular, $p>1$ is satisfied for $\e> 1- \frac{1}{\sqrt{2}} \approx 0.2929$.   The last chain of equations follows the lines of \cite[Example 1]{EHS26} for the case $n_0=1, q=\frac{1}{(1-\e)^2}$ and $c=\frac{4}{\pi\ln^{\frac{1}{(1-\e)^2}}(q)}$ using the lower bound $-1$ for the optimizing $p$.
\end{proof}

\medskip 
\subsection{\textbf{Strassen's functional law of the iterated logarithm}}\label{ss:Strassen}

\begin{proof}\textbf{(of Theorem~\ref{thm:Strassen})} The statement of item (1) is worked out in detail in \cite[Subsection 12.3]{SPB14}. 
We continue with item (2) and (3)(a): In \cite[Proof of Lemma 12.15, p.~187]{SPB14} the authors obtain for any $\e>0$, $\eta>0$ and $0 < \vartheta < \eta$ the estimate 
\begin{align}\label{e:Schilder}
\PP(d(Z_{q^n}(\cdot, \cdot),\kK(\tfrac{1}{2}+\eta))>\e)
&\lqq \exp\Big(- 2\Big(\tfrac{1}{2}+\vartheta\Big)\ln(\ln(q^n))\Big)\nonumber\\
&= \exp\Big(- \Big(1+2\vartheta\Big)(\ln(n) + \ln(\ln(q))\Big)\nonumber\\
&= \frac{1}{\ln(q)^{1+2\vartheta}} \frac{1}{n^{1+2\vartheta}}, \qquad n\gqq n_0,  
\end{align}
 for some $n_0 = n_0(\e, \eta, \vartheta,q)$ (where the authors 
in reference \cite[p.~187]{SPB14} use $\vartheta$ expressed as a difference '$\eta-\gamma$'). 
Hence there exists a constant $a = a(\e, \eta, \vartheta,q)>0$ such that 
\begin{align}\label{e:Schilder7}
\PP(d(Z_{q^n}(\cdot, \cdot),\kK(\tfrac{1}{2}+\eta))>\e)\lqq a \frac{1}{\ln(q)^{1+2\vartheta}} \frac{1}{n^{1+2\vartheta}} \qquad \mbox{ for all }n\in \NN. 
\end{align}
The usual first Borel-Cantelli lemma combined with Markov's inequality yields item (2). 

\noindent We continue with item (3)(a). In the sequel we use $\eta>\vartheta > \frac{1}{2}$ in order to calculate an explicit constant $b$, which takes the role of $a$.  
In the proof of Schilder's Theorem (1966) given in \cite[Proof of Lemma 12.10]{SPB14}, it is also specified how big this $n_0$ must be: It must be such that for $n\gqq n_0$ and $(2\ln(\ln(q^n)))^{-\frac{1}{2}}=:\tilde{\epsilon}$ the following three conditions are satisfied, where $r_0:= \frac{1}{2}+\eta$ and $m:=\lfloor\frac{8r_0}{\e^2}\rfloor+1$:
\begin{enumerate}[(i)]
\item $\tilde \epsilon\lqq\sqrt{r_0}=\sqrt{\frac{1}{2}+\eta}$.
\item $\tilde \epsilon\lqq \sqrt{\frac{\frac{1}{2}+\vartheta}{1+m+\ln(2)}}$ (obtained by estimating the probabilites of the proof's sets $A_n$ using $\alpha=\frac{\tilde\epsilon^2}{r_0}$, which is in $(0,1)$ by the above condition as $r_0=\frac{1}{2}+\eta>1$ by our assumption on $\eta$).
\item $\tilde \epsilon\lqq\sqrt{\frac{\frac{1}{2}+\vartheta}{\ln\big(\frac{8m^{3/2}}{\sqrt{2\pi}\e}\big)+\ln(2)}}$ (obtained by estimating the probabilites of the sets $C_n$ in the proof, where it is also stated that they need a large enough number $m$ such that for $n\geq m$ we also get $n\geq \frac{8r_0}{\e^2}$, hence our special choice of $m$).
\end{enumerate}
Altogether, we see that the first condition is redundant, and the others reduce to
\begin{align*}
\tilde{\epsilon}\lqq \sqrt{\frac{\frac{1}{2}+\vartheta}{\max\left\{(1+m),\ln\big(\frac{8m^{3/2}}{\sqrt{2\pi}\e}\big)\right\}+\ln(2)}}.
\end{align*}
The above equation remains satisfied if we require $\tilde{\epsilon}$ to be smaller or equal to $$\sqrt{\frac{\frac{1}{2}+\vartheta}{\max\bigg\{\big(1+1+\tfrac{8r_0}{\e^2}\big),\ln\bigg(\frac{8\big(\tfrac{8r_0}{\e^2}+1\big)^{3/2}}{\sqrt{2\pi}\e}\bigg)\bigg\}+\ln(2)}}$$ (we just omitted the $\lfloor\cdot\rfloor$ in the expression for $m$). Moreover,
$$1+1+\tfrac{8r_0}{\e^2}> \ln\bigg(\frac{8\big(\tfrac{8r_0}{\e^2}+1\big)^{3/2}}{\sqrt{2\pi}\e}\bigg),\quad\mbox{for all }\e>0,$$
which leaves us with the rather simple condition $$\tilde{\epsilon}\lqq \sqrt{\frac{\frac{1}{2}+\vartheta}{2+\frac{8r_0}{\e^2}+\ln(2)}}=\sqrt{\frac{\frac{1}{2}+\vartheta}{2+\frac{8(\frac{1}{2}+\eta)}{\e^2}+\ln(2)}}.$$
Now,
\begin{align*}
(2\ln(\ln(q^n)))^{-\frac{1}{2}}=\tilde{\epsilon}\lqq\sqrt{\frac{\frac{1}{2}+\vartheta}{2+\frac{8(\frac{1}{2}+\eta)}{\e^2}+\ln(2)}}
\end{align*}
is equivalent to
$n\geq \frac{\exp\Big(\frac{2+\frac{4+8\eta}{\e^2}+\ln(2)}{1+2\vartheta}\Big)}{\ln(q)}=:n_0.$
Therefore, we can extend \eqref{e:Schilder} to 
\begin{align*}
\PP(d(Z_{q^n}(\cdot, \cdot),\kK(\tfrac{1}{2}+\eta))>\e)
&\lqq b\frac{1}{\ln(q)^{1+2\vartheta}} \frac{1}{n^{1+2\vartheta}}, \qquad n\gqq 1,
\end{align*}
where $b = b(q, \eta, \vartheta, \e)=\ln(q)^{1+2\vartheta}\cdot n_0^{1+2\vartheta}$. Inserting the expression for $n_0$, we get that $$b=\left(\frac{\exp\Big(\frac{2+\frac{4+8\eta}{\e^2}+\ln(2)}{1+2\vartheta}\Big)}{\ln(q)}\right)^{1+2\vartheta}\ln(q)^{1+2\vartheta}=2e^2e^{\frac{4+8\eta}{\e^2}},$$
such that 
\begin{equation}\label{e:kompletteRate}
\PP(d(Z_{q^n}(\cdot, \cdot),\kK(\tfrac{1}{2}+\eta))>\e)
\lqq 2e^2e^{\frac{4+8\eta}{\e^2}}\frac{1}{\ln(q)^{1+2\vartheta}} \frac{1}{n^{1+2\vartheta}}, \qquad n\gqq 1.
\end{equation}
\noindent Therefore, by Example~\ref{ex:Riemann}, we have for all $\tfrac{1}{2}< \vartheta < \eta$ that 
for 
\begin{align*}
&\oO_{\e,q,\eta} = \sum_{n=1}^\infty \ind\{d(Z_{q^n}(\cdot, \cdot),\kK(\tfrac{1}{2}+\eta))>\e\},\\
&\m_{\e,q,\eta}(\omega) = \sup\{n-1~|~n\in \NN, d(Z_{q^n}(\cdot, \cdot),\kK(\tfrac{1}{2}+\eta))>\e\},
\end{align*}
and all $p>-1$ which satisfy $1+p< 2\vartheta$, that 
\begin{align*}
\EE[\oO_{\e,q,\eta}^{1+p}]\lqq \EE[\m_{\e,q,\eta}^{1+p}]
&\lqq  2e^2e^{\frac{4+8\eta}{\e^2}} \frac{(p+1)(1+2\vartheta)}{2\vartheta}\frac{\zeta(2\vartheta-p)}{\ln(q)^{1+2\vartheta}}+\ind(\{p<0\})\\
&\lqq  2e^2e^{\frac{4+8\eta}{\e^2}} (1+2\vartheta)\frac{\zeta(2\vartheta-p)}{\ln(q)^{1+2\vartheta}}+\ind(\{p<0\}).
\end{align*}
Markov's inequality then yields 
\begin{align*}
 &\PP(\#\{n\in \NN~\vert~d(Z_{q^n}(\cdot, \cdot),\kK(\tfrac{1}{2}+\eta))>\eps\}\gqq k) \lqq \PP(\m_{\e,q,\eta}\gqq k) \\
 &\lqq 
 2e^2e^{\frac{4+8\eta}{\e^2}} \frac{(p+1)(1+2\vartheta)}{2\vartheta}\frac{\zeta(2\vartheta-p)}{\ln(q)^{1+2\vartheta}}\cdot k^{-(1+p)}+k^{-(1+p)}\cdot\ind(\{p<0\})\\
 & \lqq 2e^2e^{\frac{4+8\eta}{\e^2}} (1+2\vartheta)\frac{\zeta(2\vartheta-p)}{\ln(q)^{1+2\vartheta}}\cdot k^{-(1+p)}+k^{-(1+p)}\cdot\ind(\{p<0\})
\end{align*}
which we asserted in \eqref{e:Strassen}. This shows item (3)(a).

\noindent We show item (3)(c). If we equal the right-hand side of \eqref{e:kompletteRate} to $\tfrac{1}{n \ln(n+1)^{1+\theta}}$ for some $\theta>0$ and solve for $\e$ we obtain 
\begin{align*}
\e = \e_n = \sqrt{\frac{4+8\eta}{\ln(\tfrac{\ln(q)^{1+2\vartheta}}{2e^2}) + 
\ln(\tfrac{n^{2\vartheta}}{\ln(n+1)^{1+\theta}})}}.
\end{align*}
The summability of $(n \ln(1+n)^{1+\theta})^{-1}$ and the usual first Borel-Cantelli lemma finishes the proof of item (3)(c).

\noindent Finally we show (3)(b): By \eqref{e:Strassen}, we also have 
\begin{align*}
 &\PP(\#\{n\in \NN~\vert~d(Z_{q^n}(\cdot, \cdot),\kK(\tfrac{1}{2}+\eta))>\eps\}\gqq k)\nonumber\\ 
 &\qquad \lqq \inf_{\substack{\tfrac{1}{2}<\vartheta<\eta\\-1<p<2\vartheta-1}}
 \!\! 2e^2e^{\frac{4+8\eta}{\e^2}}\!\cdot\!\!\!\frac{(p+1)(1+2\vartheta)}{2\vartheta}\frac{\zeta(2\vartheta-p)}{\ln(q)^{1+2\vartheta}}\cdot k^{-(1+p)}+k^{-(1+p)\cdot\ind(\{p<0\})}.
\end{align*} 
For the approximate minimization in $p$, we refer to \cite[Example 1 and Appendix B]{EHS26} in the case of $n_0=1$ to find the maximizer $p^*=2\vartheta-1-\frac{1}{\ln(k)+\gamma}$. Inserting into \eqref{e:Strassen} and letting $\vartheta\to\eta$ yields the desired bounds of (3)(b).  The last statement about the asymptotics follows since $\zeta(s)$ behaves as $\frac{1}{s-1}+\gamma$ for $s$ close to 1 and $$\lim_{k\to\infty}k^{\frac{1}{\ln(k)+\gamma}}=\lim_{k\to\infty}e^{1-\frac{\gamma}{\ln(k)+\gamma}}=e.$$
This finishes the proof of Theorem~\ref{thm:Strassen}.
\end{proof}\bigskip

\section{\textbf{Proof of Lemma~\ref{lem:fastsicherextrem}} }\label{a:fastsicherextrem}

\begin{proof}
Recall, that the Börjesson-Sundberg type estimate \eqref{eq:BSW-bound} implies for all $\al>0$ and $n\gqq 2$ we have 
\[
\PP(|Z_n|\gqq \sqrt{2(1+\al) \ln(n)})\lqq  \exp(-(1+\al) \ln(n))= \frac{1}{n^{1+\al}}.
\]
The classical Borel-Cantelli Lemma yields 
that there exists a random variable $\ti N\in \NN$ such that $\PP$-a.s. for all $n\gqq \ti N$ we have 
\[
|Z_n| \lqq  \sqrt{2(1+\al)\ln(n)}.
\]
If we define 
\[
\Gamma_{\alpha, N, M} := \sup_{N+1 \lqq n\lqq M} \frac{|Z_n|}{\sqrt{2(1+\al)\ln(n)}}, 
\qquad \mbox{ and }\qquad \Gamma_{\alpha, N} := \lim_{M\ra\infty} \Gamma_{\alpha, N, M}, 
\] 
then for all $n\gqq N+1$ we have 
\begin{align*}
|Z_n| \lqq \max\{\Gamma_{\alpha, N}, ~1\}\cdot \sqrt{2(1+\al)\ln(n)}.
\end{align*}
This shows statement \eqref{e:direktfs}. 
We continue with the tail probability for some $t>1$, using again 
$$\PP(Z_n>t)<\frac{e^{-\frac{t^2}{2}}}{\sqrt{2\pi}\cdot t}\quad\Longrightarrow\quad \PP(|Z_n|>t)<\frac{\sqrt{2}e^{-\frac{t^2}{2}}}{\sqrt{\pi}\cdot t},$$
\begin{align*}
\PP(\Gamma_{\al, N, M} \lqq t) 
&= \prod_{n=N+1}^M \PP\left(\frac{|Z_n|}{\sqrt{2(1+\al)\ln(n)}}\lqq t\right)
= \prod_{n=N+1}^M  \Big(1-\PP\Big(|Z_n|> t\sqrt{2(1+\al)\ln(n)} \Big)\Big)\\
&\gqq \prod_{n=N+1}^M \bigg(1- \frac{\sqrt{2}\cdot e^{-\left(t\sqrt{2(1+\al)\ln(n)}\right)^2/2}}{\sqrt{\pi}\cdot t\sqrt{2(1+\al)\ln(n)}} \bigg)\\
&= \prod_{n=N+1}^M \exp\Big(\ln\Big(1-\frac{1}{\sqrt{\pi(1+\al)}} \frac{1}{t n^{(1+\al)t^2}\sqrt{\ln(n)}}\Big)\Big).
\end{align*}
Since $\ln(1-x)\gqq -\sqrt{\pi}\cdot x$ for $x\in [0,0.7]$ and since $\frac{1}{\sqrt{\pi}\sqrt{\ln(2)}}\approx 0.6777<0.7$, we get
\begin{align*}
&\prod_{n=N+1}^M \exp\Big(\ln\Big(1-\frac{1}{\sqrt{\pi(1+\al)}} \frac{1}{t n^{(1+\al)t^2}\sqrt{\ln(n)}}\Big)\Big)\gqq \prod_{n=N+1}^M \exp\Big(-\frac{\sqrt{\pi}}{\sqrt{\pi(1+\al)}} \frac{1}{t n^{(1+\al)t^2}\sqrt{\ln(n)}} \Big)
\\
&=\exp\Big(- \frac{1}{\sqrt{1+\al}}\sum_{n=N+1}^M \frac{1}{t n^{(1+\al)t^2}\sqrt{\ln(n)}} \Big)= \exp\Big(-\frac{1}{\sqrt{1+\al}} \frac{1}{t}\sum_{n=N+1}^M\frac{1}{n^{(1+\al)t^2}\sqrt{\ln(n)}} \Big).
\end{align*}
Thus, sending $M\ra\infty$ we have 
\begin{align*}
\PP(\Gamma_{\alpha, N} > t) 
&\lqq 1-\exp\Big(-\frac{1}{\sqrt{1+\al}}\frac{1}{t} \sum_{n=N+1}^\infty \frac{1}{n^{(1+\al)t^2}\sqrt{\ln(n)}} \Big).
\end{align*}
By integral comparison we obtain 
\begin{align}
\sum_{n=N+1}^\infty \frac{1}{n^{(1+\al)t^2} \sqrt{\ln(n)}} 
\lqq \int_N^\infty \frac{1}{\sqrt{\ln(x)}\,x^{(1+\al)t^2}} dx
= \int_N^\infty \ln(x)^{\frac{1}{2}-1}e^{-\ln(x) (1+\al)t^2} dx.\label{e:schwanz}
\end{align}
Substituting $y = \ln(x) (1+\al) t^2$, we have $x = \exp\big(\frac{y}{(1+\al) t^2}\big)$,  
$dx = \frac{\exp\big(\frac{y}{(1+\al) t^2}\big)dy}{(1+\al) t^2}$ and $x = N$ 
implying $y = \ln(N) (1+\al) t^2$ such that by the integral criterion and Corollary~\ref{cor:Gamma} for $a = \frac{1}{2}$ we have 
\begin{align*}
&\sum_{n=N+1}^\infty \frac{1}{n^{(1+\al)t^2} \sqrt{\ln(n)}}\\ 
&\lqq ((1+\al)t^2)^{-\frac{1}{2}}\int_{\ln(N)(1+\al)t^2}^\infty y^{\frac{1}{2}-1}e^{-y(1-\frac{1}{(1+\al)t^2})} dy\\
&=((1+\al)t^2)^{-\frac{1}{2}}\Big(1- \frac{1}{(1+\al)t^2}\Big)^{-{1}/2}\int_{\ln(N)((1+\al)t^2-1)}^\infty z^{\frac{1}{2}-1}e^{-z} dz\\
&\lqq \frac{1}{((1+\al)t^2-1)^{\frac{1}{2}}} 
\Big(1+\frac{1}{2(\ln(N)((1+\al)t^2-1))}\Big)\Big(\ln(N)((1+\al)t^2-1)\Big)^{-\frac{1}{2}}
e^{-\ln(N)((1+\al)t^2-1)}\\
&= \frac{1}{{((1+\al)t^2-1)} }\Big(1+\frac{1}{2(\ln(N)((1+\al)t^2-1))}\Big)  
\frac{1}{N^{((1+\al)t^2-1)}\sqrt{\ln(N)}}.
\end{align*}
Since $1-e^{-x}\lqq x$, from \eqref{e:schwanz} we get for $t\gqq 1+\e$ 
\begin{align}\label{eq:gamma-prob-ineq}
&\PP(\Gamma_{\alpha,N} > t) 
\lqq  \frac{1}{{(1+\al)^{1/2}t((1+\al)t^2-1)} }\Big(1+\frac{1}{2\ln(N)((1+\al)t^2-1)}\Big)\frac{1}{N^{((1+\al)t^2-1)}\sqrt{\ln(N)}}\nonumber\\
&\lqq \frac{1}{{(1+\al)^{1/2}t(\al+(1+\al)(2\e+\e^2))} }\Big(1+\frac{1}{2\ln(N)(\al+(1+\al)(2\e+\e^2))}\Big)\frac{1}{N^{((1+\al)t^2-1)}\sqrt{\ln(N)}}\nonumber\\
&\lqq \frac{1}{{(1+\al)^{1/2}t\al} }\Big(1+\frac{1}{2\al\ln(N)}\Big)\frac{1}{N^{((1+\al)t^2-1)}\sqrt{\ln(N)}}.
\end{align}
Calculating the Gaussian moments 
we obtain for all $0< q< (1+\al)\ln(N)$ the moment estimate 
\begin{align}
\EE[e^{q \max\{\Gamma_{\alpha,N}^2, 1\}}] 
&= \EE[e^{q \max\{\Gamma_{\alpha,N}^2, 1\}}\ind\{\Gamma_{\alpha,N}\lqq 1+\e\}]
+\EE[e^{q \max\{\Gamma_{\alpha,N}^2, 1\}}\ind\{\Gamma_{\alpha,N}> 1+\e\}]\nonumber\\
&\lqq e^{q(1+\e)^2} + \int_{1+\e}^\infty q \,2 t\,e^{qt^2}\PP(\Gamma_{\alpha,N}>t) dt.\label{e:Gaussmoment}
\end{align}
We continue with the second term on the right-hand side 
\begin{align}\label{eq:gamma-int}
 \int_{1+\e}^\infty q \,2 t\,e^{qt^2}\PP(\Gamma_{\alpha,N}>t) dt\lqq  \frac{2q}{(1+\al)^{1/2}\al }\Big(1+\frac{1}{2\al\ln(N)}\Big)\frac{N}{\sqrt{\ln(N)}}\int_{1}^\infty e^{(q-\ln(N)(1+\al))t^2}dt.
\end{align}
Standardizing the integral yields
\begin{align*}
&\int_{1+\e}^\infty q \,2 t\,e^{qt^2}\PP(\Gamma_{\alpha,N}>t) dt\\
&\lqq \frac{2q}{(1+\al)^{1/2}\al }\Big(1+\frac{1}{2\al\ln(N)}\Big)\frac{N}{\sqrt{\ln(N)}}\sqrt{\frac{\pi}{(1+\alpha)\ln(N)-q}}\big(1-\Phi\big(\sqrt{2((1+\alpha)\ln(N)-q)}\big)\big).
\end{align*}
This finishes the proof. 
\end{proof}

\bigskip 

\section{\textbf{Optimal rates}}\label{a:optimal}

\begin{lem}\label{lem:optimal}
For any $M\gqq 1$, $b\in (0,1)$ and $k\in \NN$ we have 
\begin{equation}\label{e:argmin}
p_k := \stackrel{\mbox{\textnormal{argmin}}}{{}_{p\in [0, -\ln(b))}} e^{-kp}\Big(\frac{M}{1-e^{p}b} +1\Big)
=\ln\Big(\frac{2k (M+1)}{b (2k + M (k+1)+ \sqrt{(2k + M (k+1))^2 - 4k^2 (M+1)})}\Big)
\end{equation}
and for all $k\gqq 1$ we have 
\begin{equation}\label{e:OptimaleRate}
e^{-kp_k}\Big(\frac{M}{1-e^{p_k}b} +1\Big) 
\lqq  e^{\frac{5}{4}}  \cdot(2k+1)(M+1) \cdot b^{k}.
\end{equation}
\end{lem}
\begin{proof} 
For $f(p) = \frac{e^{-kp}}{1-e^pb}$ the conditions $e^pb < 1$ and $k\gqq 1$ imply that 
\begin{align*}
f''(p) 
&= \frac{2 e^{-kp}e^{2p}b^2}{(1-e^{p}b)^3} + \frac{e^{-kp} e^{p}b}{(1-e^{p}b)^2}+\frac{k^2e^{-kp}}{1-e^{p}b}+\frac{2kbe^{p}e^{-kp}}{(1-e^{p}b)^2}>0.
\end{align*}
Now, the sum of convex smooth functions is convex, and hence $g(p) := e^{-kp}\Big(\frac{M}{1-e^pb} +1\Big) = \frac{e^{-kp} (M +1- e^{p}b)}{1-e^pb}$ is a convex function, 
such that 
\begin{align*}
0 = \frac{dg}{dp}(p) 
&= \frac{(-ke^{-kp} (M +1- e^{p}b) - e^{-kp} e^{p} b) (1 - b e^p)+e^{-kp} (M +1- e^{p}b) b e^p}{(1 - b e^p)^2}\\
&= \frac{e^{-kp}}{(1 - b e^p)^2}((-k (M +1- e^{p}b) -  e^{p} b) (1 - b e^p)+ (M +1- e^{p}b) b e^p)
\end{align*} 
for $x = e^p \in (0, \frac{1}{b})$ reads 
\begin{align*}
0 &= (-k (M +1- xb) -  x b) (1 - b x)+ (M +1- xb) b x \\
&= -k (M +1) + xb(k-1)-bx (-k (M +1) + xb(k-1)) 
+ (M +1)xb- (xb)^2\\
&= -x^2 b^2k + x b (2k + M (k+1)) -k (M +1).
\end{align*}
Hence, a solution $x_k$ is given by
\begin{align*}
x_k
&= \frac{-b (2k + M (k+1))+ \sqrt{b^2 (2k + M (k+1))^2 - 4b^2k^2 (M+1)}}{-2 b^2k}\\
&= \frac{(2k + M (k+1))- \sqrt{(2k + M (k+1))^2 - 4k^2 (M+1)}}{2 bk}\\
&= \frac{(2k + M (k+1))^2- (2k + M (k+1))^2 + 4k^2 (M+1)}{2 bk ((2k + M (k+1))+ \sqrt{(2k + M (k+1))^2 - 4k^2 (M+1)})}\\
&= \frac{2k (M+1)}{b ((2k + M (k+1))+ \sqrt{(2k + M (k+1))^2 - 4k^2 (M+1)})}.
\end{align*}
And for at least one solution of the original equation, we get 
$$p_k = \ln(x_k) = 
\ln\Big(\frac{2k (M+1)}{b ((2k + M (k+1))+ \sqrt{(2k + M (k+1))^2 - 4k^2 (M+1)})}\Big)
$$ 
which implies \eqref{e:argmin}. Inserting $p_k$ we calculate 
\begin{align*}
(1- e^{p_k} b)^{-1}
&= \Big(1-\frac{2k (M+1)}{ 2k + M (k+1)+ \sqrt{(2k + M (k+1))^2 - 4k^2 (M+1)}}\Big)^{-1}\\
&= \Big(1-\frac{2k (M+1)}{ 2k + M (k+1)+ \sqrt{ M^2 k^2 + M^2 2k+ M^2 + 4k M   }}\Big)^{-1}\\
&= \Big(1-\frac{2k + 2 Mk }{M+ 2k + M k+Mk \sqrt{ 1 + \frac{2}{k}+ \frac{1}{k^2} + \frac{4}{k M}   }}\Big)^{-1}\\
&= \Bigg(1-\frac{1 }{\frac{M}{2k (M+1)}+ \frac{2k + M k+Mk \sqrt{ 1 + \frac{2}{k}+ \frac{1}{k^2} + \frac{4}{k M}   }}{2k + 2 Mk}}\Bigg)^{-1}\\
&\lqq \bigg(1-\frac{1 }{\frac{M}{2k(M+1)}+ \frac{2k + M k+Mk }{2k + 2 Mk}}\bigg)^{-1}= \Big(1-\frac{1 }{\frac{M}{2k(M+1)}+ 1}\Big)^{-1}
= \frac{2k (M+1)}{M}+1.
\end{align*}
Therefore, 
\begin{equation}\label{e:exp-estim-1}
    \Big(\frac{M}{1-e^{-k\tilde{p}_k}b}+1\Big)\lqq (2k+1)(M+1).
\end{equation}
Finally, we estimate 
\begin{align*}
e^{-p_k k} 
&= b^k \bigg(\frac{2k + M (k+1)+ \sqrt{(2k + M (k+1))^2 - 4k^2 (M+1)}}{2k (M+1)}\bigg)^{k}\\
&= b^k \left(\frac{M}{2k (M+1)}+ \frac{2k + M k+Mk \sqrt{ 1 + \frac{2}{k}+ \frac{1}{k^2} + \frac{4}{k M}   }}{2k + 2 Mk}\right)^{k}\\
&= b^k \bigg(\frac{2k + M k+Mk \sqrt{ 1 + \frac{2}{k}+ \frac{1}{k^2} + \frac{4}{k M}   }}{2k + 2 Mk}+\frac{M}{2 (M+1)}\frac{1}{k}\bigg)^{k}\\
&= b^k \left(1 +\frac{M\Big(\sqrt{ 1 + \frac{2}{k}+ \frac{1}{k^2} + \frac{4}{k M}   }-1\Big)}{2 (M+1)}+\frac{M}{2 (M+1)}\frac{1}{k}\right)^{k}\\
&= b^k \left(1 +\frac{ M\big(\frac{2}{k}+ \frac{1}{k^2} + \frac{4}{k M}\big)}{2 (M+1)\Big(\sqrt{ 1 + \frac{2}{k}+ \frac{1}{k^2} + \frac{4}{k M}   }+1\Big)}+\frac{M}{2 (M+1)}\frac{1}{k}\right)^{k}\\
&\lqq b^k \Big(1 +\frac{ M\big(\frac{2}{k}+ \frac{1}{k} + \frac{4}{k M}\big)}{4 (M+1)}+\frac{M}{2 (M+1)}\frac{1}{k}\Big)^{k} = b^k \bigg(1 +\Big(\frac{ 3M + 4}{4 (M+1)}+\frac{M}{2 (M+1)}\Big)\frac{1}{k}\bigg)^{k}\\
&
\lqq b^k \exp\Big(\frac{ 5M+4}{4 M +4 }\Big).  
\end{align*}
Combining the preceding inequality with \eqref{e:exp-estim-1}, we have for all $k\gqq 1$ 
\begin{align*}
e^{-k\tilde{p}_k}\Big(\frac{M}{1-e^{p_k}b} +1\Big) 
&\lqq (2k+1) (M+1)b^k \exp\Big(\frac{ 5M + 4}{4 M +4}\Big)
\end{align*}
which implies \eqref{e:OptimaleRate}. 
Note that for $M\gqq 0$ we have 
\begin{equation}\label{e:ehochneunachtel}1 \lqq \frac{ 5M + 4}{4 M +4 }\lqq \frac{5}{4}\end{equation} such that 
$\exp\big(\frac{ 5M + 4}{4 M +4}\big)\lqq e^\frac{5}{4}\approx 3.4903$. 
This finished the proof. 
\end{proof}

\bigskip 
  
  \section{\textbf{The asymptotics of the upper incomplete Gamma function}}\label{a:Gamma}
  
\noindent According to \cite{DLMF}, \S 8.11(i), for the upper incomplete Gamma function
 \begin{align*}
 \Gamma(a, z) := \int_z^\infty t^{a-1} e^{-t} dt, \qquad a\in \RR, z>0, 
 \end{align*}
we have the following (non-asymptotic) estimate: 
\begin{align*}
 \Gamma(a, z)
 &= z^{a-1} e^{-z} \cdot \Big(1 +\sum_{k=1}^{n-1} \frac{u_k}{z^k} + R_n(a,z)\Big), \qquad n=1,2\dots, a\in \RR, z>0,  
\end{align*}
where $u_k :=(a-1) (a-2) \hdots (a-k)$ and, if $n\gqq a-1$,
\begin{align*}
|R_n(a,z)| \lqq \frac{|u_n|}{z^n}. 
\end{align*}
For more details, see for instance, \cite{Ol97}, pp. 109--112. 

\begin{cor}\label{cor:Gamma} 
Taking $n=1$ in the above result and any $a\in (-\infty,2]$, we have for all $z>0$ 
\begin{align*}
 \Gamma(a, z) 
 &\lqq  \Big(1+ \frac{|a-1|}{z}\Big)\cdot z^{a-1} e^{-z}. 
\end{align*}
 \end{cor}
 
\medskip

\section*{\textbf{Acknowledgments}} 

\noindent MAH acknowledges support by project INV-2019-84-1837 of Facultad de Ciencias at Universidad de los Andes and the kind hospitality by Prof.~Dr.~E.~Hausenblas during a research stay June-July 2022 at the Chair of Applied Mathematics, at the Technical University of Leoben, Austria, where this project started. The latter was financed by the JESH 2019 [Joint Excellence in Science and Humanities] project ``Metastability in Turing patterns'' of the Austrian Academy of Sciences.

\end{document}